\documentclass[a4paper,reqno]{amsart}
\usepackage[utf8]{inputenc}
\usepackage{amssymb}
\usepackage{enumitem}
\usepackage{mathrsfs}
\usepackage{mathtools}
\usepackage{amsmath}
\usepackage [dvipsnames] { xcolor }
\usepackage{multirow}
\usepackage[colorlinks=true]{hyperref}

\hypersetup{linkcolor=RoyalBlue,urlcolor=SeaGreen,citecolor=SeaGreen}

\setlist[enumerate]{label=\emph{(\roman*)}}

\usepackage{environ}

\newtheorem{theorem}{Theorem}[section]
\newtheorem{corollary}[theorem]{Corollary}
\newtheorem{lemma}[theorem]{Lemma}
\newtheorem{proposition}[theorem]{Proposition}

\theoremstyle{definition}
\newtheorem{definition}[theorem]{Definition}
\newtheorem{remark}[theorem]{Remark}
\numberwithin{equation}{section}
\newcommand{\RR}{\mathbb{R}}

\newcommand{\dd}{{\rm d}}

\begin{document}

\title[Three-solitons for gKdV equation]{Construction of three-solitons with logarithmic distance for the mass-critical gKdV equations}

\author{Yang Lan}
\address{Yau Mathematical Sciences Center, Tsinghua University, Beijing 100084, P. R. China.}
\email{lanyang@mail.tsinghua.edu.cn}

\author{Xu Yuan}
\address{State Key Laboratory of Mathematical Sciences, Academy of Mathematics and Systems Science, Chinese Academy of Sciences, Beijing 100190, P. R. China.}
\email{xu.yuan@amss.ac.cn}


\thanks{Y. L. was partially funded by the New Cornerstone Investigator Program 100001127.}

\begin{abstract}
     For the mass-critical generalized Korteweg-de Vries equation,
\begin{equation*}
    \partial_{t}u+\partial_{x}\left(
    \partial_{x}^{2}u+u^{5}\right)=0,\quad (t,x)\in [0,\infty)\times \RR,
\end{equation*}
we prove the existence of three-soliton solutions with logarithmic relative distance and with the choice of signs $(+,-,-)$. The choice of the number three and the signs of solitons are related to the solvability of the ODE system generated by the nonlinear interactions between the three solitons and some non-localized profiles. In particular, these special behaviors are due to strong interactions between the three solitons. That is, the dynamics of each soliton is perturbed at leading order by the presence of other solitons.
\end{abstract}

\maketitle

\section{Introduction}
\subsection{Main result}\label{SS:Main}
In this article, we consider the dynamics of multi-solitons for the mass-critical generalized Korteweg-de Vries (gKdV) equation,
\begin{equation}\label{equ:gKdV}
    \partial_{t}u+\partial_{x}\left(
    \partial_{x}^{2}u+u^{5}\right)=0,\quad (t,x)\in [0,\infty)\times \RR.
\end{equation}
Recall that, the Cauchy problem for equation~\eqref{equ:gKdV} is locally well-posed in the energy space $H^{1}$ (see Kenig-Ponce-Vega~\cite{KPVgKdV1}). For any $u_{0}\in H^{1}$, there exists a unique (in a certain class) maximal solution $u$ of~\eqref{equ:gKdV} in $C\left([0,T);H^{1}\right)$ satisfying $u_{|t=0}=u_{0}$. Moreover, the following blow-up criterion holds:
	\begin{equation}\label{equ:blowCauchy}
		T<\infty\Longrightarrow 
		\lim_{t\uparrow T}\| \partial_{x} u(t)\|_{L^{2}}=\infty.
	\end{equation}

For any $H^{1}$ solution $u$, the mass $M$ and energy $E$ are conserved, where
\begin{equation*}
    M(u)=\int_{\RR}u^{2}\dd x\quad \mbox{and}\quad 
    E(u)=\frac{1}{2}\int_{\RR}(\partial_{x}u)^{2}\dd x-\frac{1}{6}\int_{\RR}u^{6}\dd x.
\end{equation*}

Recall also that, for any solution $u$ of~\eqref{equ:gKdV} and $\lambda>0$, the scaling symmetry
\begin{equation*}
    u_{\lambda}(t,x)=\lambda^{\frac{1}{2}}u(\lambda^{3}t,\lambda x),\quad \mbox{for}\ (t,x)\in [0,\infty)\times \RR,
\end{equation*}
again results in a solution to~\eqref{equ:gKdV}. This scaling symmetry keeps the $L^{2}$-norm invariant so that the problem is called \emph{mass-critical}.

\smallskip
Denote 
\begin{equation}\label{equ:defQ}
    Q(x)=\left(\frac{3}{\cosh^{2}2x}\right)^{\frac{1}{4}},\ \  -Q''+Q-Q^{5}=0\ \  \mbox{and}\quad E(Q)=0.
\end{equation}
The family of traveling wave solutions (also called \emph{soliton})
\begin{equation*}
    u(t,x)=\lambda_{0}^{-\frac{1}{2}}
    Q\left(\lambda_{0}^{-1}\left(x-\lambda_{0}^{-2}t-x_{0}\right)\right),\ \ \mbox{for}\ (\lambda_{0},x_{0})\in (0,\infty)\times \RR,
\end{equation*}
plays a distinguished role in the analysis of dynamics for~\eqref{equ:gKdV}. More precisely, the soliton $Q$ is related to the sharp Gagliardo-Nirenberg inequality (see~\cite{Wein}):
\begin{equation*}
\frac{1}{3}\int_{\RR}\phi^{6}\dd x
	\le \left(\frac{\int_{\RR}\phi^{2}\dd x}{\int_{\RR}Q^{2}\dd x}\right)^{2}
    \int_{\RR}(\partial_{x}\phi)^{2}\dd x,\quad \mbox{for any}\ \phi\in H^{1}.
\end{equation*}
Therefore, from a variational argument, the conservation of energy and the blow-up criterion~\eqref{equ:blowCauchy}, we know that any initial data $u_{0}\in H^{1}$ with subcritical mass, \emph{i.e.,} satisfying $\|u_{0}\|_{L^{2}}<\|Q\|_{L^{2}}$, generates a \emph{global and bounded} solution in $H^{1}$.

\smallskip
For the case of $\|u_{0}\|_{L^{2}}\ge \|Q\|_{L^{2}}$, it is natural to first restrict ourselves to considering the solutions with small supercritical mass, \emph{i.e.,} satisfying
	\begin{equation}\label{equ:massabove}
		\|Q\|_{L^{2}}\le \|u_{0}\|_{L^{2}}<(1+\delta)\|Q\|_{L^{2}},\quad
		\mbox{for}\ \ 0<\delta\ll 1.
	\end{equation}
The study of \emph{singularity formation} (in finite or infinite time) in such case has been an interesting problem and significant advances have been achieved in the last twenty years. In particular, a full description
and classification of the asymptotic behavior of the solution close to a single soliton has been obtained by Martel-Merle-Rapha\"el~\cite{MMRACTA}. See Subsection~\ref{SS:PreviousGKdV} for more discussion.

\smallskip
Another natural problem related to the \emph{soliton dynamics} for~\eqref{equ:gKdV} is to study multi-soliton solutions. In this direction, the first result was obtained by Martel~\cite{MartelAJM} which established the existence and uniqueness of \emph{pure multi-soliton} solutions of~\eqref{equ:gKdV} in the weak interaction regime. Then, C\^ote~\cite{CoteDUKE} constructed solutions that behave as a sum of a \emph{linear dispersive term} and \emph{multi-solitons} of~\eqref{equ:gKdV}. Indeed, the solutions  constructed in~\cite{CoteDUKE} still belong to the weak interaction regime. That is, the dynamics of each soliton is not perturbed at leading order by the presence of the linear term and other solitons. In this article, we revisit \emph{multi-soliton dynamics} and construct the three-soliton solutions in the strong interaction regime for~\eqref{equ:gKdV}.

\smallskip
Let $c_{Q}=12^{\frac{1}{4}}>0$. We set
\begin{equation*}
    \alpha=\frac{c_{Q}}{m_{0}^{2}}\int_{\RR}e^{y}Q^{5}(y)\dd y>0,\ \ \mbox{with}\ m_{0}=\frac{1}{4}\int_{\RR}Q(y)\dd y>0.
\end{equation*}
To simplify notation, we introduce
\begin{equation*}
    \beta_{1}=0,\ \  \beta_{2}=\log \left(\frac{\alpha }{15}\right)\ \ \mbox{and}\ \ 
    \beta_{3}=\beta_{2}+\log \left(\frac{\alpha }{3}\right).
\end{equation*}

\begin{theorem}\label{thm:main}
    Let $\vec
    \sigma=(\sigma_{1},\sigma_{2},\sigma_{3})=(1,-1,-1)$. Then there exists $T_0>0$ and a global-in-time solution $u\in C([T_{0},\infty);H^{1})$ of~\eqref{equ:gKdV} such that 
    \begin{equation*}
        \lim_{t\to \infty}\Big\|u(t)-\sum_{k=1}^{3}\sigma_{k}Q(\cdot-t-x_{k}(t))\Big\|_{H^{1}}=0.
    \end{equation*}
    Here, the position parameters $(x_{1},x_{2},x_{3})$ satisfy
    \begin{equation*}
     \sum_{k=1}^{3}  |x_{k}(t)+(18-3k) \log t-\beta_{k}|\lesssim t^{-\frac{1}{16}}, \ \ 
     \mbox{for all}\ t>T_{0}.
    \end{equation*}
\end{theorem}

\begin{remark}\label{re:othercases}
   The choice of three solitons, the choice of signs for the solitons and the asymptotic behavior of the position parameters are related to the solvability (and an explicit solution) of the ODE system generated by the nonlinear interactions between solitons and some non-localized profiles. We mention here that, for any other case except the symmetry one $\vec{\sigma}=(-1,1,1)$, it seems that the ODE system does not have a straightforward explicit solution  (see Subsection~\ref{SS:Sketch} for more discussion), and thus, we do not expect that Theorem~\ref{thm:main} can be extended to any other case of the mass-critical gKdV equation~\eqref{equ:gKdV} based on the current approach.
    \end{remark}
    
    \begin{remark}\label{re:BO}
       We expect that Theorem~\ref{thm:main} can be extended to some other dispersive models (for example, the modified Benjamin-Ono equation, which can be seen as a nonlocal version of the gKdV-type equation). However, for such an extension, some key additional ingredients and analysis related to the difference of the linear structure for the models will be needed.
\end{remark}

\subsection{Singularity formation for the gKdV equation}\label{SS:PreviousGKdV}
We briefly survey the literature related to the \emph{blow-up solutions} for the mass-critical gKdV equation~\eqref{equ:gKdV}. First, using the rigidity property of gKdV flow near the soliton established in Martel-Merle~\cite{MMJMPA}, the first proof of the existence of blow-up solutions (in finite or infinite time) was provided in Merle~\cite{MJAMS} for any $H^{1}$ initial data with negative energy and satisfying the small supercritical mass condition~\eqref{equ:massabove}. Then, a finite time blow-up solution was constructed in Martel-Merle~\cite{MMJAMS} for $H^{1}$ initial data with negative energy and spatial decay. Later on, in the work of Martel-Merle-Rapha\"el~\cite{MMRACTA}, the authors gave a complete description and classification of the solution flow near the soliton which completed the previous results in~\cite{MMJMPA,MMJAMS,MJAMS}. Last, we refer to~\cite{MMANN,MMRJEMS,MDASNSP,MDPRINT} for some related results on the blow-up dynamics near the single soliton for the mass-critical gKdV equation~\eqref{equ:gKdV}.

\smallskip
Concerning the \emph{multi-bubble dynamics}, it was proved in Combet-Martel~\cite{ComKdVSIAM} that there exists a finite time blow-up solution of~\eqref{equ:gKdV} containing an arbitrary number of blow-up bubbles for any choice of signs and coefficients of the scaling parameters. More precisely, the leading term of the approximate solution for the multi-bubble problem in~\cite{ComKdVSIAM} is a sum of some rescaled and modulated versions of the \emph{minimal mass blow-up solution} $S_{\rm{KdV}}$. Indeed, the existence and uniqueness of the minimal mass element $S_{\rm{KdV}}$ have been obtained by Martel-Merle-Rapha\"el~\cite{MMRJEMS} and the sharp asymptotics of such element have been obtained by Combet-Martel~\cite{Comkdv}. Therefore, the leading order of the nonlinear interaction for the approximate solution in~\cite{ComKdVSIAM} can be described which allowed the authors to refine the ansatz sufficiently sharp to close the energy estimate for this problem. To the best of our knowledge, the solution constructed in~\cite{ComKdVSIAM} is the first example of the solution containing \emph{multi-bubbles} in the strong interaction regime of~\eqref{equ:gKdV}. Very recently, an infinite time blow-up solution of~\eqref{equ:gKdV} containing two blow-up bubbles with opposite signs in the strong interaction regime has been constructed in~\cite{LYPRINT}.

\smallskip
\subsection{Multi-solitons for the gKdV equation}\label{SS:Multisoliton}
We briefly highlight a few results related to the \emph{multi-soliton dynamics} for the integrable KdV and mKdV equations.
Multi-soliton solutions have been obtained by the inverse scattering theory for the KdV equation in the work of Miura~\cite{Miura}. Then, Eckhaus-Schuur~\cite{Eckhaus} and Schuur~\cite{Schuur} showed that such behavior is fundamental for general solutions of the KdV equation. For the mKdV equation, besides the multi-soliton solutions, it admits a richer family of special solutions including breather solutions and dipole solutions (\emph{i.e.}, two-solitons with logarithmic distance). In particular, the existence of dipole solutions with \emph{opposite signs} was proved by Wadati-Ohkuma~\cite{Wadati} for the mKdV equation. See~\cite{Schuur} for more discussion on the long-time behavior of KdV and mKdV equations.

\smallskip
For the non-integrable case, the existence and uniqueness of multi-soliton solutions were proved in~\cite{MartelAJM} for subcritical and critical cases. A solution which behaves like a sum of linear dispersive term and multi-solitons was also constructed in~\cite{CoteJFA,CoteDUKE} for such cases.
For the supercritical case, the multi-solitons was constructed in C\^ote-Martel-Merle~\cite{CoteMartelMerle} for both gKdV and NLS equations. On the other hand, the existence of strongly interacting two-solitons (with logarithmic distance) was obtained in Nguyen~\cite{NguyenKdV} for the subcritical and supercritical cases via the general strategy introduced in Martel-Rapha\"el~\cite{MartelRaphael}. Later on, the description of asymptotic behaviors for the strongly interacting two-solitons was obtained in Jendrej~\cite{JacekKdV} for a class of gKdV equations. We mention here that, the results in the works of~\cite{JacekKdV,NguyenKdV} do not include the mass-critical gKdV equation due to the cancellation condition $(\Lambda Q,Q)=0$ in this case. More precisely, such condition leads to a more delicate ODE system which involves the interaction between solitons and some non-localized profiles (see Subsection~\ref{SS:Sketch} for more discussion). To the best of our knowledge, the solution in Theorem~\ref{thm:main} is the first example of a solution containing \emph{multi-solitons} in the strong interaction regime of the mass-critical gKdV equation~\eqref{equ:gKdV}.

\subsection{Related results for other models}\label{SS:Othermodels}
Historically, the existence of multi-soliton solutions for non-integrable cases was first studied by Merle~\cite{MerleCMP} for the mass-critical NLS equation and Martel~\cite{MartelAJM} for the subcritical and critical gKdV equation. Then, the strategy introduced in~\cite{MartelAJM,MerleCMP} was extended to other dispersive
or wave-type models (see for example~\cite{BGLC,CoteMartel,CoteMunoz,MMARMA,MRNSIAM,XYJFA} and references therein). We mention here that, all the above-mentioned works are related to the multi-solitons in the weak interaction regime. On the other hand, we mention quite a few other previous results of strongly interacting multi-solitons: for the Hartree equations by Krieger-Martel-Rapha\"el~\cite{KMRCPAM} and by G\'omez-Schmid-Wu~\cite{GSWARMA}, for the cubic Schr\"odinger system by \cite{MNDCDS}, for the NLS by Nguyen~\cite{NguyenNLS}, for the half-wave equation by G\'erard-Lenzmann-Pocovnicu-Rapha\"el~\cite{GLPRANNPDE}, for the NLKG equation by Aryan~\cite{Aryan}, for the damped NLKG equation by C\^ote-Martel-Yuan~\cite{CoteMartelYuan} and by Ishizuka-Nakanishi~\cite{Nakanishi},
for the fractional mKdV equation by Eychenne-Valet~\cite{EVJFA} and for the gBO equations by Lan-Wang~\cite{LanWangCVPDE}. Last, we refer to Jendrej-Kowalczyk-Lawrie~\cite{JacekKink1} and Jendrej-Lawrie~\cite{JacekKink2} for the study of strongly interacting multi-kinks (with logarithmic distance) for scalar fields in dimension 1+1, which is a phenomenon closely related to multi‑soliton dynamics.

\subsection{Sketch of the proof}\label{SS:Sketch}
Following the spirit of the remarkable work of Martel-Rapha\"el~\cite{MartelRaphael}, we construct the three-soliton solutions with logarithmic distance. We now give an outline of the proof and explain why we chose the three solitons and the specific signs in Theorem~\ref{thm:main}.
Consider the following change of variable
\begin{equation*}
    w(t,y)=u(t,y+t),\quad \mbox{for}\ (t,y)\in [0,\infty)\times \RR.
\end{equation*}
Here, $u(t,x)$ is a solution to~\eqref{equ:gKdV}. By an elementary computation, we find  
\begin{equation}\label{equ:gKdv2}
\partial_{t}w+\partial_{y}\left(\partial_{y}^{2}w-w+w^{5}\right)=0,\quad 
\mbox{for}\ (t,y)\in [0,\infty)\times \RR.
\end{equation}
From now on, we focus on the equation~\eqref{equ:gKdv2} and discuss suitable approximate multi-soliton solutions in the strong interaction regime for this equation.

\smallskip
Let $K\in \mathbb{N}^{+}$ with $K\ge 2$ and $\vec{\sigma}=(\sigma_{1},\sigma_{2},\dots,\sigma_{K})\in \left\{1\right\}\times\left\{-1,1\right\}^{K-1}$ be the signs of solitons. Let $\Gamma=(\Gamma_{1},\dots,\Gamma_{K})$ with $\Gamma_{k}=(x_{k},\nu_{k},b_{k})\in \RR^{3}$. Here, $x_{k}$ is the position parameter for the $k$-th soliton, $\nu_{k}$ is the scaling parameter for the $k$-th soliton and $b_{k}$ is a parameter related to a non-localized profile $X$ in the region of the $k$-th soliton. For any $k\in[\![1,K]\!]$, we denote by $Q_{k}$ the modulated soliton with position and scaling parameters $(x_{k},\nu_{k})$. Inspired by the previous work Martel-Pilod~\cite{MDASNSP}, we introduce the non-localized profile $X$ which only has a tail on the right-hand side of $\RR$ (see Corollary~\ref{coro:X}) to describe the projection of the nonlinear interaction between solitons onto the directions $Q_{k}$ for any $k\in[\![1,K]\!]$. Indeed, we will consider the backward-in-time evolution of~\eqref{equ:gKdv2} later and will choose a decreasing weight function to establish an energy functional with monotonicity property. Note that, a decreasing weight function will match the control for the effect of the tail on the right-hand side of $\RR$. This is the main reason why we choose the non-localized $X$ instead of the usual one $P$, which has a tail only on the left-hand side of $\RR$.

\smallskip
Because of the presence of the right-hand side tail of $X$, for any $(k,\ell)\in [\![1,K]\!]\times[\![1,K]\!]$ with $k<\ell$, the interaction between the modulated profile $X_{k}$ and modulated soliton $Q_{\ell}$ will appear in the ODE system for the parameters as the leading-order term. To refine such interactions, we need to introduce some additional terms in the approximate solution related to the localized profile $Y$ (see Lemma~\ref{le:Y}). In conclusion, we consider the following approximate multi-soliton solutions of~\eqref{equ:gKdv2}:
\begin{equation}\label{equ:defSintro}
\begin{aligned}
    S\approx\sum_{k=1}^{K}\sigma_{k}Q_{k}+\sum_{k=1}^{K}\sigma_{k}b_{k}X_{k}\phi-
    10m_{0}\sum_{k=1}^{K-1}\sum_{\ell=k+1}^{K}\sigma_{k}b_{k}Y_{\ell}.
    \end{aligned}
\end{equation}
From now on, we denote
\begin{equation*}
    r_{0}=r_{K}=0\ \  \mbox{and}\ \ 
    r_{k}=x_{k+1}-x_{k}\ \ \mbox{for}\ k\in  [\![1,K-1]\!].
\end{equation*}
Assume that the approximate solution $S$ is in the strong interaction regime. That is, for any $k\in [\![1,K]\!]$, the geometric parameters $\Gamma_{k}$ satisfy
\begin{equation}\label{equ:ODEsystem}
\left\{
\begin{aligned}
    \frac{\dd x_{k}}{\dd t}&=\nu_{k},\\
    \frac{\dd \nu_{k}}{\dd t}&=2b_{k},
    \end{aligned}
    \right.
    \ \  \mbox{and}\ \
    \frac{\dd b_{k}}{\dd t}=\alpha_{k,k-1}e^{-r_{k-1}}+\alpha_{k,k}e^{-r_{k}}.
\end{equation}
Here, the coefficients $(\alpha_{k,k-1},\alpha_{k,k})$ for $k\in [\![1,K]\!]$ are to be fixed later. To simplify notation, we denote by $\vec{\alpha}$ the collection of all coefficients $(\alpha_{k,k-1},\alpha_{k,k})$.

\smallskip
Actually, the choice of the non-localized profiles has some flexibility which means that we could allow the non-localized profiles to have tails on both sides of $\RR$. However, due to the special structure of the error terms (see for example Proposition~\ref{prop:appPhi} for the three-solitons case), such left-hand side tails will lead to some extra restrictions on the coefficients $\vec{\alpha}$ of the ODE system for parameters $\Gamma$. Based on some elementary computations, we find that these restrictions lead to the same coefficients as those obtained by choosing $X$ directly. In other words, this flexibility does not change the structure of the ODE system for parameters $\Gamma$. 

\smallskip
The first, and most important ingredient of the proof is the study of the dynamics of the above-mentioned ODE system~\eqref{equ:ODEsystem}. More precisely, we need to choose the suitable signs $\vec{\sigma}$ and coefficients $\vec{\alpha}$ such that the ODE system~\eqref{equ:ODEsystem} has an explicit solution related to the multi-soliton dynamics. Inspired by the discussion in~\cite[Lemma 4]{MartelRaphael}, for any  $k\in [\![1,K]\!]$, we expect that the solution is given by 
\begin{equation}\label{equ:solutionapp}
    \left(x_{k},\nu_{k},b_{k}\right)= \left((a+3k)\log t+\beta_{k},\frac{a+3k}{t},-\frac{a+3k}{2t^{2}}\right).
\end{equation}
Here, $(a,\beta_{1},\dots,\beta_{K})\in \RR^{K+1}$ are some constants to be determined later. The key point of the ingredient is that the error terms generated by the approximate solution should not affect the leading order of the ODE system~\eqref{equ:ODEsystem}. It follows directly that the coefficients $\vec{\alpha}$ are uniquely determined by the choice of signs $\vec{\sigma}$ and some constants related to integrations of $Q$ and $X$. Surprisingly, using a standard ODE argument, we find that the only choice of the number of solitons $K$ and signs $\vec{\sigma}$ for which the ODE system admits a straightforward explicit solution~\eqref{equ:solutionapp} is $K=3$ and $\vec{\sigma}=(1,-1,-1)$. In particular, for this choice, the solution matches the asymptotic behavior of the three-soliton solutions in Theorem~\ref{thm:main}. We refer to Appendix~\ref{Appen:ODE} for more details of the discussion on the ODE system~\eqref{equ:ODEsystem}.

\smallskip
The second ingredient of the proof is the monotonicity of the energy functional for KdV-type equation which is inspired by Kato's early work~\cite{Kato}. More precisely, for any solution $u$ of~\eqref{equ:gKdV} and smooth decreasing function $\phi$, we have 
\begin{equation*}
    \frac{\dd}{\dd t}\int_{\RR}u^{2}\phi \dd x+3\int_{\RR}(\partial_{x}u)^{2}\phi'\dd x={\rm{L.O.T}},
\end{equation*}
which means that the above $L^{2}$ quantity enjoys a monotonicity property along the solution flow. Recall that, the approximate solution $S$ in~\eqref{equ:defSintro} involves some non-localized profiles that have tails on the right-hand side of $\RR$. To control the effect of such tails, we introduce a localized $H^{1}$ norm for the remainder term and establish the monotonicity for the energy functional to control this norm (see~\eqref{est:Boot} and Subsection~\ref{SS:Energy} for more details). On the other hand, from~\eqref{equ:ODEsystem}, some delicate scaling terms related to $\nu_{k}$ do not have favorable signs. To overcome this difficulty, we follow a similar approach to the one applied in the subcritical and supercritical cases in~\cite[Subsection 3.1.3]{NguyenKdV} to construct a suitable refined term of the energy functional for the remainder term. Based on such energy estimate and ODE argument, we obtain a sequence of (backward-in-time) three-soliton solutions on finite time intervals. Last, we pass to the limit for the sequence of three-soliton solutions and obtain the solution described in Theorem~\ref{thm:main} via the weak $H^{1}$ continuity of~\eqref{equ:gKdV}.

\subsection{Notation and conventions}\label{SS:Nota}
Denote by $\mathcal{Y}$ the set of functions $f\in C^{\infty}(\RR;\RR)$ such that, for all $n\in \mathbb{N}$, there exist $(c_{n},C_{n})\in (0,\infty)^{2}$ such that 
\begin{equation*}
    |f^{(n)}(y)|\lesssim C_{n}(1+|y|)^{c_{n}}e^{-|y|},\quad \mbox{on}\ \RR.
\end{equation*}
Let $\vec{\sigma}=(\sigma_{1},\sigma_{2},\sigma_{3})=(1,-1,-1)$ be the choice of the signs for the solitons.

\smallskip
For any $\left(N_{1},N_{2}\right)\in \mathbb{N}\times \mathbb{N}$ with $N_{1}\le N_{2}$, we set $ [\![N_{1},N_{2}]\!]=\left\{N_{1},\dots,N_{2}\right\}$.

\smallskip
We define the $L^{2}$-scaling operator and the linearized operator around $Q$:
\begin{equation*}
    \Lambda f=\frac{1}{2}f+yf'\quad \mbox{and}\quad 
    \mathcal{L}f=-f''+f-5Q^{4}f.
\end{equation*}
For any functions $f\in L^{2}(\RR)$ and $g\in L^{2}(\RR)$, we denote the $L^{2}$-scalar product by 
\begin{equation*}
    (f,g)=\int_{\RR}f(y)g(y)\dd y.
\end{equation*}
Using the explicit expression of $Q$ in~\eqref{equ:defQ}, we find that as $y\to -\infty$,
\begin{equation}\label{est:asymQ-}
    Q(y)=c_{Q}e^{y}+O(e^{2y})\quad \mbox{and}\quad 
    Q'(y)=c_{Q}e^{y}+O(e^{2y}).
\end{equation}
We also find that as $y\to \infty$,
\begin{equation}\label{est:asymQ+}
    Q(y)=c_{Q}e^{-y}+O(e^{-2y})\quad \mbox{and}\quad 
    Q'(y)=-c_{Q}e^{-y}+O(e^{-2y}).
\end{equation}
Here, we denote $c_{Q}=12^{\frac{1}{4}}>0$. In addition, we define 
\begin{equation}\label{equ:defm0}
    \alpha=\frac{c_{Q}}{m_{0}^{2}}\int_{\RR}e^{y}Q^{5}(y)\dd y,\ \ \mbox{with}\ m_{0}=\frac{1}{4}\int_{\RR}Q(y)\dd y.
\end{equation}
Recall that, we introduce
\begin{equation}\label{equ:defbeta}
    \beta_{1}=0,\ \  \beta_{2}=\log \left(\frac{\alpha }{15}\right)\ \ \mbox{and}\ \ 
    \beta_{3}=\beta_{2}+\log \left(\frac{\alpha }{3}\right).
\end{equation}
For any $n\in \mathbb{N}$ and smooth real-valued function $f:\RR\to \RR$, we denote 
\begin{equation*}
    \partial_{y}^{\le n}f=\left(f,\partial_{y}f,\dots,\partial_{y}^{n}f\right),\quad \mbox{on}\ \RR.
\end{equation*}
Moreover, for any parameter $\nu\in (0,1)$ and such function $f:\RR\to \RR$, we set 
\begin{equation*}
    f_{1+\nu}(y)=(1+\nu)^{\frac{1}{4}}f((1+\nu)^{\frac{1}{2}}y),\quad \mbox{on}\ \RR.
\end{equation*}
From the Fundamental Theorem of Calculus, for such $\nu\in (0,1)$ and $f:\RR\to \RR$,
\begin{equation}\label{equ:Taylornu}
    f_{1+\nu}=f+\frac{\nu}{2}\Lambda f+\frac{\nu^{2}}{2}
    \int_{0}^{1}\left[
    \frac{(\Lambda^{2}f)_{1+\sigma\nu}}{2(1+\sigma\nu)^{2}}
    -\frac{(\Lambda f)_{1+\sigma \nu}}{(1+\sigma \nu)^{2}}
    \right](1-\sigma)\dd \sigma.
\end{equation}

\section{Non-linear profiles and interactions between the solitons}\label{S:intera}

\subsection{Structure of the linearized operator}\label{SS:linear}
We first recall the following standard properties without proof (see \emph{e.g.}~\cite[Lemma 2]{MMANN},~\cite[Lemma 2.1]{MMRACTA} and~\cite{Wein}).

\begin{proposition}
[Spectral theory of $\mathcal{L}$]\label{prop:L}
    The self-adjoint operator $\mathcal{L}$ on $L^{2}$ satisfies the following properties.
    \begin{enumerate}
        \item \emph{Eigenfunctions}. It holds 
        \begin{equation*}
            \mathcal{L}Q^{3}=-8Q^{3}\ \ 
            \mbox{and}\ \ {{\rm{Ker}}}\mathcal{L}={{\rm{Span}}}\left\{Q'\right\}.
        \end{equation*}

        \item \emph{Scaling identities}. It holds 
        \begin{equation*}
        \mathcal{L}\Lambda Q=-2Q\ \  \mbox{and}\ \  (\Lambda Q,Q)=0.
        \end{equation*}

        \item \emph{Inversion of $\mathcal{L}$}. For any real-valued function $h\in L^{2}$ with $(h,Q')=0$, there exists a unique real-valued function $f\in H^{2}$ with $(f,Q')=0$ such that $\mathcal{L}f=h$. Moreover, if $h\in \mathcal{Y}$, then we have $f\in \mathcal{Y}$.

        \item \emph{Coercivity of $\mathcal{L}$.} There exists $\mu>0$ such that 
        \begin{equation*}
            \left(\mathcal{L}f,f\right)\ge \mu \|f\|_{H^{1}}^{2}-
            \frac{1}{\mu}\left((f,Q)^{2}+(f,Q')^{2}+(f,\Lambda Q)^{2}\right).
        \end{equation*}
    \end{enumerate}
\end{proposition}

Second, we introduce the following localized profile $Y$ related to the existence of the resonance of the operator $\partial_{y}\mathcal{L}$. Such resonance has previously been used by Martel-Pilod~\cite{MDASNSP,MDPRINT} to study the finite point blow-up solution of gKdV equation~\eqref{equ:gKdV}.
\begin{lemma}\label{le:Y}
    There exists smooth function $Y\in \mathcal{Y}$ such that 
    \begin{equation*}
        \mathcal{L}Y=Q^{4}\ \ \mbox{and}\ \
        \left(\mathcal{L}(1+5Y)\right)'=0.
    \end{equation*}
    In particular, the following two identities hold:
    \begin{equation*}
       (Y,Q')=0\ \ \mbox{and}\ \  (Y,Q)=-\frac{3}{5}m_{0}.
    \end{equation*}
\end{lemma}

\begin{proof}
The existence of the function $Y$ follows directly from (iii) of Proposition~\ref{prop:L}. In addition, from (ii) of Proposition~\ref{prop:L} and integration by parts, we obtain
\begin{equation*}
    (Y,Q)=-\frac{1}{2}\left(Y,\mathcal{L}\Lambda Q\right)=-\frac{1}{4}\left(Q^{4},Q\right)-\frac{1}{2}\left(Q^{4},yQ'\right)=-\frac{3}{5}m_{0},
\end{equation*}
which directly completes the proof of Lemma~\ref{le:Y}.
\end{proof}

In addition, the non-localized profile $P$ was first introduced in~\cite[Proposition 2.2]{MMRACTA} to study the dynamics near the single soliton for mass-critical gKdV equation~\eqref{equ:gKdV}.

\begin{lemma}
    [\cite{MMRACTA}]\label{le:P}
    There exists a unique function $P$ such that $P'\in \mathcal{Y}$ and 
    \begin{equation*}
        (\mathcal{L}P)'=\Lambda Q,\ \ 
        \lim_{y\to -\infty}P(y)=2m_{0},\ \ \lim_{y\to \infty}P(y)=0.
    \end{equation*}
    In particular, the following two identities hold:
    \begin{equation*}
        (P,Q')=0\ \ \mbox{and}\ \
        (P,Q)=m_{0}^{2}>0.
    \end{equation*}
\end{lemma}

Define $X=P-2m_{0}(1+5Y)$. The following properties then follow from Lemma~\ref{le:Y}--\ref{le:P} and the definition of $\mathcal{Y}$ (see~\cite[Remark 2.5 and Lemma 2.6]{LYPRINT} for more details).
\begin{corollary}
[\cite{LYPRINT}]\label{coro:X}
   We have $X'\in \mathcal{Y}$ and 
    \begin{equation*}
        (\mathcal{L}X)'=\Lambda Q,\ \ 
        \lim_{y\to -\infty}X(y)=0,\ \ \lim_{y\to \infty}X(y)=-2m_{0}.
    \end{equation*}
     In particular, the following two identities hold:
    \begin{equation*}
        (X,Q')=0\ \ \mbox{and}\ \
        (X,Q)=-m_{0}^{2}<0.
    \end{equation*}
    Moreover, for any $n\in \mathbb{N}$, we have 
    \begin{equation*}
        \begin{aligned}
            |X^{(n)}(y)|\mathbf{1}_{(-\infty,0]}(y)&\lesssim (1+|y|^{2})e^{-|y|},\\
              |(X(y)+2m_{0})^{(n)}|\mathbf{1}_{(0,\infty)}(y)&\lesssim (1+|y|^{2})e^{-|y|}.
        \end{aligned}
    \end{equation*}
\end{corollary}

\subsection{Leading-order of the interactions}\label{SS:Leading}
Consider the following geometric parameters for the 3-soliton approximate solution:
\begin{equation*}
    \Gamma=\left(\Gamma_{1},\Gamma_{2},\Gamma_{3}\right)\in \RR^{9}, \ \ \
    \mbox{with}\ \  \Gamma_{k}=\left(x_{k},\nu_{k},b_{k}\right)\in \RR^{3}\ \ \mbox{for any}\  k\in [\![1,3]\!].
\end{equation*}
Here, the parameters $(x_{k},\nu_{k},b_{k})$ denote respectively the position, the scaling and a quantity related to the non-localized profile $X$ for the $k$-th soliton. Assume that 
\begin{equation}\label{est:xnu}
   r_{1}= x_{2}-x_{1}\gg 1\ \ \mbox{and}\ \ r_{2}=x_{3}-x_{2}\gg 1.
\end{equation}
Assume moreover that 
\begin{equation}\label{est:nub}
    \nu=\max(|\nu_{1}|,|\nu_{2}|,|\nu_{3}|)\ll 1\ \ \mbox{and}\ \ 
    b=\max(|b_{1}|,|b_{2}|,|b_{3}|)\ll 1.
\end{equation}
To state the estimates for the nonlinear interactions, we introduce 
\begin{equation}\label{est:rR}
    r=\min(r_{1},r_{2})\gg 1\ \ \mbox{and}\ \ 
    R=\max(r_{1},r_{2})\gg 1.
\end{equation}
In addition, we assume
\begin{equation}\label{est:xnu2}
    \nu R^{2}\lesssim \nu \left((x_{2}-x_{1})^{2}+(x_{3}-x_{2})^{2}\right)\ll 1.
\end{equation}

For any $k\in [\![1,3]\!]$ and smooth real-valued function $f:\RR\to \RR$, we denote 
\begin{equation}\label{def:T}
    (\Theta_{k}f)(t,y)=f_{1+\nu_{k}}(y-x_{k})
    =(1+\nu_{k})^{\frac{1}{4}}f((1+\nu_{k})^{\frac{1}{2}}(y-x_{k})).
\end{equation}
To simplify of notation, we also denote 
\begin{equation*}
    \Lambda_{k}\left(\Theta_{k}f\right)=\Theta_{k}(\Lambda f)=(1+\nu_{k})^{\frac{1}{4}}\left(\Lambda f\right)\left((1+\nu_{k})^{\frac{1}{2}}(y-x_{k})\right).
\end{equation*}
In particular, we set
\begin{equation}\label{equ:defQkYkXk}
    Q_{k}=\Theta_{k}Q,\ \
    \Lambda_{k}Q_{k}=\Theta_{k}(\Lambda Q), \ \
    Y_{k}=\Theta_{k}Y \ \  \mbox{and} \  \
    X_{k}=\Theta_{k}X.
\end{equation}
To state the nonlinear interactions between three solitons, we define 
\begin{equation}\label{equ:defG}
   U=Q_{1}-Q_{2}-Q_{3}\ \ \mbox{and}\ \  G=U^{5}-Q_{1}^{5}+Q_{2}^{5}+Q_{3}^{5}.
\end{equation}
Moreover, we introduce the following nonlinear interaction term:
\begin{equation}\label{equ:defH}
    H=5Q_{1}Q_{2}^{4}-5Q_{1}^{4}Q_{2}-5Q_{2}^{4}Q_{3}-5Q_{2}Q_{3}^{4}.
\end{equation}

We start with the following technical lemma concerning nonlinear interactions.
\begin{lemma}\label{le:boundinte}
    Let $(\theta,M_{1},M_{2})\in (0,\infty)\times \mathbb{N}^{2}$. Then the following estimates hold.
    \begin{enumerate}
        \item \emph{First-type estimates}.  Assume that $f_{1}\in C^{\infty}(\RR;\RR)\cap L^{\infty}(\RR;\RR)$ satisfies
    \begin{equation}\label{est:f1}
        |f_{1}(y)|+|f'_{1}(y)|\lesssim (1+|y|)^{M_{1}}e^{-\theta |y|},\ \ \mbox{on}\  \RR.
    \end{equation}
    Assume moreover that $f_{2}\in C^{\infty}(\RR;\RR)\cap L^{\infty}(\RR;\RR)$ satisfies
    \begin{equation}\label{est:f2}
    |f_{2}(y)|\mathbf{1}_{(-\infty,0]}(y)+|f'_{2}(y)|\lesssim (1+|y|)^{M_{2}}e^{-\theta|y|}, \ \ \mbox{on}\ \RR.
    \end{equation}
    Then, for all $(z_{1},z_{2})\in \RR^{2}$ with $D=z_{2}-z_{1}\gg 1$, we have 
    \begin{equation*}
    \begin{aligned}
       \left| \left(f_{1}(\cdot-z_{1}),f_{2}(\cdot-z_{2})\right)\right|&\lesssim 
       \left(|z_{1}|^{M_{1}+M_{2}+1}+|z_{2}|^{M_{1}+M_{2}+1}\right)e^{-\theta D},\\
       \|f_{1}(\cdot-z_{1})f_{2}(\cdot-z_{2})\|_{H^{1}}
       &\lesssim 
       \left(|z_{1}|^{M_{1}+M_{2}+\frac{1}{2}}+|z_{2}|^{M_{1}+M_{2}+\frac{1}{2}}\right)e^{-\theta D}.
       \end{aligned}
    \end{equation*}
    \item \emph{Second-type estimates}.  Assume that $f_{1}\in C^{\infty}(\RR;\RR)\cap L^{\infty}(\RR;\RR)$ satisfies
    \begin{equation*}
        |f_{1}(y)|\mathbf{1}_{(0,\infty)}(y)+|f'_{1}(y)|\lesssim (1+|y|)^{M_{1}}e^{-\theta|y|}, \ \ \mbox{on}\ \RR.
    \end{equation*}
    Assume moreover that $f_{2}\in C^{\infty}(\RR;\RR)\cap L^{\infty}(\RR;\RR)$ satisfies
     \begin{equation*}
        |f_{2}(y)|+|f'_{2}(y)|\lesssim (1+|y|)^{M_{2}}e^{-\theta |y|},\ \ \mbox{on}\  \RR.
    \end{equation*}
    Then, for all $(z_{1},z_{2})\in \RR^{2}$ with $D=z_{2}-z_{1}\gg 1$, we have 
    \begin{equation*}
    \begin{aligned}
       \left| \left(f_{1}(\cdot-z_{1}),f_{2}(\cdot-z_{2})\right)\right|&\lesssim 
       \left(|z_{1}|^{M_{1}+M_{2}+1}+|z_{2}|^{M_{1}+M_{2}+1}\right)e^{-\theta D},\\
       \|f_{1}(\cdot-z_{1})f_{2}(\cdot-z_{2})\|_{H^{1}}
       &\lesssim 
       \left(|z_{1}|^{M_{1}+M_{2}+\frac{1}{2}}+|z_{2}|^{M_{1}+M_{2}+\frac{1}{2}}\right)e^{-\theta D}.
       \end{aligned}
    \end{equation*}
    \end{enumerate}
  \end{lemma}

  \begin{proof}
      Proof of (i). First, from~\eqref{est:f2}, for $y\in (-\infty,z_{1})$, we have 
      \begin{equation*}
        |f_{2}(y-z_{2})|+|f'_{2}(y-z_{2})|\lesssim
        \left(1+|y-z_{1}|^{M_{2}}+|z_{2}|^{M_{2}}\right)e^{-\theta|y-z_{1}|}e^{-\theta D}.
      \end{equation*}
      It follows directly from~\eqref{est:f1} that 
      \begin{equation*}
      \begin{aligned}
         & |f_{1}(y-z_{1})f_{2}(y-z_{2})|+|(f_{1}(y-z_{1})f_{2}(y-z_{2}))'|\\
         &\lesssim 
         e^{-\theta D}\left(1+|y-z_{1}|^{M_{1}+M_{2}}+|z_{2}|^{M_{1}+M_{2}}\right)e^{-2\theta|y-z_{1}|}.
         \end{aligned}
      \end{equation*}
      Similarly, for $y\in (z_{2},\infty)$, we find
          \begin{equation*}
      \begin{aligned}
         & |f_{1}(y-z_{1})f_{2}(y-z_{2})|+|(f_{1}(y-z_{1})f_{2}(y-z_{2}))'|\\
         &\lesssim 
         e^{-\theta D}\left(1+|y-z_{2}|^{M_{1}+M_{2}}+|z_{1}|^{M_{1}+M_{2}}\right)e^{-\theta|y-z_{2}|}.
         \end{aligned}
      \end{equation*}
      Then, using again~\eqref{est:f1}--\eqref{est:f2}, for $y\in (z_{1},z_{2})$, we obtain 
          \begin{equation*}
      \begin{aligned}
          |f_{1}(y-z_{1})f_{2}(y-z_{2})|&\lesssim  e^{-\theta D}\left(1+|z_{1}|^{M_{1}+M_{2}}+|z_{2}|^{M_{1}+M_{2}}\right),\\
          |(f_{1}(y-z_{1})f_{2}(y-z_{2}))'|&\lesssim  e^{-\theta D}\left(1+|z_{1}|^{M_{1}+M_{2}}+|z_{2}|^{M_{1}+M_{2}}\right).
         \end{aligned}
      \end{equation*}
      Integrating the above estimates over $\RR$, we complete the proof of (i).

      \smallskip
      Proof of (ii). The proof is similar to (i), and we omit it.
  \end{proof}

  The following estimates then follow from~\eqref{est:asymQ-}--\eqref{est:asymQ+},~\eqref{est:xnu2} and Corollary~\ref{coro:X}.
  \begin{corollary}\label{coro:inte}
      It holds
\begin{equation*}
\begin{aligned}
    \left|(Q_{1},X_{2})\right|
    +\left|(Q_{2},X_{1}+2(1+\nu_{1})^{\frac{1}{4}}m_{0})\right|
    \lesssim R^{3}e^{-r},
    \\
\left|(Q_{1},X_{3})\right| +\left|(Q_{3},X_{1}+2(1+\nu_{1})^{\frac{1}{4}}m_{0})\right|
\lesssim R^{3}e^{-r},
\\
\left|(Q_{2},X_{3})\right| +\left|(Q_{3},X_{2}+2(1+\nu_{2})^{\frac{1}{4}}m_{0})\right|\lesssim R^{3}e^{-r}.
    \end{aligned}
\end{equation*}
Moreover, we have 
\begin{equation*}
    \begin{aligned}
         \left\|\partial_{y}\left(Q^{4}_{1}X_{2}\right)\right\|_{H^{1}}
          +\left\|
          \partial_{y}\left(Q^{4}_{2}(X_{1}+2(1+\nu_{1})^{\frac{1}{4}}m_{0})\right)\right\|_{H^{1}}\lesssim 
          R^{\frac{5}{2}}e^{-r},\\
           \left\|\partial_{y}\left(Q^{4}_{1}X_{3}\right)\right\|_{H^{1}}
          +\left\|\partial_{y}\left(Q^{4}_{3}(X_{1}+2(1+\nu_{1})^{\frac{1}{4}}m_{0})\right)\right\|_{H^{1}}\lesssim 
          R^{\frac{5}{2}}e^{-r},\\
           \left\|\partial_{y}\left(Q^{4}_{2}X_{3}\right)\right\|_{H^{1}}
          +\left\|\partial_{y}\left(Q^{4}_{3}(X_{2}+2(1+\nu_{2})^{\frac{1}{4}}m_{0})\right)\right\|_{H^{1}}\lesssim 
          R^{\frac{5}{2}}e^{-r}.
    \end{aligned}
\end{equation*}
  \end{corollary}
  Then, we study the nonlinear interactions between the three solitons.
  \begin{lemma}\label{le:asymptotic}
      The following estimates hold.
      \begin{enumerate}
          \item  \emph{Asymptotic.} We have 
      \begin{equation*}
      \begin{aligned}
         \partial_{y}H&=5c_{Q}e^{-r_{1}}\partial_{y}\left(e^{-(y-x_{2})}Q^{4}(y-x_{2})\right)
         \\
         &-5c_{Q}e^{-r_{2}}\partial_{y}\left(e^{-(y-x_{3})}Q^{4}(y-x_{3})\right)\\
         &-5c_{Q}e^{-r_{1}}\partial_{y}\left(e^{y-x_{1}}Q^{4}(y-x_{1})\right)
         +O_{H^{1}}\left(e^{-2r}\right)
         \\
         &-5c_{Q}e^{-r_{2}}\partial_{y}\left(e^{y-x_{2}}Q^{4}(y-x_{2})\right)
         +O_{H^{1}}\left(\nu R^{\frac{3}{2}}e^{-r}\right).
         \end{aligned}
      \end{equation*}

      \item \emph{Bound}. We have 
      \begin{equation*}
        \left\|  \partial_{y}G-\partial_{y}H\right\|_{H^{1}}\lesssim R^{\frac{1}{2}}e^{-2r}.
      \end{equation*}
      \end{enumerate}
  \end{lemma}

  \begin{proof}
      Proof of (i). Using~\eqref{equ:Taylornu}, we decompose 
      \begin{equation*}
      \begin{aligned}
          Q_{1}(y)&=Q(y-x_{1})+\frac{\nu_{1}}{2}(\Lambda Q)(y-x_{1})+O\left(\nu^{2}(y-x_{1})^{2}e^{-(1-\nu)|y-x_{1}|}\right),\\
          Q'_{1}(y)&=Q'(y-x_{1})+\frac{\nu_{1}}{2}(\Lambda Q)'(y-x_{1})+O\left(\nu^{2}(y-x_{1})^{2}e^{-(1-\nu)|y-x_{1}|}\right),\\
          Q''_{1}(y)&=Q''(y-x_{1})+\frac{\nu_{1}}{2}(\Lambda Q)''(y-x_{1})+O\left(\nu^{2}(y-x_{1})^{2}e^{-(1-\nu)|y-x_{1}|}\right).
          \end{aligned}
      \end{equation*}
      Similarly, we decompose
       \begin{equation*}
      \begin{aligned}
          Q_{2}(y)&=Q(y-x_{2})+\frac{\nu_{2}}{2}(\Lambda Q)(y-x_{2})+O\left(\nu^{2}(y-x_{2})^{2}e^{-(1-\nu)|y-x_{2}|}\right),\\
          Q'_{2}(y)&=Q'(y-x_{2})+\frac{\nu_{2}}{2}(\Lambda Q)'(y-x_{2})+O\left(\nu^{2}(y-x_{2})^{2}e^{-(1-\nu)|y-x_{2}|}\right),\\
          Q''_{2}(y)&=Q''(y-x_{2})+\frac{\nu_{2}}{2}(\Lambda Q)''(y-x_{2})+O\left(\nu^{2}(y-x_{2})^{2}e^{-(1-\nu)|y-x_{2}|}\right).
          \end{aligned}
      \end{equation*}
      It follows directly from~\eqref{est:xnu2} and Lemma~\ref{le:boundinte} that 
      \begin{equation}\label{est:Q14Q21}
      \begin{aligned}
         \partial_{y} \left(Q_{1}^{4}(y)Q_{2}(y)\right)&=
         \partial_{y}\left(Q^{4}(y-x_{1})Q(y-x_{2})\right)\\
         &+O_{H^{1}}\left(\nu R^{\frac{3}{2}}e^{-r}+\nu^{2}R^{\frac{5}{2}}e^{-r}\right).
         \end{aligned}
      \end{equation}
      Moreover, for $y\in \left(\frac{x_{1}+x_{2}}{2},\infty\right)$, from~\eqref{est:asymQ+}, we find 
      \begin{equation*}
          |\partial_{y}^{\le 2}Q(y-x_{1})|\lesssim e^{-(y-x_{1})}
          \Longrightarrow 
          \|Q^{4}(\cdot-x_{1})Q(\cdot-x_{2})\|_{H^{2}\left(\frac{x_{1}+x_{2}}{2},\infty\right)}\lesssim e^{-2r}.
      \end{equation*}
      Then, for $y\in \left(-\infty,\frac{x_{1}+x_{2}}{2}\right)$, from~\eqref{est:asymQ-}, we also find 
      \begin{equation*}
      \begin{aligned}
          Q(y-x_{2})&=c_{Q}e^{y-x_{2}}+O\left(e^{2(y-x_{2})}\right)=c_{Q}e^{-r_{1}}e^{y-x_{1}}+O\left(e^{2(y-x_{2})}\right),\\
          Q'(y-x_{2})&=c_{Q}e^{y-x_{2}}+O\left(e^{2(y-x_{2})}\right)
          =c_{Q}e^{-r_{1}}e^{y-x_{1}}+O\left(e^{2(y-x_{2})}\right),\\
           Q''(y-x_{2})&=c_{Q}e^{y-x_{2}}+O\left(e^{2(y-x_{2})}\right)
          =c_{Q}e^{-r_{1}}e^{y-x_{1}}+O\left(e^{2(y-x_{2})}\right).
          \end{aligned}
      \end{equation*}
      Therefore, for $y\in \left(-\infty,\frac{x_{1}+x_{2}}{2}\right)$, from the fact that $r_{1}=x_{2}-x_{1}$,
      \begin{equation*}
         \partial_{y} \left(Q^{4}(y-x_{1})Q(y-x_{2})\right)=c_{Q}e^{-r_{1}}\partial_{y}\left(e^{y-x_{1}}Q^{4}(y-x_{1})\right)+O_{H^{1}\left(-\infty,\frac{x_{1}+x_{2}}{2}\right)}(e^{-2r}).
      \end{equation*}
      Here, we use the fact that 
      \begin{equation*}
      \begin{aligned}
        &\int_{-\infty}^{\frac{x_{1}+x_{2}}{2}}\exp\left(-8|y-x_{1}|+4(y-x_{2})\right)\dd y\\
        &=e^{-4r_{1}}\int_{-\infty}^{\frac{r_{1}}{2}}\exp\left(-8|\sigma|+4\sigma\right)\dd \sigma=O\left(e^{-4r}\right).
        \end{aligned}
      \end{equation*}
      Combining the above estimate with~\eqref{est:Q14Q21}, we obtain 
      \begin{equation*}
      \begin{aligned}
           \partial_{y} \left(Q_{1}^{4}(y)Q_{2}(y)\right)&=c_{Q}e^{-r_{1}}\partial_{y}\left(e^{y-x_{1}}Q^{4}(y-x_{1})\right)\\
           &+O_{H^{1}}\left(\nu R^{\frac{3}{2}}e^{-r}
           +\nu^{2}R^{\frac{5}{2}}e^{-r}
           +e^{-2r}\right).
    \end{aligned}
      \end{equation*}
      Based on a similar argument as above, we also obtain 
     \begin{equation*}
      \begin{aligned}
           \partial_{y} \left(Q_{1}(y)Q_{2}^{4}(y)\right)&=c_{Q}e^{-r_{1}}\partial_{y}\left(e^{-(y-x_{2})}Q^{4}(y-x_{2})\right)\\
           &+O_{H^{1}}\left(\nu R^{\frac{3}{2}}e^{-r}
           +\nu^{2}R^{\frac{5}{2}}e^{-r}
           +e^{-2r}\right),
    \end{aligned}
      \end{equation*}
      \begin{equation*}
      \begin{aligned}
           \partial_{y} \left(Q_{2}^{4}(y)Q_{3}(y)\right)&=c_{Q}e^{-r_{2}}\partial_{y}\left(e^{y-x_{2}}Q^{4}(y-x_{2})\right)\\
           &+O_{H^{1}}\left(\nu R^{\frac{3}{2}}e^{-r}
           +\nu^{2}R^{\frac{5}{2}}e^{-r}
           +e^{-2r}\right),
    \end{aligned}
      \end{equation*}
      \begin{equation*}
      \begin{aligned}
           \partial_{y} \left(Q_{2}(y)Q_{3}^{4}(y)\right)&=c_{Q}e^{-r_{2}}\partial_{y}\left(e^{-(y-x_{3})}Q^{4}(y-x_{3})\right)\\
           &+O_{H^{1}}\left(\nu R^{\frac{3}{2}}e^{-r}
           +\nu^{2}R^{\frac{5}{2}}e^{-r}
           +e^{-2r}\right).
    \end{aligned}
      \end{equation*}
      We see that the estimate in (i) follows from the above estimates and~\eqref{est:xnu2}.

\smallskip
      Proof of (ii). From the definition of $G$ and $H$ in~\eqref{equ:defG}--\eqref{equ:defH}, we check that 
      \begin{equation*}
      \begin{aligned}
         \left|\partial_{y}^{\le 1} \left(\partial_{y}G-\partial_{y}H\right)\right|
         &\lesssim \left|\partial_{y}^{\le 2}Q_{1}\right|^{4}\left|\partial_{y}^{\le 2}Q_{3}\right|+\left|\partial_{y}^{\le 2}Q_{1}\right|\left|\partial_{y}^{\le 2}Q_{3}\right|^{4}\\
         &+\left|\partial_{y}^{\le 2}Q_{1}\right|^{3}\left|\partial_{y}^{\le 2}Q_{2}\right|^{2}+\left|\partial_{y}^{\le 2}Q_{1}\right|^{2}\left|\partial_{y}^{\le 2}Q_{2}\right|^{3}\\
         &+\left|\partial_{y}^{\le 2}Q_{2}\right|^{3}\left|\partial_{y}^{\le 2}Q_{3}\right|^{2}+\left|\partial_{y}^{\le 2}Q_{2}\right|^{2}\left|\partial_{y}^{\le 2}Q_{3}\right|^{3}.
         \end{aligned}
      \end{equation*}
Recall that, from~\eqref{equ:Taylornu}, for any $k\in [\![1,3]\!]$,
\begin{equation*}
      \begin{aligned}
          Q_{k}(y)&=Q(y-x_{k})+\frac{\nu_{k}}{2}(\Lambda Q)(y-x_{k})+O\left(\nu^{2}(y-x_{k})^{2}e^{-(1-\nu)|y-x_{k}|}\right),\\
          Q'_{k}(y)&=Q'(y-x_{k})+\frac{\nu_{k}}{2}(\Lambda Q)'(y-x_{k})+O\left(\nu^{2}(y-x_{k})^{2}e^{-(1-\nu)|y-x_{k}|}\right),\\
          Q''_{k}(y)&=Q''(y-x_{k})+\frac{\nu_{k}}{2}(\Lambda Q)''(y-x_{k})+O\left(\nu^{2}(y-x_{k})^{2}e^{-(1-\nu)|y-x_{k}|}\right).
          \end{aligned}
      \end{equation*}
      Combining the above estimates with~\eqref{est:xnu2} and Lemma~\ref{le:boundinte}, we obtain 
      \begin{equation*}
          \left\|\partial_{y}G-\partial_{y}H\right\|_{H^{1}}\lesssim R^{\frac{1}{2}}e^{-2r}+\nu R^{\frac{3}{2}}e^{-2r}+\nu^{4}R^{\frac{9}{2}}e^{-2r}.
      \end{equation*}
       We see that the estimate in (ii) follows from the above estimates and~\eqref{est:xnu2}.
      \end{proof}

    \section{Construction of the approximate solution}\label{S:ConsAPP}

\subsection{Approximate solution}\label{SS:App}
This subsection is devoted to the construction of a 3-soliton approximate solution to~\eqref{equ:gKdv2} and the derivation of the ODE system governing the evolution of the geometric parameters $\Gamma$. This system contains source terms that originate from the nonlinear interactions between the three solitons.

\smallskip
Recall that, we consider the following geometric parameters:
\begin{equation*}
    \Gamma=\left(\Gamma_{1},\Gamma_{2},\Gamma_{3}\right)\in \RR^{9}, \ \ 
    \mbox{with}\   \Gamma_{k}=\left(x_{k},\nu_{k},b_{k}\right)\in \RR^{3}\ \ \mbox{for any}\  k\in [\![1,3]\!].
\end{equation*}
We start with the definition of admissible geometrical parameters $\Gamma$.

\begin{definition}\label{def:admissible}
    Let $I=[t_{0},t_{1}]\subset (1,\infty)$ be a time interval.
    We say that the function $\Gamma(t):I\to \RR^{9}$ is admissible if it satisfies 
    \begin{equation*}
    \begin{aligned}
        |x_{1}(t)+15\log t-\beta_{1}|&\le 1,\ \ \nu_{1}\in \left(-\frac{16}{t},-\frac{14}{t}\right), \ \  \ b_{1}\in \left(\frac{7}{t^{2}},\frac{8}{t^{2}}\right),\\
     |x_{2}(t)+12\log t-\beta_{2}|&\le 1, \ \ \nu_{2}\in \left(-\frac{13}{t},-\frac{11}{t}\right),\ \ \
      b_{2}\in \left(\frac{5}{t^{2}},\frac{7}{t^{2}}\right),\\
      |x_{3}(t)+9\log t-\beta_{3}|&\le 1, \ \ \nu_{3}\in \left(-\frac{10}{t},-\frac{8}{t}\right),\ \ \ \ 
        b_{3}\in \left(\frac{4}{t^{2}},\frac{5}{t^{2}}\right).
        \end{aligned}
    \end{equation*}
\end{definition}

\begin{definition}\label{def:Sfunction}
    Let $I=[t_{0},t_{1}]\subset (1,\infty)$ be a time interval and $\Gamma$ be an admissible function. We denote by $\mathcal{S}$ the set of smooth functions $f:I\times \RR\to \RR$ such that 
    \begin{equation*}
        \|f(t)\|_{H^{1}}\lesssim \sum_{k=1}^{3}\frac{|\dot{x}_{k}|}{t^{\frac{5}{2}}}
        +
        \sum_{k=1}^{3}\frac{|\dot{\nu}_{k}|}{t^{\frac{3}{2}}}
    +\sum_{k=1}^{3}\frac{|\dot{b}_{k}|}{t^{\frac{1}{2}}}+\frac{1}{t^{\frac{7}{2}}}.
    \end{equation*}
    Moreover, for such function $f\in \mathcal{S}$, we set $f=O_{\mathcal{S}}(1)$.
\end{definition}

We denote 
\begin{equation*}
    a=\sqrt{\frac{2b_{1}}{15}}\in\left(\sqrt{\frac{14}{15}}t^{-1},\sqrt{\frac{16}{15}}t^{-1}\right)\Longrightarrow a^{-1}
    \in\left(\sqrt{\frac{15}{16}}t,\sqrt{\frac{15}{14}}t\right).
\end{equation*}
We also denote by $\chi:\RR\to [0,1]$ a non-decreasing $C^{\infty}$ function such that 
\begin{equation*}
    \chi_{|(-\infty,1)}\equiv 1 \ \ \mbox{and}\ \ 
    \chi_{|(2,\infty)}\equiv 0.
\end{equation*}
Recall that, the function $X$ is a non-localized profile which does not belong to $L^{2}$ space, and thus, we need the following suitable cut-off function:
\begin{equation}\label{equ:defphi}
    \phi(t,y)=\chi(a(t)y)\Longrightarrow 
    \phi_{|(-\infty,a^{-1})}\equiv 1\ \ \mbox{and}\ \ 
    \phi_{|(2a^{-1},\infty)}\equiv 0.
\end{equation}

\begin{lemma}\label{le:localX}
Let $\Gamma:[t_{0},t_{1}]\to \RR^{9}$ be an admissible function.
    We have
    \begin{equation*}
    \begin{aligned}
        \|X_{k}\partial^{3}_{y}\phi\|_{H^{1}}+\|(\partial_{y}X_{k})(\partial_{y}^{2}\phi)\|_{H^{1}}&\lesssim t^{-2},\\
        \|(\partial_{y}^{2}X_{k})(\partial_{y}\phi)\|_{H^{1}}+\|(\partial_{y}(U^{4}X_{k}))(\partial_{y}\phi)\|_{H^{1}}&\lesssim t^{-2}.
        \end{aligned}
    \end{equation*}
\end{lemma}

\begin{proof}
    From the definition of $\phi$ in~\eqref{equ:defphi}, we have 
    \begin{equation*}
t|\partial_{y}\phi|+t^{2}|\partial_{y}^{2}\phi|+t^{3}|\partial_{y}^{3}\phi|+
t^{4}|\partial_{y}^{4}\phi|
\lesssim \textbf{1}_{\left[\frac{t}{2},\frac{5t}{2}\right]}(y).
    \end{equation*}
    It follows from $X\in L^{\infty}$ and $X'\in \mathcal{Y}$ that 
    \begin{equation*}
        \sum_{k=1}^{3}\left\|X_{k}\partial_{y}^{3}\phi\right\|_{H^{1}}\lesssim t^{-3}\left\|\textbf{1}_{\left[\frac{t}{2},\frac{5t}{2}\right]}\right\|_{L^{2}}\lesssim t^{-\frac{5}{2}}.
    \end{equation*}
    Then, using again $X'\in \mathcal{Y}$ and the exponential decay of $Q$, 
    \begin{equation*}
    \begin{aligned}
        &\sum_{k=1}^{3}\left\|(\partial_{y}X_{k})(\partial_{y}^{2}\phi)\right\|_{H^{1}}
        +\sum_{k=1}^{3}\left\|(\partial_{y}^{2}X_{k})(\partial_{y}\phi)\right\|_{H^{1}}\\
        &+\sum_{k=1}^{3}\left\|(\partial_{y}(U^{4}X_{k}))(\partial_{y}\phi)\right\|_{H^{1}}
        \lesssim t^{-1}\sum_{k=1}^{3}\left(\int_{\frac{t}{2}}^{\frac{5t}{2}}e^{-\frac{1}{2}|y-x_{k}|}\dd y\right)^{\frac{1}{2}}\lesssim t^{-2}.
        \end{aligned}
    \end{equation*}
    Combining the above two estimates, we complete the proof of Lemma~\ref{le:localX}.
\end{proof}

To construct the explicit approximate solution, we need the following refined term related to the non-localized profile $X$\footnote{See~\eqref{equ:defQkYkXk} and Corollary~\ref{coro:X} for the definition of $(X_{1},X_{2},X_{3})$.}:
\begin{equation}\label{equ:defV}
    V=\sum_{k=1}^{3}\sigma_{k}\theta_{k}X_{k}=\theta_{1}X_{1}-\theta_{2}X_{2}-\theta_{3}X_{3}.
\end{equation}
Here, we denote 
\begin{equation}\label{equ:defthetak}
\theta_{k}=\frac{b_{k}}{(1+\nu_{k})^{\frac{5}{2}}}
\Longleftrightarrow 
b_{k}=\theta_{k}(1+\nu_{k})^{\frac{5}{2}},\ \  \mbox{for any}\ k\in [\![1,3]\!].
\end{equation}
Moreover, we introduce a refined term related to the localized profile $Y$\footnote{See~\eqref{equ:defQkYkXk} and Lemma~\ref{le:Y} for the definition of $(Y_{1},Y_{2},Y_{3})$.}:
\begin{equation}\label{equ:defW}
\begin{aligned}
    W&=-10m_{0}\theta_{1}\left(\frac{1+\nu_{1}}{1+\nu_{2}}\right)^{\frac{1}{4}}Y_{2}\\
    &+10m_{0}\left(\theta_{2}\left(\frac{1+\nu_{2}}{1+\nu_{3}}\right)^{\frac{1}{4}}-\theta_{1}\left(\frac{1+\nu_{1}}{1+\nu_{3}}\right)^{\frac{1}{4}}\right)Y_{3}.
    \end{aligned}
\end{equation}
\begin{remark}\label{re:VW}
We mention here that, the non-localized profile $V$ is used to express the nonlinear interactions between the three solitons as well as the localized profile $W$ is used to refine the interaction between the right-hand tail of $X_{k}$ and the three solitons. See Step 2 of the proof of Proposition~\ref{prop:appPhi} for more details.
\end{remark}
We now introduce the approximate solution $S$ and the related error term:
\begin{equation}\label{equ:defS}
    S=U+V\phi+W\ \ \mbox{and}\ \ 
    \Phi(S)=\partial_{t}S+\partial_{y}\left(\partial_{y}^{2}S-S+S^{5}\right).
\end{equation}
Consider the following functions which are related to the modulation equations:
\begin{equation}\label{equ:defMod}
\begin{aligned}
    {\vec{\rm{Mod}}}_{1}&=\left(\dot{x}_{1}-\nu_{1},\frac{\dot{\nu}_{1}-2b_{1}}{2(1+\nu_{1})},\dot{b}_{1}+\alpha e^{-r_{1}}\right),\\
    {\vec{\rm{Mod}}}_{2}&=\left(\dot{x}_{2}-\nu_{2},\frac{\dot{\nu}_{2}-2b_{2}}{2(1+\nu_{2})},\dot{b}_{2}
    +\alpha e^{-r_{1}}-\alpha e^{-r_{2}}
    \right),\\
     {\vec{\rm{Mod}}}_{3}&=\left(\dot{x}_{3}-\nu_{3},\frac{\dot{\nu}_{3}-2b_{3}}{2(1+\nu_{3})},\dot{b}_{3}
    +3\alpha e^{-r_{2}}
    \right).
    \end{aligned}
\end{equation}
Moreover, we set 
\begin{equation}\label{equ:defMQ}
\vec{M}_{k}Q=\sigma_{k}\left(-\partial_{y}Q_{k},\Lambda_{k}Q_{k},Z_{k}\right)^{T},\quad \mbox{for any}\ k\in [\![1,3]\!].
\end{equation}
Here, we denote 
\begin{equation}\label{equ:defZ}
    Z_{1}=X_{1}\phi-10m_{0}(Y_{2}+Y_{3}),\ \ 
    Z_{2}=X_{2}\phi-10m_{0}Y_{3} \ \ \mbox{and}\ \ 
    Z_{3}=X_{3}\phi.
\end{equation}
We now prove that the smooth real-valued function $S:I\times \RR\to \RR$ is an approximate solution of the gKdV equation~\eqref{equ:gKdv2} in the following sense.
\begin{proposition}
    \label{prop:appPhi}
     Let $\Gamma:[t_{0},t_{1}]\to \RR^{9}$ be an admissible function.
    We have
    \begin{equation}\label{equ:defPhi}
        \Phi(S)=\sum_{k=1}^{3}{\vec{\rm{Mod}}}_{k}\cdot\vec{M}_{k}Q
        +\sum_{k=1}^{3}\Phi_{k}(S)
        +O_{\mathcal{S}}(1).
    \end{equation}
\end{proposition}
Here, we denote 
\begin{equation*}
    \begin{aligned}
        \Phi_{1}(S)&=e^{-r_{1}}\left(
        5c_{Q}\partial_{y}
        \left(
        e^{-(y-x_{2})}Q^{4}(y-x_{2})
        \right)
        -\alpha X_{1}\phi 
        \right)\\
        &+e^{-r_{1}}\left(
         -5c_{Q}\partial_{y}
        \left(
        e^{y-x_{1}}Q^{4}(y-x_{1})
        \right)
          +\alpha X_{2}\phi+10\alpha m_{0} Y_{2}
        \right),\quad
    \end{aligned}
\end{equation*}

\begin{equation*}
    \begin{aligned}
        \Phi_{2}(S)
        &=e^{-r_{2}}
        \left(
 -5c_{Q}\partial_{y}\left(e^{y-x_{2}}Q^{4}(y-x_{2})\right)
 +3\alpha X_{3}\phi
 \right)\\
 &+e^{-r_{2}}\left(
        -5c_{Q}\partial_{y}\left(e^{-(y-x_{3})}Q^{4}(y-x_{3})\right)
        -\alpha X_{2}\phi+10\alpha m_{0}Y_{3}
        \right),\\
  \Phi_{3}(S)&=
  -\sum_{k=1}^{3}\sigma_{k}b_{k}(\dot{x}_{k}-\nu_{k})\partial_{y}X_{k}+10m_{0}b_{1}(\dot{x}_{2}-\nu_{2})\partial_{y}Y_{2}\\
  &+\sum_{k=1}^{3}\sigma_{k}\theta_{k}X_{k}(\partial_{t}\phi-\partial_{y}\phi)+10m_{0}(\dot{x}_{3}-\nu_{3})(b_{1}-b_{2})\partial_{y}Y_{3}.
        \end{aligned}
\end{equation*}

\begin{proof}
    \textbf{Step 1}. General computation. From the definition of $\Theta_{k}$ in~\eqref{def:T}, for any smooth real-valued function $f:\RR\to \RR$, we compute
    \begin{equation}\label{equ:ptTf}
        \partial_{t}\left(\Theta_{k}f\right)=-\dot{x}_{k}\partial_{y}(\Theta_{k}f)+\frac{\dot{\nu}_{k}}{2(1+\nu_{k})}\Theta_{k}\left(\Lambda f\right).
    \end{equation}
    It follows from the definition of $U$ in~\eqref{equ:defG} that 
    \begin{equation}\label{est:ptU}
        \partial_{t}U=-\sum_{k=1}^{3}\sigma_{k}\dot{x}_{k}\partial_{y}Q_{k}
        +\sum_{k=1}^{3}\frac{\sigma_{k}\dot{\nu}_{k}}{2(1+\nu_{k})}\Theta_{k}(\Lambda Q).
    \end{equation}
    Similarly, from the definition of $V$ in~\eqref{equ:defV}, we compute 
    \begin{equation*}
    \begin{aligned}
        \partial_{t}(V\phi)&=\sum_{k=1}^{3}\sigma_{k}\dot{\theta}_{k}X_{k}\phi
        -\sum_{k=1}^{3}\sigma_{k}\dot{x}_{k}\theta_{k}\left(\partial_{y}X_{k}\right)\phi\\
        &+\sum_{k=1}^{3}\sigma_{k}\theta_{k}X_{k}\partial_{t}\phi
        +\sum_{k=1}^{3}\frac{\sigma_{k}\theta_{k}\dot{\nu}_{k}}{2(1+\nu_{k})}\left(\Theta_{k}\left(\Lambda X\right)\right)\phi.
        \end{aligned}
    \end{equation*}
    Using the definition of $\theta_{k}$ in~\eqref{equ:defthetak} and Definition~\ref{def:admissible}, for any $k\in [\![1,3]\!]$,
    \begin{equation}\label{est:dttheta}
    \begin{aligned}
    \theta_{k}=b_{k}+O\left(t^{-3}\right)\ \ \mbox{and}\ \ 
        \dot{\theta}_{k}=\dot{b}_{k}+O\left(t^{-2}|\dot{\nu}_{k}|+t^{-1}|\dot{b}_{k}|\right).
        \end{aligned}
    \end{equation}
    It follows from $X'\in \mathcal{Y}$ and Definition~\ref{def:admissible}--\ref{def:Sfunction} that 
   \begin{equation}\label{est:ptVphi}
    \begin{aligned}
        \partial_{t}(V\phi)&=\sum_{k=1}^{3}\sigma_{k}\dot{b}_{k}X_{k}\phi
       +\sum_{k=1}^{3}\sigma_{k}\theta_{k}X_{k}\partial_{t}\phi\\
       & -\sum_{k=1}^{3}\sigma_{k}\dot{x}_{k}b_{k}\partial_{y}X_{k}
       +O_{\mathcal{S}}(1).
        \end{aligned}
    \end{equation}
    Here, we use the fact that 
    \begin{equation*}
       \sum_{k=1}^{3}\|X_{k}\phi\|_{H^{1}}+
        \sum_{k=1}^{3}\|(\Theta_{k}(\Lambda X))\phi\|_{H^{1}}\lesssim
        \left(
        \int_{\RR} \textbf{1}_{[\frac{t}{2},\frac{5t}{2}]}(y)\dd y
        \right)^{\frac{1}{2}}\lesssim t^{\frac{1}{2}}.
    \end{equation*}
   On the other hand, using again~\eqref{equ:ptTf},~\eqref{est:dttheta} and Definition~\ref{def:admissible}--\ref{def:Sfunction}, we find 
    \begin{equation}\label{est:dtW}
    \begin{aligned}
        \partial_{t}W&=10m_{0}\big(\dot{b}_{2}
        -\dot{b}_{1}\big)Y_{3}
        +10m_{0}\dot{x}_{3}(b_{1}-b_{2})\partial_{y}Y_{3}\\
        &-10m_{0}\dot{b}_{1}Y_{2}+10m_{0}\dot{x}_{2}b_{1}\partial_{y}Y_{2}
        +O_{\mathcal{S}}(1).
        \end{aligned}
    \end{equation}
    Here, we use the fact that 
    \begin{equation*}
        \frac{\dd}{\dd t}\left(\theta_{1}\left(\frac{1+\nu_{1}}{1+\nu_{2}}\right)^{\frac{1}{4}}\right)=\dot{b}_{1}+O\left(\frac{|\dot{\nu}_{1}|}{t^{2}}
        +\frac{|\dot{\nu}_{2}|}{t^{2}}
        +\frac{|\dot{b}_{1}|}{t}\right),
    \end{equation*}
    \begin{equation*}
         \frac{\dd}{\dd t}\left(\theta_{1}\left(\frac{1+\nu_{1}}{1+\nu_{3}}\right)^{\frac{1}{4}}\right)=\dot{b}_{1}+O\left(\frac{|\dot{\nu}_{1}|}{t^{2}}+\frac{|\dot{\nu}_{3}|}{t^{2}}+\frac{|\dot{b}_{1}|}{t}\right),
    \end{equation*}
      \begin{equation*}
         \frac{\dd}{\dd t}\left(\theta_{2}\left(\frac{1+\nu_{2}}{1+\nu_{3}}\right)^{\frac{1}{4}}\right)=\dot{b}_{2}+O\left(\frac{|\dot{\nu}_{2}|}{t^{2}}+\frac{|\dot{\nu}_{3}|}{t^{2}}+\frac{|\dot{b}_{2}|}{t}\right).
    \end{equation*}
    Combining~\eqref{est:ptVphi}--\eqref{est:dtW} with~\eqref{est:ptU}, we obtain 
    \begin{equation}\label{equ:dtS}
    \begin{aligned}
\partial_{t}S&=-\sum_{k=1}^{3}\sigma_{k}\dot{x}_{k}\partial_{y}Q_{k}
+\sum_{k=1}^{3}\frac{\sigma_{k}\dot{\nu}_{k}}{2(1+\nu_{k})}\Theta_{k}(\Lambda Q)
-10m_{0}\dot{b}_{1}Y_{2}
\\
&+\sum_{k=1}^{3}\sigma_{k}\theta_{k}X_{k}\partial_{t}\phi
-\sum_{k=1}^{3}\sigma_{k}\dot{x}_{k}b_{k}\partial_{y}X_{k}
+10m_{0}(\dot{b}_{2}-\dot{b}_{1})Y_{3}\\
&+\sum_{k=1}^{3}\sigma_{k}\dot{b}_{k}X_{k}\phi
+10m_{0}\dot{x}_{3}(b_{1}-b_{2})\partial_{y}Y_{3}+10m_{0}\dot{x}_{2}b_{1}\partial_{y}Y_{2}
+O_{\mathcal{S}}(1).
\end{aligned}
    \end{equation}
    On the other hand, from~\eqref{equ:defQ} and~\eqref{def:T}, we compute 
    \begin{equation*}
        -\partial_{y}^{2}Q_{k}+(1+\nu_{k})Q_{k}-Q_{k}^{5}=0,\quad \mbox{for any}\ k\in [\![1,3]\!].
    \end{equation*}
It follows directly from~\eqref{equ:defS} that 
\begin{equation}\label{est:pyS}
    \partial_{y}\left(
    \partial_{y}^{2}S-S+S^{5}
\right)=\sum_{k=1}^{3}\sigma_{k}\nu_{k}\partial_{y}Q_{k}+\mathcal{H}_{1}+\mathcal{H}_{2}.
\end{equation}
Here, we denote 
\begin{equation*}
    \begin{aligned}
     \mathcal{H}_{1}&=\partial_{y}\left(
        \partial_{y}^{2}(V\phi+W)-(V\phi+W)+5U^{4}(V\phi+W)
        \right)+\partial_{y}H,\\
        \mathcal{H}_{2}&=\partial_{y}\left((U+V\phi+W)^{5}-U^{5}-5U^{4}(V\phi+W)\right)+\partial_{y}G-\partial_{y}H.
    \end{aligned}
\end{equation*}

\smallskip
\textbf{Step 2.} Estimate on $\mathcal{H}_{1}$. We claim that 
\begin{equation}\label{est:H1}
\begin{aligned}
    \mathcal{H}_{1}&=
    5c_{Q}e^{-r_{1}}\partial_{y}\big(e^{-(y-x_{2})}Q^{4}(y-x_{2})\big)
    -\sum_{k=1}^{3}\frac{\sigma_{k}b_{k}}{1+\nu_{k}}\Theta_{k}(\Lambda Q)
    \\
    &-5c_{Q}e^{-r_{1}}\partial_{y}\left(e^{y-x_{1}}Q^{4}(y-x_{1})\right)
    -10m_{0}\sum_{k=1}^{2}\sigma_{k}b_{k}\nu_{3}\partial_{y}Y_{3}
\\
&-5c_{Q}e^{-r_{2}}\partial_{y}\big(e^{-(y-x_{3})}Q^{4}(y-x_{3})\big)-\sum_{k=1}^{3}\sigma_{k}\theta_{k}X_{k}\partial_{y}\phi+O_{\mathcal{S}}(1)\\
    &-5c_{Q}e^{-r_{2}}\partial_{y}\left(e^{y-x_{2}}Q^{4}(y-x_{2})\right)
+\sum_{k=1}^{3}\sigma_{k}b_{k}\nu_{k}\partial_{y}X_{k}-10m_{0}b_{1}\nu_{2}\partial_{y}Y_{2}.
\end{aligned}
\end{equation}
Indeed, from the definition of $V$ in~\eqref{equ:defV}, we decompose
\begin{equation}\label{equ:pyVphi}
\begin{aligned}
&\partial_{y}\left(
\partial_{y}^{2}(V\phi)-V\phi+5U^{4}(V\phi)
\right)\\
&=\sum_{k=1}^{3}\sigma_{k}\theta_{k}\left(
\partial_{y}\left(
\partial_{y}^{2}X_{k}-X_{k}+5U^{4}X_{k}
\right)
\right)\phi
+\sum_{k=1}^{3}\sigma_{k}\theta_{k}X_{k}\partial_{y}^{3}\phi
\\
&+\sum_{k=1}^{3}\sigma_{k}\theta_{k}\left(3\partial_{y}^{2}X_{k}-X_{k}+5U^{4}X_{k}\right)\partial_{y}\phi+3\sum_{k=1}^{3}\sigma_{k}\theta_{k}(\partial_{y}X_{k})\partial_{y}^{2}\phi.
\end{aligned}
\end{equation}
Using Corollary~\ref{coro:X}, for any $k\in [\![1,3]\!]$, we compute
\begin{equation*}
\begin{aligned}
    \partial_{y}\left(
\partial_{y}^{2}X_{k}-X_{k}+5U^{4}X_{k}\right)&=-(1+\nu_{k})^{\frac{3}{2}}\Theta_{k}(\Lambda Q)+\nu_{k}\partial_{y}X_{k}\\
&+5\partial_{y}\Big(\Big(U^{4}-\sum_{j=1}^{3}Q_{j}^{4}\Big)X_{k}\Big)+5\partial_{y}\sum_{\ell\ne k}Q_{\ell}^{4}X_{k}.
\end{aligned}
\end{equation*}
Moreover, from Lemma~\ref{le:boundinte},~Corollary~\ref{coro:inte} and Definition~\ref{def:admissible},
\begin{equation*}
\begin{aligned}
  &5\partial_{y}\Big(\Big(U^{4}-\sum_{j=1}^{3}Q_{j}^{4}\Big)X_{k}\Big)+5\partial_{y}\sum_{\ell\ne k}Q_{\ell}^{4}X_{k}\\
   &=-10m_{0}(1+\nu_{k})^{\frac{1}{4}}\partial_{y}\sum_{\ell=k+1}^{3}Q_{\ell}^{4}+O_{H^{1}}\left(t^{-\frac{5}{2}}\right).
    \end{aligned}
\end{equation*}
It follows from Definition~\ref{def:admissible}--\ref{def:Sfunction} and $(Q,X')\in \mathcal{Y}\times\mathcal{Y}$ that 
\begin{equation*}
\begin{aligned}
   & \sum_{k=1}^{3}\sigma_{k}\theta_{k}\left(
\partial_{y}\left(
\partial_{y}^{2}X_{k}-X_{k}+5U^{4}X_{k}
\right)
\right)\phi\\
&=-\sum_{k=1}^{3}\sigma_{k}\theta_{k}(1+\nu_{k})^{\frac{3}{2}}\Theta_{k}(\Lambda Q)+10m_{0}\theta_{2}(1+\nu_{2})^{\frac{1}{4}}\partial_{y}(Q_{3}^{4})\\
&+\sum_{k=1}^{3}\sigma_{k}\theta_{k}\nu_{k}\partial_{y}X_{k}-10m_{0}\theta_{1}(1+\nu_{1})^{\frac{1}{4}}\partial_{y}\big(Q_{2}^{4}+Q_{3}^{4}\big)+O_{\mathcal{S}}(1).
\end{aligned}
\end{equation*}
Then, using again Definition~\ref{def:admissible}--\ref{def:Sfunction} and Lemma~\ref{le:localX}, we find 
\begin{equation*}
    \begin{aligned}
      \sum_{k=1}^{3}
\sigma_{k}\theta_{k}\left(3\partial_{y}^{2}X_{k}+5U^{4}X_{k}\right)\partial_{y}\phi&=O_{\mathcal{S}}(1),\\
        \sum_{k=1}^{3}\sigma_{k}\theta_{k}X_{k}\partial_{y}^{3}\phi
      +3\sum_{k=1}^{3}\sigma_{k}\theta_{k}(\partial_{y}X_{k})\partial_{y}^{2}\phi&=O_{\mathcal{S}}(1).
    \end{aligned}
\end{equation*}
Combining the above estimates with~\eqref{equ:defthetak} and~\eqref{equ:pyVphi}, we obtain 
\begin{equation}\label{est:pyV}
\begin{aligned}
    &\partial_{y}\left(
\partial_{y}^{2}(V\phi)-V\phi+5U^{4}(V\phi)
\right)\\
&=\sum_{k=1}^{3}\sigma_{k}b_{k}\nu_{k}\partial_{y}X_{k}
+10m_{0}\theta_{2}(1+\nu_{2})^{\frac{1}{4}}\partial_{y}(Q_{3}^{4})
-\sum_{k=1}^{3}\sigma_{k}\theta_{k}X_{k}\partial_{y}\phi\\
&-\sum_{k=1}^{3}\frac{\sigma_{k}b_{k}}{1+\nu_{k}}\Theta_{k}(\Lambda Q)
-10m_{0}\theta_{1}(1+\nu_{1})^{\frac{1}{4}}\partial_{y}\big(Q_{2}^{4}+Q_{3}^{4}\big)
+O_{\mathcal{S}}(1).
\end{aligned}
\end{equation}
On the other hand, from the definition of $W$ in~\eqref{equ:defW}, we decompose 
\begin{equation*}
\begin{aligned}
    &\partial_{y}\left(
\partial_{y}^{2}W-W+5U^{4}W
\right)\\
&=-10m_{0}\theta_{1}\left(\frac{1+\nu_{1}}{1+\nu_{2}}\right)^{\frac{1}{4}}\partial_{y}\left(\partial_{y}^{2}Y_{2}-Y_{2}+5U^{4}Y_{2}\right)\\
&-10m_{0}\sum_{k=1}^{2}\sigma_{k}\theta_{k}\left(\frac{1+\nu_{k}}{1+\nu_{3}}\right)^{\frac{1}{4}}\partial_{y}\left(\partial_{y}^{2}Y_{3}-Y_{3}+5U^{4}Y_{3}\right).
\end{aligned}
\end{equation*}
Note that, from Lemma~\ref{le:Y}, Lemma~\ref{le:boundinte} and Definition~\ref{def:admissible},
\begin{equation*}
    \begin{aligned}
       \partial_{y} \left(\partial_{y}^{2}Y_{2}-Y_{2}+5U^{4}Y_{2}\right)&=\nu_{2}\partial_{y}Y_{2}-(1+\nu_{2})^{\frac{1}{4}}\partial_{y}\left(Q_{2}^{4}\right)+O_{H^{1}}\left(t^{-\frac{5}{2}}\right),\\
       \partial_{y} \left(\partial_{y}^{3}Y_{2}-Y_{3}+5U^{4}Y_{3}\right)&
        =\nu_{3}\partial_{y}Y_{3}-(1+\nu_{3})^{\frac{1}{4}}\partial_{y}\left(Q_{3}^{4}\right)+O_{H^{1}}\left(t^{-\frac{5}{2}}\right).
    \end{aligned}
\end{equation*}
It follows from~\eqref{est:dttheta} and Definition~\ref{def:admissible}--\ref{def:Sfunction} that 
\begin{equation}\label{est:pyW}
\begin{aligned}
    &\partial_{y}\left(
\partial_{y}^{2}W-W+5U^{4}W
\right)\\
&=-10m_{0}b_{1}\nu_{2}\partial_{y}Y_{2}+10m_{0}\theta_{1}(1+\nu_{1})^{\frac{1}{4}}\partial_{y}\left(Q_{2}^{4}\right)+O_{\mathcal{S}}(1)\\
&-10m_{0}\sum_{k=1}^{2}\sigma_{k}b_{k}\nu_{3}\partial_{y}Y_{3}+10m_{0}\sum_{k=1}^{2}\sigma_{k}\theta_{k}(1+\nu_{k})^{\frac{1}{4}}\partial_{y}\left(Q_{3}^{4}\right).
\end{aligned}
\end{equation}
Last, from (i) of Lemma~\ref{le:asymptotic} and Definition~\ref{def:admissible}--\ref{def:Sfunction}, we obtain 
\begin{equation*}
\begin{aligned}
    \partial_{y}H&=5c_{Q}\sum_{k=1}^{2}\sigma_{k}e^{-r_{k}}\partial_{y}\left(e^{-(y-x_{k+1})}Q^{4}(y-x_{k+1})\right)\\
         &-5c_{Q}\sum_{k=1}^{2}e^{-r_{k}}\partial_{y}\left(e^{y-x_{k}}Q^{4}(y-x_{k})\right)+O_{\mathcal{S}}(1).
    \end{aligned}
\end{equation*}
We see that~\eqref{est:H1} follows from~\eqref{est:pyV}--\eqref{est:pyW} and the above estimate.

\smallskip
\textbf{Step 3.} Estimate on $\mathcal{H}_{2}$. We claim that 
\begin{equation}\label{est:H2}
    \mathcal{H}_{2}=O_{\mathcal{S}}(1).
\end{equation}
Indeed, by an elementary computation, we find 
\begin{equation*}
\begin{aligned}
    S^{5}-U^{5}-5U^{4}(V\phi+W)
    &=10U^{3}(V\phi+W)^{2}+5U(V\phi+W)^{4}\\
    &+10U^{2}(V\phi+W)^{3}
    +(V\phi+W)^{5},
    \end{aligned}
\end{equation*}
which implies that 
\begin{equation*}
\begin{aligned}
   & \left|\partial_{y}\left(S^{5}-U^{5}-5U^{4}(V\phi+W)\right)\right|
   +\left|\partial^{2}_{y}\left(S^{5}-U^{5}-5U^{4}(V\phi+W)\right)\right|\\
   &\lesssim 
   |\partial_{y}^{\le 2}U|^{3}\left(|\partial_{y}^{\le 2}(V\phi)|^{2}+|\partial_{y}^{\le 2}W|^{2}\right)+|\partial_{y}^{\le 2}(V\phi)|^{5}+|\partial_{y}^{\le 2}W|^{5}.
    \end{aligned}
\end{equation*}
Based on the above estimate and Definition~\ref{def:admissible}, we directly have 
\begin{equation*}
\begin{aligned}
   &\left\|\partial_{y} \left((U+V\phi+W)^{5}-U^{5}-5U^{4}(V\phi+W)\right)\right\|_{H^{1}}\\
   &\lesssim \left\|\partial_{y}^{\le 2}(V\phi)\right\|_{L^{\infty}}^{4}
   \left\|\partial_{y}^{\le 2}(V\phi)\right\|_{L^{2}}+ 
   \left\|\partial_{y}^{\le 2}W\right\|_{L^{\infty}}^{4}
   \left\|\partial_{y}^{\le 2}W\right\|_{L^{2}}\\
   &+ \left\|\partial_{y}^{\le 2}(V\phi)\right\|_{L^{\infty}}^{2}\left\|\partial_{y}^{\le 2}U\right\|_{L^{6}}^{3}+\left\|\partial_{y}^{\le 2}W\right\|_{L^{\infty}}^{2}\left\|\partial_{y}^{\le 2}U\right\|_{L^{6}}^{3} 
   \lesssim
   t^{-4}.
   \end{aligned}
\end{equation*}
On the other hand, using again Lemma~\ref{le:asymptotic} and Definition~\ref{def:admissible},
\begin{equation*}
    \|\partial_{y}G-\partial_{y}H\|_{H^{1}}\lesssim R^{\frac{1}{2}}e^{-2r}\lesssim t^{-6}\log^{\frac{1}{2}} t\lesssim t^{-4}.
\end{equation*}
We see that~\eqref{est:H2} follows from the above two estimates and Definition~\ref{def:Sfunction}.

\smallskip
\textbf{Step 4.} Conclusion. From~\eqref{equ:dtS},~\eqref{est:pyS},~\eqref{est:H1} and~\eqref{est:H2}, we have 
\begin{equation*}
\begin{aligned}
    \Phi(S)&=-\sum_{k=1}^{3}\sigma_{k}b_{k}(\dot{x}_{k}-\nu_{k})\partial_{y}X_{k}+10m_{0}b_{1}(\dot{x}_{2}-\nu_{2})\partial_{y}Y_{2}\\
     &-\sum_{k=1}^{3}\sigma_{k}(\dot{x}_{k}-\nu_{k})\partial_{y}Q_{k}
    +\sum_{k=1}^{3}\frac{\sigma_{k}(\dot{\nu}_{k}-2b_{k})}{2(1+\nu_{k})}\Theta_{k}(\Lambda Q)\\
   &+\sum_{k=1}^{3}\sigma_{k}\theta_{k}X_{k}(\partial_{t}\phi-\partial_{y}\phi)
   +10m_{0}(\dot{x}_{3}-\nu_{3})(b_{1}-b_{2})\partial_{y}Y_{3}
   \\
   &+5c_{Q}e^{-r_{1}}\partial_{y}\left(e^{-(y-x_{2})}Q^{4}(y-x_{2})\right)
   -5c_{Q}e^{-r_{1}}\partial_{y}\left(e^{y-x_{1}}Q^{4}(y-x_{1})\right)\\
   &-5c_{Q}e^{-r_{2}}\partial_{y}\left(e^{-(y-x_{3})}Q^{4}(y-x_{3})\right)
        -5c_{Q}e^{-r_{2}}\partial_{y}\left(e^{y-x_{2}}Q^{4}(y-x_{2})\right)\\
    &+\dot{b}_{1}\left(X_{1}\phi-10m_{0}(Y_{2}+Y_{3})\right)-\dot{b}_{2}\left(X_{2}\phi-10m_{0}Y_{3}\right)-\dot{b}_{3}X_{3}\phi+O_{\mathcal{S}}(1).
    \end{aligned}
\end{equation*}
We see that~\eqref{equ:defPhi} follows directly from the above identity and~\eqref{equ:defMod}--\eqref{equ:defMQ}.
\end{proof}

Last, we introduce some technical estimates to be used in the study of the evolution for geometric parameters $\Gamma=(\Gamma_{1},\Gamma_{2},\Gamma_{3})\in \RR^{9}$.

\begin{lemma}
    The following estimates hold true.
    \begin{enumerate}\label{le:PhiTech}
        \item \emph{Estimate related to $\partial_{y}Q_{k}$}. We have 
        \begin{equation*}
           \sum_{k=1}^{3} \left|(\Phi_{1}(S),\partial_{y}Q_{k})\right|
            +\sum_{k=1}^{3} \left|(\Phi_{2}(S),\partial_{y}Q_{k})\right|
            \lesssim t^{-3}.
        \end{equation*}
        In addition, we have 
        \begin{equation*}
            \sum_{k=1}^{3} \left|(\Phi_{3}(S),\partial_{y}Q_{k})\right|\lesssim t^{-3}+t^{-2}\sum_{k=1}^{3}|\dot{x}_{k}-\nu_{k}|+t^{-2}\sum_{k=1}^{3}|\dot{b}_{k}|.
        \end{equation*}
          \item \emph{Estimate related to $\Lambda_{k}Q_{k}$}. We have 
        \begin{equation*}
           \sum_{k=1}^{3} \left|(\Phi_{1}(S),\Lambda_{k}Q_{k})\right|
            +\sum_{k=1}^{3} \left|(\Phi_{2}(S),\Lambda_{k}Q_{k})\right|
            \lesssim t^{-3}.
        \end{equation*}
        In addition, we have 
        \begin{equation*}
            \sum_{k=1}^{3} \left|(\Phi_{3}(S),\Lambda_{k}Q_{k})\right|\lesssim t^{-3}+t^{-2}\sum_{k=1}^{3}|\dot{x}_{k}-\nu_{k}|+t^{-2}\sum_{k=1}^{3}|\dot{b}_{k}|.
        \end{equation*}
        \item \emph{Estimate related to $Q_{k}$}. We have 
        \begin{equation*}
           \sum_{k=1}^{3} \left|(\Phi_{1}(S),Q_{k})\right|
            +\sum_{k=1}^{3} \left|(\Phi_{2}(S),Q_{k})\right|
           \lesssim t^{-4}.
        \end{equation*}
         In addition, we have 
        \begin{equation*}
            \sum_{k=1}^{3} \left|(\Phi_{3}(S),Q_{k})\right|\lesssim t^{-4}+t^{-3}\sum_{k=1}^{3}|\dot{x}_{k}-\nu_{k}|+t^{-3}\sum_{k=1}^{3}|\dot{b}_{k}|.
        \end{equation*}
    \end{enumerate}
\end{lemma}

\begin{proof}
    Proof of (i) and (ii). From the definition of $\phi$ in~\eqref{equ:defphi}, we compute 
    \begin{equation}\label{est:dtdyphi}
    \begin{aligned}
        |\partial_{t}\phi|+|\partial_{y}\phi|&\lesssim\frac{|\dot{b}_{1}|t}{\sqrt{b_{1}}}\textbf{1}_{[\frac{t}{2},\frac{5t}{2}]}(y)+
        t^{-1}\textbf{1}_{[\frac{t}{2},\frac{5t}{2}]}(y)\\
        &\lesssim 
        t^{2}|\dot{b}_{1}|\textbf{1}_{[\frac{t}{2},\frac{5t}{2}]}(y)+
        t^{-1}\textbf{1}_{[\frac{t}{2},\frac{5t}{2}]}(y).
        \end{aligned}
    \end{equation}
    Therefore, using Definition~\ref{def:admissible}, the expression of $(\Phi_{1},\Phi_{2},\Phi_{3})$ in Proposition~\ref{prop:appPhi} and  $(Q,X',Y,X)\in \mathcal{Y}\times \mathcal{Y}\times \mathcal{Y}\times L^{\infty}$, we directly complete the proof of (i) and (ii).

    \smallskip
    Proof of (iii). Using Lemma~\ref{le:Y},~Corollary~\ref{coro:X}, Lemma~\ref{le:boundinte} and Definition~\ref{def:admissible}, 
    \begin{equation*}
        (X_{1},Q_{1})=-m_{0}^{2}\quad \mbox{and}
\quad |(X_{2},Q_{1})|+|(Y_{2},Q_{1})|\lesssim t^{-1}.
    \end{equation*}
    Then, from~\eqref{equ:defm0}, Lemma~\ref{le:boundinte} and integration by parts, 
    \begin{equation*}
    \begin{aligned}
        \left( 5c_{Q}\partial_{y}
        \left(
        e^{-(y-x_{2})}Q^{4}(y-x_{2})
        \right),Q_{1}\right)&=O\left(t^{-1}\right),\\
        \left( -5c_{Q}\partial_{y}
        \left(
        e^{y-x_{1}}Q^{4}(y-x_{1})
        \right),Q_{1}\right)&=-\alpha m_{0}^{2}+O\left(t^{-1}\right).
        \end{aligned}
    \end{equation*}
    Gathering the above estimates, we obtain 
    \begin{equation*}
    \begin{aligned}
        &\left(\left(
         -5c_{Q}\partial_{y}
        \left(
        e^{y-x_{1}}Q^{4}(y-x_{1})
        \right)
          +\alpha X_{2}\phi+10\alpha m_{0} Y_{2}
        \right),Q_{1}\right)\\
        & +\left( \left(
        5c_{Q}\partial_{y}
        \left(
        e^{-(y-x_{2})}Q^{4}(y-x_{2})
        \right)
        -\alpha X_{1}\phi 
        \right),Q_{1}\right)=O\left(t^{-1}\right),
    \end{aligned}
    \end{equation*}
   which directly completes the proof of $(\Phi_{1},Q_{1})$ via Definition~\ref{def:admissible}.

   \smallskip
   Second, from Lemma~\ref{le:Y}, Lemma~\ref{le:boundinte} and Definition~\ref{def:admissible},
   \begin{equation*}
        (Y_{2},Q_{2})=-\frac{3}{5}m_{0}\quad \mbox{and}\quad 
        |(Y_{2},Q_{3})|=O(t^{-1}).
   \end{equation*}
  Similarly, using Lemma~\ref{le:boundinte} and Corollary~\ref{coro:inte}, we find 
   \begin{equation*}
   \begin{aligned}
       (X_{1},Q_{2})&=-8m_{0}^{2}+O\left(t^{-1}\right),\quad (X_{1},Q_{3})=-8m_{0}^{2}+O\left(t^{-1}\right),\\
       (X_{2},Q_{2})&=-m_{0}^{2}+O\left(t^{-1}\right),\quad \ \ (X_{2},Q_{3})=-8m_{0}^{2}+O\left(t^{-1}\right).
       \end{aligned}
   \end{equation*}
     Then, from~\eqref{equ:defm0}, Lemma~\ref{le:boundinte} and integration by parts, 
    \begin{equation*}
    \begin{aligned}
        \left( -5c_{Q}\partial_{y}
        \left(
        e^{y-x_{1}}Q^{4}(y-x_{1})
        \right),Q_{2}\right)&=O\left(t^{-1}\right),\\
          \left( 5c_{Q}\partial_{y}
        \left(
        e^{-(y-x_{2})}Q^{4}(y-x_{2})
        \right),Q_{3}\right)&=O\left(t^{-1}\right),\qquad \qquad
        \end{aligned}
    \end{equation*}
    \begin{equation*}
    \begin{aligned}
        \left( -5c_{Q}\partial_{y}
        \left(
        e^{y-x_{1}}Q^{4}(y-x_{1})
        \right),Q_{3}\right)&=O\left(t^{-1}\right),\\
         \left( 5c_{Q}\partial_{y}
        \left(
        e^{-(y-x_{2})}Q^{4}(y-x_{2})
        \right),Q_{2}\right)&=-\alpha m_{0}^{2}
        +O\left(t^{-1}\right).
        \end{aligned}
    \end{equation*}
   Gathering the above estimates, we obtain 
   \begin{equation*}
    \begin{aligned}
        &\left(\left(
         -5c_{Q}\partial_{y}
        \left(
        e^{y-x_{1}}Q^{4}(y-x_{1})
        \right)
          +\alpha X_{2}\phi+10\alpha m_{0} Y_{2}
        \right),Q_{2}\right)\\
        & +\left( \left(
        5c_{Q}\partial_{y}
        \left(
        e^{-(y-x_{2})}Q^{4}(y-x_{2})
        \right)
        -\alpha X_{1}\phi 
        \right),Q_{2}\right)=O\left(t^{-1}\right),
    \end{aligned}
    \end{equation*}
    \begin{equation*}
    \begin{aligned}
        &\left(\left(
         -5c_{Q}\partial_{y}
        \left(
        e^{y-x_{1}}Q^{4}(y-x_{1})
        \right)
          +\alpha X_{2}\phi+10\alpha m_{0} Y_{2}
        \right),Q_{3}\right)\\
        & +\left( \left(
        5c_{Q}\partial_{y}
        \left(
        e^{-(y-x_{2})}Q^{4}(y-x_{2})
        \right)
        -\alpha X_{1}\phi 
        \right),Q_{3}\right)=O\left(t^{-1}\right),
    \end{aligned}
    \end{equation*}
   which directly completes the proof of $(\Phi_{1},Q_{2})$ and $(\Phi_{1},Q_{3})$ via Definition~\ref{def:admissible}.

   \smallskip
   Third, using again Lemma~\ref{le:Y}, Lemma~\ref{le:boundinte},~Corollary~\ref{coro:inte} and Definition~\ref{def:admissible},
   \begin{equation*}
       |(X_{2},Q_{1})|+|(X_{3},Q_{1})|+|(Y_{3},Q_{1})|\lesssim t^{-1}.
   \end{equation*}
   Then, from Lemma~\ref{le:boundinte} and Definition~\ref{def:admissible}, we check that 
   \begin{equation*}
       \begin{aligned}
          \left( -5c_{Q}\partial_{y}\left(e^{y-x_{2}}Q^{4}(y-x_{2})\right),Q_{1}\right)&=O\left(t^{-1}\right),\\
           \big(-5c_{Q}\partial_{y}(e^{-(y-x_{3})}Q^{4}(y-x_{3})),Q_{1}\big)&=O\left(t^{-1}\right).
       \end{aligned}
   \end{equation*}
   Gathering the above estimates, we obtain 
   \begin{equation*}
   \begin{aligned}
       & \left(\left(
 -5c_{Q}\partial_{y}\left(e^{y-x_{2}}Q^{4}(y-x_{2})\right)
 +3\alpha X_{3}\phi
 \right),Q_{1}
 \right)
 \\
 &+
 \left(
 \left(
        -5c_{Q}\partial_{y}\left(e^{-(y-x_{3})}Q^{4}(y-x_{3})\right)
        -\alpha X_{2}\phi+10\alpha m_{0}Y_{3}
        \right),Q_{1}
\right)=O\left(t^{-1}\right),
        \end{aligned}
   \end{equation*}
   which directly completes the proof of $(\Phi_{2},Q_{1})$ via Definition~\ref{def:admissible}.

   \smallskip
   Similar to the above, we compute 
   \begin{equation*}
       \begin{aligned}
           (X_{2},Q_{2})&=-m_{0}^{2},\quad (X_{3},Q_{2})=O\left(t^{-1}\right),\quad 
           (Y_{3},Q_{2})=O\left(t^{-1}\right),\\
           (X_{3},Q_{3})&=-m_{0}^{2},\quad (Y_{3},Q_{3})=-\frac{3}{5}m_{0},\quad \ \ (X_{2},Q_{3})=-8m_{0}^{2}+O\left(t^{-1}\right).
       \end{aligned}
   \end{equation*}
    Then, using again Lemma~\ref{le:boundinte} and Definition~\ref{def:admissible}, we check that 
   \begin{equation*}
       \begin{aligned}
          \left( -5c_{Q}\partial_{y}\left(e^{y-x_{2}}Q^{4}(y-x_{2})\right),Q_{2}\right)&=-\alpha m_{0}^{2}+O\left(t^{-1}\right),\\
           \big(-5c_{Q}\partial_{y}(e^{-(y-x_{3})}Q^{4}(y-x_{3})),Q_{2}\big)&=O\left(t^{-1}\right),
       \end{aligned}
   \end{equation*}
   \begin{equation*}
       \begin{aligned}
          \left( -5c_{Q}\partial_{y}\left(e^{y-x_{2}}Q^{4}(y-x_{2})\right),Q_{3}\right)&=O\left(t^{-1}\right),\\
           \big(-5c_{Q}\partial_{y}(e^{-(y-x_{3})}Q^{4}(y-x_{3})),Q_{3}\big)&=\alpha m_{0}^{2}+O\left(t^{-1}\right),
       \end{aligned}
   \end{equation*}
    which directly completes the proof of $(\Phi_{2},Q_{2})$ and $(\Phi_{2},Q_{3})$ via Definition~\ref{def:admissible}.

    \smallskip
    Last, from Lemma~\ref{le:Y}, Corollary~\ref{coro:X} and Definition~\ref{def:admissible}, we find 
    \begin{equation*}
    \begin{aligned}
        \sum_{k=1}^{3}\left|(\partial_{y}X_{k},Q_{1})\right|+\left|(\partial_{y}Y_{2},Q_{1})\right|+\left|(\partial_{y}Y_{3},Q_{1})\right|&\lesssim t^{-1},\\
        \sum_{k=1}^{3}\left|(\partial_{y}X_{k},Q_{2})\right|+\left|(\partial_{y}Y_{2},Q_{2})\right|+\left|(\partial_{y}Y_{3},Q_{2})\right|&\lesssim t^{-1},\\
        \sum_{k=1}^{3}\left|(\partial_{y}X_{k},Q_{3})\right|+\left|(\partial_{y}Y_{2},Q_{3})\right|+\left|(\partial_{y}Y_{3},Q_{3})\right|&\lesssim t^{-1}.
        \end{aligned}
    \end{equation*}
    Therefore, using again~\eqref{est:dtdyphi}, the expression of $\Phi_{3}$ in Proposition~\ref{prop:appPhi} and  $(Q,X)\in \mathcal{Y}\times L^{\infty}$, we directly complete the proof of $(\Phi_{3},Q_{k})$ for any $k\in [\![1,3]\!]$.
\end{proof}

\subsection{Modulation of the approximate solution}
We state a standard modulation result for the approximate 3-soliton solution 
$S$. The proof follows a standard argument based on the Implicit Function Theorem (see \emph{e.g}~\cite[Proposition 3.1]{LYPRINT} and~\cite[Lemma 3.1]{MMARMA}) and we only sketch it below. For any $0<\delta\ll 1$, we denote\footnote{See~\eqref{est:xnu} and~\eqref{est:nub} for the definition of $(r_{1},r_{2},\nu,b)\in \RR^{4}$.}
\begin{equation*}
    \mathcal{V}_{\delta}=\left\{\Gamma\in \RR^{9}:r_{1}^{-1}+r_{2}^{-1}+\nu+b<\delta\right\}.
\end{equation*}

\begin{proposition}\label{prop:decom}
    Let $I=[t_{0},t_{1}]\subset (1,\infty)$ be a time interval. There exists $\delta_{0}>0$ such that if $w(t)$ is a solution of~\eqref{equ:gKdv2} on $I$ satisfying 
    \begin{equation}\label{est:modu}
       \inf_{\Gamma_{0}\in \mathcal{V}_{\delta_{0}}}\|w(t)-S(\cdot;\Gamma_{0})\|_{H^{1}}\le \delta_{0},\quad \mbox{for any}\ t\in I,
    \end{equation}
    then there exists a unique decomposition $(\varepsilon(t),\Gamma(t))$ of $w(t)$ on $I$,
    \begin{equation}\label{equ:defe}
        w(t,y)=S(y;\Gamma(t))+\varepsilon(t,y),\quad\mbox{for}\ (t,y)\in I\times\RR,
    \end{equation}
    such that for any $t\in I$,
    \begin{equation}\label{equ:ortho}
    \begin{aligned}
        (\varepsilon(t),\partial_{y}Q_{1})=(\varepsilon(t),\Lambda_{1}Q_{1})=(\varepsilon(t),Q_{1})&=0,\\
         (\varepsilon(t),\partial_{y}Q_{2})=(\varepsilon(t),\Lambda_{2}Q_{2})=(\varepsilon(t),Q_{2})&=0,\\
          (\varepsilon(t),\partial_{y}Q_{3})=(\varepsilon(t),\Lambda_{3}Q_{3})=(\varepsilon(t),Q_{3})&=0.
        \end{aligned}
    \end{equation}
    Moreover, the decomposition $(\varepsilon(t),\Gamma(t))$ satisfies
   
    \begin{equation*}
        r_{1}^{-1}(t)+r_{2}^{-1}(t)+\nu(t)+b(t)+\|\varepsilon(t)\|_{H^{1}}\lesssim \delta_{0}.
    \end{equation*}
\end{proposition}

\begin{proof}
    [Sketch of the proof]
    Define the functional
    \begin{equation*}
        (\omega,\Gamma)\mapsto  \left((\varepsilon(t),\partial_{y}Q_{k}),(\varepsilon(t),\Lambda_{k}Q_{k}),(\varepsilon(t),Q_{k})\right)_{k=1}^{3}\in \RR^{9},
    \end{equation*}
    where $\Gamma=(\Gamma_{1},\Gamma_{2},\Gamma_{3})\in \RR^{9}$. We compute the Jacobian matrix of the above mapping with respect to $\Gamma$ and evaluate it at $(S(\cdot;\Gamma_{0}),\Gamma_{0})\in H^{1}\times \mathcal{V}_{\delta_{0}}$. Up to some rescaling and translations, the heart of the proof is the invertibility of the matrix:
    \begin{equation*}
        \mathcal{M}=
        \begin{pmatrix}
            M_{1} & 0 & 0\\
            N_{1} & M_{2} & 0\\
            N_{2} & N_{3} & M_{3}
        \end{pmatrix}\in \RR^{9\times 9}.
    \end{equation*}
    Here, we set 
    \begin{equation*}
        M_{k}=\sigma_{k}\begin{pmatrix}
           \|Q'\|^{2}_{L^{2}}  & 0 & 0\\
            0 & -\frac{1}{2}\|\Lambda Q\|_{L^{2}}^{2}& -(X,\Lambda Q)\\
            0 & 0 & m_{0}^{2}
        \end{pmatrix}
        ,\quad \mbox{for any}\ k\in [\![1,3]\!].
    \end{equation*}
    Note that, the above explicit expression of the matrices $(M_{k})_{k=1}^{3}$ directly implies that the matrix $\mathcal{M}$ is invertible. Therefore, we can complete the proof of Proposition~\ref{prop:decom} via the uniform variant of the Implicit Function Theorem.
\end{proof}

From~\eqref{equ:gKdv2},~\eqref{equ:defS} and Proposition~\ref{prop:decom}, we directly obtain the equation of $\varepsilon(t)$.
\begin{lemma}[Equation of $\varepsilon$]\label{le:eque}
It holds 
\begin{equation*}
\begin{aligned}
    &\partial_{t}\varepsilon+\partial_{y}\left(\partial_{y}^{2}\varepsilon-\varepsilon+(S+\varepsilon)^{5}-S^{5}\right)+\Phi(S)=0.
    \end{aligned}
\end{equation*}
\end{lemma}

\section{Backwards uniform estimates}

In this section, we prove uniform estimates on a sequence of particular backward solutions. The key point is to carefully adjust their final data to obtain uniform estimates corresponding to the special strong interaction regime of Theorem~\ref{thm:main}.

\smallskip
Let $\Gamma^{in}=(\Gamma_{1}^{in},\Gamma_{2}^{in},\Gamma_{3}^{in})\in \RR^{9}$ to be chosen with $(r_{1}^{in})^{-1}+(r_{2}^{in})^{-1}+\nu^{in}+b^{in}\ll 1$.
Here, we set $\Gamma^{in}_{k}=(x_{k}^{in},\nu_{k}^{in},b_{k}^{in})\in \RR^{3}$ for any $k\in [\![1,3]\!]$.
In addition, we denote
\begin{equation*}
       r_{1}^{in}=x_{2}^{in}-x_{1}^{in}, \ \ 
    r_{2}^{in}=x_{3}^{in}-x_{2}^{in}, \ \
    \nu^{in}=\max_{k\in [\![1,3]\!]}|\nu_{k}^{in}|\ \ \mbox{and}\ \ 
    b^{in}=\max_{k\in [\![1,3]\!]}|b_{k}^{in}|.
\end{equation*}
Let $u(t)$ be the solution of~\eqref{equ:gKdV} with the following final data 
\begin{equation}\label{equ:deffinaldata}
    u(T_{n},x)=S\left(x-T_{n};\Gamma^{in}\right)\in H^{1}.
\end{equation}
By the renormalization $(t,x)\mapsto (t,y+t)$, we consider $w(t,y)=u(t,y+t)$ to be a solution of~\eqref{equ:gKdv2} on $t\le T_{n}$. Note that, $w(T_{n})$ satisfies~\eqref{est:modu} and, by continuity of the solution of~\eqref{equ:gKdV} in $H^{1}$, it exists and satisfies~\eqref{est:modu} on some maximal time interval $(T_{\rm{mod}},T_{n}]$ where $T_{\rm{mod}}\in (1,T_{n})$. For any $n\in \mathbb{N}^{+}$, we fix $T_{n}=n^{2}+1>1$.

\smallskip
On the time interval $(T_{\rm{mod}},T_{n}]$, we consider $(\varepsilon, \Gamma)$ as the decomposition of $w$ defined from Proposition~\ref{prop:decom}. Observe from~\eqref{equ:deffinaldata} and Proposition~\ref{prop:decom} that 
\begin{equation}\label{equ:modufinal}
    (\varepsilon(T_{n}),\Gamma(T_{n}))=(0,\Gamma^{in})\in H^{1}\times \RR^{9}.
\end{equation}
\begin{proposition}[Uniform backward estimates]\label{prop:uni}
There exists $T_{0}>1$ such that for all $T_{n}>T_{0}$ with $n\in \mathbb{N}^{+}$, there exists a choice of final data for geometric parameters $\Gamma^{in}\in \RR^{9}$ 
such that the solution $u$ of~\eqref{equ:gKdV} corresponding to~\eqref{equ:deffinaldata} exists on the time interval $[T_{0},T_{n}]$ and the solution $w(t,y)=u(t,y+t)$ of~\eqref{equ:gKdv2} satisfies~\eqref{est:modu} and~\eqref{equ:ortho} on the rescaled framework $(t,y)$. Moreover, the decomposition of $w$ satisfies the following uniform estimates, for all $t\in [T_{0},T_{n}]$,
\begin{equation}\label{est:unisolu}
        \begin{aligned}
            |x_{1}(t)+15\log t-\beta_{1}|\le \frac{1}{t^{\frac{1}{16}}},\ \ \big|\nu_{1}(t)+\frac{15}{t}\big|
            \le \frac{1}{t^{\frac{17}{16}}},
            \\
            |x_{2}(t)+12\log t-\beta_{2}|\le \frac{1}{t^{\frac{1}{16}}}, \ \ 
            \big|\nu_{2}(t)+\frac{12}{t}\big|\le \frac{1}{t^{\frac{17}{16}}},
            \\
            |x_{3}(t)+9\log t-\beta_{3}|\le \frac{1}{t^{\frac{1}{16}}}, \ \ 
            \big|\nu_{3}(t)+\frac{9}{t}\big|\le \frac{1}{t^{\frac{17}{16}}},\\
            \big|b_{1}-\frac{15}{2t^{2}}\big|+
             \big|b_{2}-\frac{6}{t^{2}}\big|+
              \big|b_{3}-\frac{9}{2t^{2}}\big|\le \frac{1}{t^{\frac{33}{16}}}\ \ \mbox{and}\ \ 
             \|\varepsilon(t)\|_{H^{1}}\le \frac{1}{t}.
        \end{aligned}
\end{equation}
\end{proposition}

The key point in Proposition~\ref{prop:uni} is that the initial time $T_{0}$ and the implied constants in~\eqref{est:unisolu} are independent of $n$ as $n\to \infty$. The rest of this section is devoted to the proof of Proposition~\ref{prop:uni}. First, in Subsection~\ref{SS:BOOT}, we introduce the bootstrap assumption which is related to the
above uniform estimates. Then, in Subsection~\ref{SS:Moduequ}-\ref{SS:Energy}, we establish some standard estimates related to the parameters $\Gamma\in \RR^{9}$ and the error term $\varepsilon\in H^{1}$ in the multi-soliton framework. Last, in Subsection~\ref{SS:EndPropuni}, we complete the proof of Proposition~\ref{prop:uni} via an ODE argument.

\subsection{Bootstrap setting}\label{SS:BOOT}
For $t\le T_{n}$, as long as $u(t)$ is well-defined, we decompose $w(t,y)=u(t,y+t)$ as in Proposition~\ref{prop:decom}. In particular, we denote by $(\varepsilon(t),\Gamma(t))$ the remainder term and geometric parameters of the decomposition of $u(t)$. In the rescaled framework of $(t,y)$, we denote by $\chi_{1}:\RR\to [0,1]$ a decreasing smooth function such that 
\begin{equation}\label{equ:defchi1}
		\chi_{1}(y)
		=\left\{
		\begin{aligned}
			& 1-e^{y} \ \ \mbox{on} \ (-\infty,-1], \\ 
			& e^{-y}\quad \ \ \mbox{on}\ [2,\infty),
		\end{aligned}
		\right.
		\quad \mbox{and}\quad 
		\chi'_{1}<0\ \ \mbox{on}\ \RR.
	\end{equation}
	For $A>1$ large to be chosen later, we define
	\begin{equation}\label{equ:defphi1}
		\mathcal{N}_{A}(\varepsilon)=\left(\int_{\RR}\left((\partial_{y}\varepsilon)^{2}+\varepsilon^{2}\phi_{1}\right)\dd y\right)^{\frac{1}{2}}\quad \mbox{with}\ \ \phi_{1}(y)=\chi_{1}\left(\frac{y}{A}-t^{\frac{1}{2}}\right).
	\end{equation}
    The proof of Proposition~\ref{prop:uni} follows from bootstrapping the following estimates:
   \begin{equation}\label{est:Boot}
        \begin{aligned}
            |x_{1}(t)+15\log t-\beta_{1}|\le \frac{1}{t^{\frac{1}{16}}}
            ,\ \ \big|\nu_{1}(t)+\frac{15}{t}\big|
            \le \frac{1}{t^{\frac{17}{16}}},
            \\
            |x_{2}(t)+12\log t-\beta_{2}|\le \frac{1}{t^{\frac{1}{16}}},
           \ \ 
           \big|\nu_{2}(t)+\frac{12}{t}\big|\le \frac{1}{t^{\frac{17}{16}}},
            \\
            |x_{3}(t)+9\log t-\beta_{3}|\le \frac{1}{t^{\frac{1}{16}}},
            \ \ 
            \big|\nu_{3}(t)+\frac{9}{t}\big|\le \frac{1}{t^{\frac{17}{16}}},\\
            \big|b_{1}-\frac{15}{2t^{2}}\big|+
             \big|b_{2}-\frac{6}{t^{2}}\big|+
              \big|b_{3}-\frac{9}{2t^{2}}\big|\le
              \frac{1}{t^{\frac{33}{16}}}
              \ \ \mbox{and}\ \ 
              \mathcal{N}_{A}(\varepsilon)\le \frac{C_{0}}{t^{\frac{7}{4}}}.
        \end{aligned}
    \end{equation}
    Here, $C_{0}>1$ is a large enough constant to be chosen later.

    \smallskip
    For $T_{0}\gg 1$ to be chosen (independent of $n\in \mathbb{N}^{+}$), and all $T_{n}\gg T_{0}$, we set 
    \begin{equation}\label{equ:estTBoot}
        T^{*}=T^{*}(\Gamma^{in})=\inf
        \left\{
        t\in [T_{0},T_{n}]: w(s)\ \mbox{satisfies}~\eqref{est:modu}\ \mbox{and}~\eqref{est:Boot} \ \mbox{on}\ [t,T_{n}]
        \right\}.
    \end{equation}
    Note that, from the choice of final data~\eqref{equ:deffinaldata} and the continuity of $H^{1}$ flow for~\eqref{equ:gKdV}, for the proof of Proposition~\ref{prop:uni}, we only need to prove that there exist $T_{0}\gg 1$ (independent of $n\in \mathbb{N}^{+}$) and at least one choice of $\Gamma^{in}\in \RR^{9}$ such that $T^{*}=T_{0}$.

    \smallskip
    Note also that, the geometric parameters $\Gamma\in \mathbb{R}^{9}$ satisfying~\eqref{est:Boot} directly satisfy the smallness conditions~\eqref{est:xnu}-\eqref{est:xnu2} and Definition~\ref{def:admissible}. Therefore, the technical lemmas that are established in Section~\ref{S:intera}-\ref{S:ConsAPP} hold true on $[T^{*},T_{n}]$.

    \smallskip
    In the rest of this section, the implied constants in $\lesssim$ and $O$ do not depend on the constant $C_{0}>1$ appearing in the bootstrap setting~\eqref{est:Boot} and do not depend on the constant $A>1$ appearing in~\eqref{equ:defphi1} for the definition of $\mathcal{N}_{A}(\varepsilon)$. In addition, we tacitly take the initial time $T_{0}\gg 1$ large enough which may depend on the constants $C_{0}>1$ and $A>1$.

    \smallskip
    We now state the $L^{2}$ and $L^{\infty}$ estimates for the remainder term $\varepsilon$.
    \begin{lemma}\label{le:L2e}
        It holds 
    \begin{equation*}
        \|\varepsilon(t)\|_{L^{2}}\lesssim t^{-1}\quad \mbox{and}\quad 
        \|\varepsilon(t)\|_{L^{\infty}}\lesssim t^{-\frac{5}{4}}.
    \end{equation*}
    \end{lemma}

    \begin{proof}
        First, using the orthogonality conditions~\eqref{equ:ortho}, we compute 
        \begin{equation*}
    \|w(t)\|_{L^{2}}^{2}=\|S(t)\|_{L^{2}}^{2}+\|\varepsilon(t)\|^{2}_{L^{2}}
            +2(V(t)\phi(t),\varepsilon(t))+2(W(t),\varepsilon(t)).
        \end{equation*}
        Based on the bootstrap setting~\eqref{est:Boot} and the definition of $(V,W)$ in~\eqref{equ:defV}--\eqref{equ:defW},
        \begin{equation*}
            \|V\phi\|_{L^{2}}^{2}+\|W\|_{L^{2}}^{2}\lesssim t\left(b_{1}^{2}+b_{2}^{2}+b_{3}^{2}\right)\lesssim t^{-3}.
        \end{equation*}
        It follows from the bootstrap setting~\eqref{est:Boot}, Lemma~\ref{le:boundinte} and Corollary~\ref{coro:inte} that 
        \begin{equation*}
            \|S(t)\|_{L^{2}}^{2}=3\|Q\|_{L^{2}}^{2}
            +\|V\phi\|_{L^{2}}^{2}+\|W\|_{L^{2}}^{2}
            +O\left(t^{-2}\right)
            =3\|Q\|_{L^{2}}^{2} +O\left(t^{-2}\right).
        \end{equation*}
        On the other hand, from the choice of the final data in~\eqref{equ:deffinaldata}, we find 
        \begin{equation*}
\|w(T_{n})\|_{L^{2}}^{2}=\|S(T_{n})\|_{L^{2}}^{2}=3\|Q\|_{L^{2}}^{2}+O\left(T_{n}^{-2}\right).
        \end{equation*}
        Combining the above estimates with the conservation law of mass, we obtain 
        \begin{equation*}
            \|\varepsilon(t)\|^{2}_{L^{2}}
            +2(V(t)\phi(t),\varepsilon(t))+2(W(t),\varepsilon(t))=O\left(t^{-2}\right),
        \end{equation*}
        which completes the proof for the $L^{2}$ estimate via the Cauchy-Schwarz inequality. Last, the $L^{\infty}$ estimate follows directly from
        the bootstrap setting~\eqref{est:Boot},
        the above $L^{2}$ estimate and the Fundamental Theorem of Calculus.
    \end{proof}

\subsection{Control of the geometric parameters}\label{SS:Moduequ}
First, we state the standard control of geometric parameters $\Gamma\in \RR^{9}$ via the orthogonality conditions~\eqref{equ:ortho}.

\smallskip
Note that, from the definition of $\phi_{1}$ in~\eqref{equ:defphi1} and the bootstrap setting~\eqref{est:Boot}, 
\begin{equation}\label{est:Naloc}
    \sum_{k=1}^{3}e^{-\frac{1}{2}|y-x_{k}|}\lesssim \phi_{1}(y),\ \ \mbox{on}\  \RR
    \Longrightarrow
    \sum_{k=1}^{3}\int_{\RR} |\varepsilon(t)|e^{-\frac{1}{2}|y-x_{k}|}\dd y\lesssim \mathcal{N}_{A}(\varepsilon).
\end{equation}
The above estimates will be used frequently in the proof of the following lemma.
\begin{lemma}\label{le:para1}
    The following estimates hold.
    \begin{enumerate}
        \item \emph{Estimates on $(x_{k},\nu_{k})$.} 
        We have 
        \begin{equation*}
            \sum_{k=1}^{3}|\dot{x}_{k}-\nu_{k}|
            +\sum_{k=1}^{3}|\dot{\nu}_{k}-2b_{k}|\lesssim \frac{C_{0}}{t^{\frac{7}{4}}}.
        \end{equation*}

        \item \emph{Estimates on $b_{k}$.}
        We have 
        \begin{equation*}
        \begin{aligned}
        |\dot{b}_{2}+\alpha e^{-r_{1}}-\alpha e^{-r_{2}}|&\lesssim \frac{C_{0}}{t^{\frac{13}{4}}}+\frac{C^{2}_{0}}{t^{\frac{7}{2}}} ,\\
           |\dot{b}_{1}+\alpha e^{-r_{1}}|+|\dot{b}_{3}+3\alpha e^{-r_{2}}|&\lesssim \frac{C_{0}}{t^{\frac{13}{4}}}+\frac{C^{2}_{0}}{t^{\frac{7}{2}}}.
           \end{aligned}
        \end{equation*}
    \end{enumerate}
\end{lemma}

\begin{proof}
\textbf{Step 1.} First estimate on $(x_{k},\nu_{k})$. From the equation of $\varepsilon$ in Lemma~\ref{le:eque} and the orthogonality condition $(\varepsilon(t),\partial_{y}Q_{1})=0$ in~\eqref{equ:ortho}, we compute 
    \begin{equation*}
    \begin{aligned}
      0=  \frac{\dd }{\dd t}(\varepsilon(t),\partial_{y}Q_{1})
      &=\left(\partial_{y}^{2}\varepsilon-\varepsilon+(S+\varepsilon)^{5}-S^{5},\partial_{y}^{2}Q_{1}\right)\\
      &  +(\varepsilon(t),\partial_{t}\partial_{y}Q_{1})
       -\left(\Phi(S),\partial_{y}Q_{1}\right).
        \end{aligned}
    \end{equation*}
    Note that, from the bootstrap setting~\eqref{est:Boot} and Lemma~\ref{le:L2e}, 
    \begin{equation*}
       \left| \left(\partial_{y}^{2}\varepsilon-\varepsilon+(S+\varepsilon)^{5}-S^{5},\partial_{y}^{2}Q_{1}\right)\right|\lesssim \mathcal{N}_{A}(\varepsilon)\lesssim \frac{C_{0}}{t^{\frac{7}{4}}}.
    \end{equation*}
    By an elementary computation, we find 
    \begin{equation*}
        \partial_{t}\partial_{y}Q_{1}=-\dot{x}_{1}\partial_{y}^{2}Q_{1}+\frac{\dot{\nu}_{1}}{2(1+\nu_{1})}\partial_{y}\Theta_{1}(\Lambda Q),
    \end{equation*}
    which directly implies that 
    \begin{equation*}
       \left| (\varepsilon(t),\partial_{t}\partial_{y}Q_{1})\right|\lesssim
       \left(|\dot{x}_{1}|+|\dot{\nu}_{1}|\right)\mathcal{N}_{A}(\varepsilon)\lesssim t^{-1} \left(|\dot{x}_{1}-\nu_{1}|+|\dot{\nu}_{1}-2b_{1}|\right)+\frac{C_{0}}{t^{\frac{7}{4}}}.
    \end{equation*}
    Similarly, from~\eqref{equ:defMod}--\eqref{equ:defMQ} and $(\Lambda Q,Q')=(Q,Q')=0$, we have
    \begin{equation*}
    \begin{aligned}
\sum_{k=1}^{3}\left({\vec{\rm{Mod}}}_{k}\cdot\vec{M}_{k}Q,\partial_{y}Q_{1}\right)
&=-(\dot{x}_{1}-\nu_{1})\|Q'\|_{L^{2}}^{2}
+O\Big(t^{-1}\sum_{k=1}^{3}|\dot{x}_{k}-\nu_{k}|\Big)\\
&+O\Big(t^{-1}\sum_{k=1}^{3}|\dot{\nu}_{k}-2b_{k}|\Big)+O\Big(|\dot{b}_{1}+\alpha e^{-r_{1}}|\Big)\\
&+O\Big(|\dot{b}_{3}+3\alpha e^{-r_{1}}|\Big)
+O\Big(|\dot{b}_{2}+\alpha e^{-r_{1}}-\alpha e^{-r_{2}}|\Big).
\end{aligned}
    \end{equation*}
    It follows from Lemma~\ref{le:PhiTech} and the definition of $\Phi(S)$ in Proposition~\ref{prop:appPhi} that 
    \begin{equation*}
        \begin{aligned}
\left(\Phi(S),\partial_{y}Q_{1}\right)
&=-(\dot{x}_{1}-\nu_{1})\|Q'\|_{L^{2}}^{2}
+O\Big(t^{-1}\sum_{k=1}^{3}|\dot{x}_{k}-\nu_{k}|\Big)\\
&+O\Big(t^{-1}\sum_{k=1}^{3}|\dot{\nu}_{k}-2b_{k}|\Big)+O\Big(|\dot{b}_{1}+\alpha e^{-r_{1}}|+t^{-3}\Big)\\
&+O\Big(|\dot{b}_{3}+3\alpha e^{-r_{1}}|\Big)
+O\Big(|\dot{b}_{2}+\alpha e^{-r_{1}}-\alpha e^{-r_{2}}|\Big).
        \end{aligned}
    \end{equation*}
    Here, we use the fact that for any $f\in \mathcal{S}$,
    \begin{equation*}
\begin{aligned}
      \left|(f,\partial_{y}Q_{1})\right|&\lesssim t^{-\frac{5}{2}}\sum_{k=1}^{3}|\dot{x}_{k}-\nu_{k}|+t^{-\frac{3}{2}}\sum_{k=1}^{3}|\dot{\nu}_{k}-2b_{k}|
      +t^{-\frac{1}{2}}|\dot{b}_{1}+\alpha e^{-r_{1}}|
      \\
      &+t^{-\frac{1}{2}}\left(|\dot{b}_{2}+\alpha e^{-r_{1}}-\alpha e^{-r_{2}}|+|\dot{b}_{3}+3\alpha e^{-r_{2}}|\right)+t^{-\frac{7}{2}}.
      \end{aligned}
    \end{equation*}
    In conclusion, the orthogonality condition $(\varepsilon(t),\partial_{y}Q_{1})=0$ gives
    \begin{equation*}
    \begin{aligned}
       |\dot{x}_{1}-\nu_{1}|&\lesssim 
       t^{-1}\sum_{k=1}^{3}|\dot{x}_{k}-\nu_{k}|
       +|\dot{b}_{2}+\alpha e^{-r_{1}}-\alpha e^{-r_{2}}|
       +\frac{C_{0}}{t^{\frac{7}{4}}}
       \\
&+t^{-1}\sum_{k=1}^{3}|\dot{\nu}_{k}-2b_{k}|+|\dot{b}_{1}+\alpha e^{-r_{1}}|
+|\dot{b}_{3}+3\alpha e^{-r_{1}}|+t^{-3}.
        \end{aligned}
    \end{equation*}
    Similar to the above, using the orthogonality conditions $(\varepsilon(t),\partial_{y}Q_{k})=0$ and $(\varepsilon(t),\Lambda_{k}Q_{k})=0$ for any $k\in [\![1,3]\!]$, we obtain 
    \begin{equation}\label{est:para11}
    \begin{aligned}
    \sum_{k=1}^{3}   |\dot{x}_{k}-\nu_{k}|
   + \sum_{k=1}^{3}   |\dot{\nu}_{k}-2b_{k}|
    &\lesssim |\dot{b}_{2}+\alpha e^{-r_{1}}-\alpha e^{-r_{2}}|
       +\frac{C_{0}}{t^{\frac{7}{4}}}
       \\
&+|\dot{b}_{1}+\alpha e^{-r_{1}}|
+|\dot{b}_{3}+3\alpha e^{-r_{1}}|+t^{-3}.
        \end{aligned}
    \end{equation}

    \smallskip
    \textbf{Step 2.} First estimate on $b_{k}$. Using again the equation of $\varepsilon$ in Lemma~\ref{le:eque} and the orthogonality condition $(\varepsilon(t),Q_{1})=0$ in~\eqref{equ:ortho}, we compute 
    \begin{equation*}
    \begin{aligned}
      0=  \frac{\dd }{\dd t}(\varepsilon(t),Q_{1})
      &=\left(\partial_{y}^{2}\varepsilon-\varepsilon+5Q_{1}^{4}\varepsilon
      ,\partial_{y}Q_{1}\right)+(\varepsilon(t),\partial_{t}Q_{1})\\
      &+\left((S+\varepsilon)^{5}-S^{5}-5Q_{1}^{4}\varepsilon,\partial_{y}Q_{1}\right)
       -\left(\Phi(S),Q_{1}\right).
        \end{aligned}
    \end{equation*}
    Based on ${{\mathrm{Ker}}\mathcal{L}}={\rm{Span}}\{Q'\}$ and the orthogonality condition $(\varepsilon(t),\partial_{y}Q_{1})=0$,
    \begin{equation*}
        \left(\partial_{y}^{2}\varepsilon-\varepsilon+5Q_{1}^{4}\varepsilon
      ,\partial_{y}Q_{1}\right)
      =\left(\partial_{y}^{2}\varepsilon-(1+\nu_{1})\varepsilon+5Q_{1}^{4}\varepsilon
      ,\partial_{y}Q_{1}\right)
      =0.
    \end{equation*}
    Then, using again the orthogonality conditions $(\varepsilon(t),\partial_{y}Q_{1})=(\varepsilon(t),\Lambda_{1}Q_{1})=0$, 
    \begin{equation*}
        \partial_{t}Q_{1}=-\dot{x}_{1}\partial_{y}Q_{1}+\frac{\dot{\nu}_{1}}{2(1+\nu_{1})}\Lambda_{1}Q_{1}\Longrightarrow
        (\varepsilon(t),\partial_{t}Q_{1})=0.
    \end{equation*}
    Next, using the definition of $S$ in~\eqref{equ:defS}, we have 
    \begin{equation*}
    \begin{aligned}
      \left|  (S+\varepsilon)^{5}-S^{5}-5Q_{1}^{4}\varepsilon\right|
      &\lesssim \varepsilon^{2}\left(|S|^{3}+|\varepsilon|^{3}\right)+\sum_{k_{1}\ne 1}\sum_{k_{2}=1}^{3}|Q_{k_{1}}||Q_{k_{2}}|^{3}|\varepsilon|\\
      &+\left(|V\phi|+|W|\right)\left(|U|^{3}+|V\phi|^{3}+|W|^{3}\right)|\varepsilon|.
      \end{aligned}
    \end{equation*}
    It follows from Lemma~\ref{le:boundinte}, Lemma~\ref{le:L2e} and the bootstrap setting~\eqref{est:Boot} that 
    \begin{equation*}
    \begin{aligned}
      \left|  \left((S+\varepsilon)^{5}-S^{5}-5Q_{1}^{4}\varepsilon,\partial_{y}Q_{1}\right)\right|&
      \lesssim \mathcal{N}_{A}(\varepsilon)\left(\mathcal{N}_{A}(\varepsilon)+r_{1}e^{-r_{1}}+r_{2}e^{-r_{2}}\right)\\
      &+\mathcal{N}_{A}(\varepsilon)\left(|b_{1}|+|b_{2}|+|b_{3}|\right)
      \lesssim \frac{C_{0}}{t^{\frac{15}{4}}}+\frac{C^{2}_{0}}{t^{\frac{7}{2}}}.
      \end{aligned}
    \end{equation*}
    On the other hand, from~\eqref{equ:defZ},~\eqref{est:Boot}, Corollary~\ref{coro:X} and Lemma~\ref{le:boundinte}, we find 
    \begin{equation*}
        (Z_{1},Q_{1})=(X_{1}\phi,Q_{1})+O\left(t^{-2}\right)
        =-m_{0}^{2}+O\left(t^{-2}\right).
    \end{equation*}
     Therefore, using again~\eqref{equ:defMod}--\eqref{equ:defMQ}, $(Q',Q)=(\Lambda Q,Q)=0$ and Corollary~\ref{coro:inte},
    \begin{equation*}
    \begin{aligned}
&\sum_{k=1}^{3}\left({\vec{\rm{Mod}}}_{k}\cdot\vec{M}_{k}Q,Q_{1}\right)\\
&=\left(-m_{0}^{2}+O\left(
t^{-2}\right)\right)
(\dot{b}_{1}+\alpha e^{-r_{1}})
+O(t^{-2}|\dot{x}_{2}-\nu_{2}|)\\
&+O(t^{-2}|\dot{x}_{3}-\nu_{3}|+t^{-2}|\dot{\nu}_{2}-2b_{2}|+t^{-2}|\dot{\nu}_{3}-2b_{3}|)\\
&+O(t^{-2}|\dot{b}_{2}+\alpha e^{-r_{1}}-\alpha e^{-r_{2}}|
+t^{-2}|\dot{b}_{3}+3\alpha e^{-r_{2}}|
).
\end{aligned}
 \end{equation*}
Combining the above estimate with Definition~\ref{def:Sfunction},~Proposition~\ref{prop:appPhi}, Lemma~\ref{le:PhiTech} and the estimates of $(x_{k},\nu_{k})$ in Step 1, we obtain
 \begin{equation*}
 \begin{aligned}
     \left(\Phi(S),Q_{1}\right)&=\left(-m_{0}^{2}+O\left(
t^{-2}\right)\right)
(\dot{b}_{1}+\alpha e^{-r_{1}})
+O\left(\frac{C_{0}}{t^{\frac{13}{4}}}+\frac{1}{t^{\frac{7}{2}}}\right)
\\
&+O\left(t^{-\frac{1}{2}}|\dot{b}_{2}+\alpha e^{-r_{1}}-\alpha e^{-r_{2}}|
+t^{-\frac{1}{2}}|\dot{b}_{3}+3\alpha e^{-r_{2}}|
\right).
\end{aligned}
 \end{equation*}
 In conclusion, the orthogonality condition $(\varepsilon(t),Q_{1})=0$ gives
 \begin{equation*}
 \begin{aligned}
     |\dot{b}_{1}+\alpha e^{-r_{1}}|
     &\lesssim t^{-\frac{1}{2}}|\dot{b}_{2}+\alpha e^{-r_{1}}-\alpha e^{-r_{2}}|\\
     &+t^{-\frac{1}{2}}|\dot{b}_{3}+3\alpha e^{-r_{2}}|+\frac{C_{0}}{t^{\frac{13}{4}}}+\frac{C^{2}_{0}}{t^{\frac{7}{2}}}.
     \end{aligned}
 \end{equation*}
 Using the other orthogonality conditions $(\varepsilon(t),Q_{2})=(\varepsilon(t),Q_{3})=0$, we obtain 
 \begin{equation}\label{est:parab1}
     \begin{aligned}
          |\dot{b}_{2}+\alpha e^{-r_{1}}-\alpha e^{-r_{2}}| &\lesssim 
     t^{-\frac{1}{2}}|\dot{b}_{3}+3\alpha e^{-r_{2}}|
     +t^{-\frac{1}{2}}|\dot{b}_{1}+\alpha e^{-r_{1}}|\\
     &+t^{-\frac{1}{2}}|\dot{b}_{2}+\alpha e^{-r_{1}}-\alpha e^{-r_{2}}|
     +\frac{C_{0}}{t^{\frac{13}{4}}}+\frac{C^{2}_{0}}{t^{\frac{7}{2}}},\\
      |\dot{b}_{1}+\alpha e^{-r_{1}}|+|\dot{b}_{3}+3\alpha e^{-r_{2}}|&
       \lesssim 
     t^{-\frac{1}{2}}|\dot{b}_{3}+3\alpha e^{-r_{2}}|
     +t^{-\frac{1}{2}}|\dot{b}_{1}+\alpha e^{-r_{1}}|\\
     &+t^{-\frac{1}{2}}|\dot{b}_{2}+\alpha e^{-r_{1}}-\alpha e^{-r_{2}}|
     +\frac{C_{0}}{t^{\frac{13}{4}}}+\frac{C^{2}_{0}}{t^{\frac{7}{2}}}.
     \end{aligned}
 \end{equation}
\textbf{Step 3.} Conclusion. Note that, from~\eqref{est:parab1}, we directly have 
 \begin{equation*}
        \begin{aligned}
        |\dot{b}_{2}+\alpha e^{-r_{1}}-\alpha e^{-r_{2}}|&\lesssim \frac{C_{0}}{t^{\frac{13}{4}}}+\frac{C^{2}_{0}}{t^{\frac{7}{2}}} ,\\
           |\dot{b}_{1}+\alpha e^{-r_{1}}|+|\dot{b}_{3}+3\alpha e^{-r_{2}}|&\lesssim \frac{C_{0}}{t^{\frac{13}{4}}}+\frac{C^{2}_{0}}{t^{\frac{7}{2}}},
           \end{aligned}
        \end{equation*}
        which completes the proof for the estimates in (ii) of Lemma~\ref{le:para1}. Combining the above estimates with~\eqref{est:para11}, we obtain 
        \begin{equation*}
            \begin{aligned}
    \sum_{k=1}^{3}   |\dot{x}_{k}-\nu_{k}|
   + \sum_{k=1}^{3}   |\dot{\nu}_{k}-2b_{k}|
   \lesssim\frac{C_{0}}{t^{\frac{7}{4}}}
+\frac{C_{0}}{t^{\frac{13}{4}}}+\frac{C^{2}_{0}}{t^{\frac{7}{2}}}
   +t^{-3},
        \end{aligned}
        \end{equation*}
        which completes the proof for the estimates in (i) of Lemma~\ref{le:para1}.
\end{proof}

Using the above lemma and the expression of $(\Phi_{1},\Phi_{2},\Phi_{3})$ in Proposition~\ref{prop:appPhi}, we obtain the following weighted $H^{1}$ and $L^{2}$ norm estimate for such error terms.

\begin{lemma}\label{le:Phiweight}
   The following estimates hold.
   \begin{enumerate}
       \item \emph{Weighted $H^{1}$ norm estimate.} We have 
   \begin{equation*}
       \begin{aligned}
            \|\Phi_{1}(S)\|_{\dot{H}^{1}}
            + \left( \int_{\RR}\left|\Phi_{1}(S)\right|^{2}\phi_{1}\dd y\right)^{\frac{1}{2}}\lesssim
            \frac{A^{\frac{1}{2}}}{t^{\frac{11}{4}}},\\
            \|\Phi_{2}(S)\|_{\dot{H}^{1}}
            + \left( \int_{\RR}\left|\Phi_{2}(S)\right|^{2}\phi_{1}\dd y\right)^{\frac{1}{2}}\lesssim
           \frac{A^{\frac{1}{2}}}{t^{\frac{11}{4}}},\\
            \|\Phi_{3}(S)\|_{\dot{H}^{1}}
            + \left( \int_{\RR}\left|\Phi_{3}(S)\right|^{2}\phi_{1}\dd y\right)^{\frac{1}{2}}\lesssim
           \frac{A^{\frac{1}{2}}}{t^{\frac{11}{4}}}.
       \end{aligned}
   \end{equation*}
      \item \emph{$L^{2}$ norm estimate.} We have 
       \begin{equation*}
          \|\Phi_{1}(S)\|_{L^{2}}
          +\|\Phi_{2}(S)\|_{L^{2}}
          +\|\Phi_{3}(S)\|_{L^{2}}
          \lesssim \frac{1}{t^{\frac{5}{2}}}.
       \end{equation*}
   \end{enumerate}
\end{lemma}

\begin{proof}
Proof of (i).
    Note that, from Corollary~\ref{coro:X} and the definition of $\phi_{1}$ in~\eqref{equ:defphi1},
    \begin{equation*}
    \begin{aligned}
        \int_{\RR}|X_{1}(y)|^{2}\phi_{1}(y)\dd y
        &\lesssim \int_{-\infty}^{x_{1}}e^{-|y-x_{1}|}\dd y+\int_{x_{1}}^{A(t^{\frac{1}{2}}+2)}1\dd y\\
        &+\int_{A(t^{\frac{1}{2}}+2)}^{\infty}
        \exp \left(-\left(\frac{y}{A}-t^{\frac{1}{2}}\right)\right)\dd y\lesssim
        A t^{\frac{1}{2}},
        \end{aligned}
    \end{equation*}
      \begin{equation*}
    \begin{aligned}
        \int_{\RR}|X_{2}(y)|^{2}\phi_{1}(y)\dd y
        &\lesssim \int_{-\infty}^{x_{2}}e^{-|y-x_{2}|}\dd y+\int_{x_{2}}^{A(t^{\frac{1}{2}}+2)}1\dd y\\
        &+\int_{A(t^{\frac{1}{2}}+2)}^{\infty}
        \exp \left(-\left(\frac{y}{A}-t^{\frac{1}{2}}\right)\right)\dd y\lesssim
        A t^{\frac{1}{2}},
        \end{aligned}
    \end{equation*}
      \begin{equation*}
    \begin{aligned}
        \int_{\RR}|X_{3}(y)|^{2}\phi_{1}(y)\dd y
        &\lesssim \int_{-\infty}^{x_{3}}e^{-|y-x_{3}|}\dd y+\int_{x_{3}}^{A(t^{\frac{1}{2}}+2)}1\dd y\\
        &+\int_{A(t^{\frac{1}{2}}+2)}^{\infty}
        \exp \left(-\left(\frac{y}{A}-t^{\frac{1}{2}}\right)\right)\dd y\lesssim
        A t^{\frac{1}{2}}.
        \end{aligned}
    \end{equation*}
    Similarly, we check that 
    \begin{equation*}
       \sum_{k=1}^{3} \int_{\RR}|X_{k}|^{2}|\partial_{y}\phi|^{2}\dd y\lesssim t^{-2}\int_{\RR}\mathbf{1}_{\left[\frac{t}{2},\frac{5t}{2}\right]}(y)\dd y\lesssim t^{-1}.
    \end{equation*}
    On the other hand, using the definition of $\phi$ in~\eqref{equ:defphi} and Lemma~\ref{le:para1}, we find 
    \begin{equation*}
    \begin{aligned}
        |\partial_{t}\phi-\partial_{y}\phi|&\lesssim 
\left(|\dot{a}|y+t^{-1}\right)\textbf{1}_{\left[\frac{t}{2},\frac{5t}{2}\right]}(y)
        \lesssim
        t^{-1}\textbf{1}_{\left[\frac{t}{2},\frac{5t}{2}\right]}(y),\\
         |\partial_{y}\left(\partial_{t}\phi-\partial_{y}\phi\right)|&\lesssim 
t^{-1}\left(|\dot{a}|y+t^{-1}\right)\textbf{1}_{\left[\frac{t}{2},\frac{5t}{2}\right]}(y)
        \lesssim
        t^{-2}\textbf{1}_{\left[\frac{t}{2},\frac{5t}{2}\right]}(y).
        \end{aligned}
    \end{equation*}
    It follows from Corollary~\ref{coro:X} and the definition of $\phi_{1}$ in~\eqref{equ:defphi1} that 
    \begin{equation*}
    \begin{aligned}
        \sum_{k=1}^{3}\|X_{k}(\partial_{t}\phi-\partial_{y}\phi)\|_{\dot{H}^{1}}&\lesssim t^{-1},\\
        \sum_{k=1}^{3}\left(\int_{\RR}X^{2}_{k}(\partial_{t}\phi-\partial_{y}\phi)^{2}\phi_{1}\dd y\right)^{\frac{1}{2}}&\lesssim A^{\frac{1}{2}}t^{-\frac{3}{4}}.
        \end{aligned}
    \end{equation*}
    Combining the above estimates with Proposition~\ref{prop:appPhi}, Lemma~\ref{le:para1} and the bootstrap setting~\eqref{est:Boot}, we directly complete the proof of the estimates in (i).

    \smallskip
     Proof of (ii). Using again Corollary~\ref{coro:X}, we compute 
    \begin{equation*}
    \begin{aligned}
        \sum_{k=1}^{3}\|X_{k}\phi\|_{L^{2}}&\lesssim 
        \left(\int_{\RR}\textbf{1}_{\left[\frac{t}{2},\frac{5t}{2}\right]}(y)\dd y\right)^{\frac{1}{2}}\lesssim t^{\frac{1}{2}},\\
         \sum_{k=1}^{3}\|X_{k}(\partial_{t}\phi-\partial_{y}\phi)\|_{L^{2}}&\lesssim 
       t^{-1} \left(\int_{\RR}\textbf{1}_{\left[\frac{t}{2},\frac{5t}{2}\right]}(y)\dd y\right)^{\frac{1}{2}}\lesssim t^{-\frac{1}{2}}.
        \end{aligned}
    \end{equation*}
    Combining the above estimates with Proposition~\ref{prop:appPhi}, Lemma~\ref{le:para1} and the bootstrap setting~\eqref{est:Boot}, we directly complete the proof of the estimates in (ii).
\end{proof}

For any $k\in [\![1,3]\!]$, we introduce
\begin{equation}\label{equ:defrhok}
    \rho_{k}(t,y)=\int_{y}^{\infty}(\Lambda Q)(\sigma-x_{k})\dd \sigma\in L^{\infty}\cap C^{\infty}.
\end{equation}
We denote by $\chi_{2}:\RR\to [0,1]$ a increasing smooth function such that 
\begin{equation}\label{equ:defphi2}
    {\chi_{2}}_{|(-\infty,-2)}\equiv 0\ \ 
    \mbox{and}\ \ 
    {\chi_{2}}_{|(-1,\infty)}\equiv 1\Longrightarrow
    \phi_{2}(y)=\chi_{2}\left(\frac{y}{t}\right).
\end{equation}
In addition, we denote 
\begin{equation}\label{equ:defF}
\begin{aligned}
    \mathcal{F}_{k}(t)&=\frac{1}{m_{0}^{2}}\left(\varepsilon(t),{\Omega}_{k}\phi_{2}\right)\ \ \mbox{for any}\  k\in [\![1,3]\!],
    \end{aligned}
\end{equation}
where 
\begin{equation*}
    {\Omega}_{1}=\rho_{1},\ \ {\Omega}_{2}=2\rho_{1}-\rho_{2}\ \ \mbox{and}\ \ 
    {\Omega}_{3}=-2\rho_{1}+2\rho_{2}-\rho_{3}.
\end{equation*}
\begin{lemma}\label{le:refinednu}
    It holds 
    \begin{equation*}
    \begin{aligned}
        \big|\dot{\nu}_{1}-2b_{1}+\dot{\mathcal{F}}_{1}\big|&\lesssim \frac{C_{0}}{t^{\frac{9}{4}}}+\frac{1}{t^{3}},\\
         \big|\dot{\nu}_{2}-2b_{2}+\dot{\mathcal{F}}_{2}\big|&\lesssim 
         \frac{C_{0}}{t^{\frac{9}{4}}}+\frac{1}{t^{3}},
         \\
          \big|\dot{\nu}_{3}-2b_{3}+\dot{\mathcal{F}}_{3}\big|&\lesssim 
          \frac{C_{0}}{t^{\frac{9}{4}}}+\frac{1}{t^{3}}.
        \end{aligned}
    \end{equation*}
\end{lemma}
We mention here that, the above lemma is inspired by~\cite[(iv) of Lemma 2.7]{MMRACTA} for the study of blow-up dynamics close to a single soliton.
To complete the proof of Lemma~\ref{le:refinednu}, we first need the following technical estimates.
\begin{lemma}\label{le:estrho1}
    The following estimates hold.
    \begin{enumerate}
\item \emph{Estimates related to $\partial_{y}Q_{k}$}. We have 
 \begin{equation*}
        \begin{aligned}
            \left|\left(\partial_{y}Q_{1},\rho_{1}\phi_{2}\right)\right|+ \left|\left(\partial_{y}Q_{1},\rho_{2}\phi_{2}\right)\right|+ \left|\left(\partial_{y}Q_{1},\rho_{3}\phi_{2}\right)\right|\lesssim t^{-1},\\
            \left|\left(\partial_{y}Q_{2},\rho_{1}\phi_{2}\right)\right|+ \left|\left(\partial_{y}Q_{2},\rho_{2}\phi_{2}\right)\right|+ \left|\left(\partial_{y}Q_{2},\rho_{3}\phi_{2}\right)\right|\lesssim t^{-1},\\
            \left|\left(\partial_{y}Q_{3},\rho_{1}\phi_{2}\right)\right|+ \left|\left(\partial_{y}Q_{3},\rho_{2}\phi_{2}\right)\right|+ \left|\left(\partial_{y}Q_{3},\rho_{3}\phi_{2}\right)\right|\lesssim t^{-1}.
            \end{aligned}
        \end{equation*}
    
        \item \emph{Estimates related to $\Lambda_{k}Q_{k}$}. We have 
        \begin{equation*}
        \begin{aligned}
            \left(\Lambda_{2}Q_{2},\rho_{1}\phi_{2}\right)=(\Lambda_{3}Q_{3},\rho_{1}\phi_{2})
            =(\Lambda_{3}Q_{3},\rho_{2}\phi_{2})
            &=O\left(t^{-1}\right),\\
              \left(\Lambda_{1}Q_{1},\rho_{2}\phi_{2}\right)=(\Lambda_{1}Q_{1},\rho_{3}\phi_{2})
            =(\Lambda_{2}Q_{2},\rho_{3}\phi_{2})
            &=4m_{0}^{2}+O\left(t^{-1}\right),\\
             \left(\Lambda_{1}Q_{1},\rho_{1}\phi_{2}\right)=(\Lambda_{2}Q_{2},\rho_{2}\phi_{2})
            =(\Lambda_{3}Q_{3},\rho_{3}\phi_{2})
            &=2m_{0}^{2}+O\left(t^{-1}\right).
            \end{aligned}
        \end{equation*}
        \item \emph{Estimate related to $Z_{k}$.} We have 
        \begin{equation*}
        \begin{aligned}
            \left|\left(Z_{1},\rho_{1}\phi_{2}\right)\right|+ \left|\left(Z_{1},\rho_{2}\phi_{2}\right)\right|+ \left|\left(Z_{1},\rho_{3}\phi_{2}\right)\right|\lesssim \log t,\\
            \left|\left(Z_{2},\rho_{1}\phi_{2}\right)\right|+ \left|\left(Z_{2},\rho_{2}\phi_{2}\right)\right|+ \left|\left(Z_{2},\rho_{3}\phi_{2}\right)\right|\lesssim \log t,\\
            \left|\left(Z_{3},\rho_{1}\phi_{2}\right)\right|+ \left|\left(Z_{3},\rho_{2}\phi_{2}\right)\right|+ \left|\left(Z_{3},\rho_{3}\phi_{2}\right)\right|\lesssim \log t.
            \end{aligned}
        \end{equation*}

        \item \emph{Estimate related to $\Phi_{k}$.} We have 
          \begin{equation*}
        \begin{aligned}
            \left|\left(\Phi_{1},\rho_{1}\phi_{2}\right)\right|+ \left|\left(\Phi_{1},\rho_{2}\phi_{2}\right)\right|+ \left|\left(\Phi_{1},\rho_{3}\phi_{2}\right)\right|\lesssim t^{-3}\log t,\\
            \left|\left(\Phi_{2},\rho_{1}\phi_{2}\right)\right|+ \left|\left(\Phi_{2},\rho_{2}\phi_{2}\right)\right|+ \left|\left(\Phi_{2},\rho_{3}\phi_{2}\right)\right|\lesssim t^{-3}\log t,\\
            \left|\left(\Phi_{3},\rho_{1}\phi_{2}\right)\right|+ \left|\left(\Phi_{3},\rho_{2}\phi_{2}\right)\right|+ \left|\left(\Phi_{3},\rho_{3}\phi_{2}\right)\right|\lesssim t^{-3}\log t.
            \end{aligned}
        \end{equation*}
    \end{enumerate}
\end{lemma}

\begin{proof}
Proof of (i). Note that, using integration by parts, we find 
\begin{equation*}
\begin{aligned}
    \left(\partial_{y}Q_{1},\rho_{1}\phi_{2}\right)&=\left(Q_{1},\phi_{2}(\Lambda Q)(\cdot-x_{1})\right)-\left(Q_{1},\rho_{1}\partial_{y}\phi_{2}\right),\\
     \left(\partial_{y}Q_{1},\rho_{2}\phi_{2}\right)&=\left(Q_{1},\phi_{2}(\Lambda Q)(\cdot-x_{2})\right)-\left(Q_{1},\rho_{2}\partial_{y}\phi_{2}\right),\\
      \left(\partial_{y}Q_{1},\rho_{3}\phi_{2}\right)&=\left(Q_{1},\phi_{2}(\Lambda Q)(\cdot-x_{3})\right)-\left(Q_{1},\rho_{3}\partial_{y}\phi_{2}\right).
    \end{aligned}
\end{equation*}
Note also that, from the definition of~\eqref{equ:defphi2}, we have 
\begin{equation*}
    |\partial_{y}\phi_{2}(y)|
    \lesssim t^{-1}\left|\chi_{2}'\left(\frac{y}{t}\right)\right|
    \lesssim t^{-1}\textbf{1}_{[-2t,-t]}(y),\quad \mbox{on}\ \RR.
\end{equation*}
It follows from the bootstrap setting~\eqref{est:Boot} and $(\Lambda Q,Q)=0$ that 
\begin{equation*}
    \left| \left(\partial_{y}Q_{1},\rho_{1}\phi_{2}\right)\right|\lesssim |\nu_{1}|+\int_{\RR}Q_{1}\left(|\partial_{y}\phi_{2}|+|1-\phi_{2}|\right)\dd y\lesssim t^{-1}.
\end{equation*}
On the other hand, using the bootstrap setting~\eqref{est:Boot} and Lemma~\ref{le:boundinte},
\begin{equation*}
\begin{aligned}
      \left|\left(\partial_{y}Q_{1},\rho_{2}\phi_{2}\right)\right|+  \left|\left(\partial_{y}Q_{1},\rho_{3}\phi_{2}\right)\right|
      &\lesssim\int_{\RR}
      Q_{1}\left(|\partial_{y}\phi_{2}|+|1-\phi_{2}|\right)\dd y\\
      & +r_{2}^{2} e^{-r_{2}}+\left(r_{2}^{2}+r_{3}^{2}\right)e^{-(r_{2}+r_{3})}\lesssim t^{-1}.
      \end{aligned}
\end{equation*}
Combining the above estimates, we complete the proof of the estimates in the first line. By a similar argument, we also obtain the remaining estimates in (i).

\smallskip
    Proof of (ii). To simplify the notation, we introduce 
    \begin{equation*}
        \rho(y)=\int_{y}^{\infty}\Lambda Q(\sigma)\dd \sigma
        \Longrightarrow \rho_{k}(y)=\rho(y-x_{k}),\ \  \mbox{for any}\ k\in [\![1,3]\!].
    \end{equation*}
    First, from the exponential decay of $\Lambda Q$, we check that
\begin{equation*}
    |\rho(y)|\textbf{1}_{(0,\infty)}(y)+|\rho'(y)|\lesssim (1+|y|)e^{-|y|}.
\end{equation*}
Then, using again ~\eqref{equ:Taylornu} and the bootstrap setting~\eqref{est:Boot}, we rewrite 
\begin{equation}\label{est:LambdaQk}
   \begin{aligned}
       (\Lambda_{1}Q_{1})(y)=(\Lambda Q)(y-x_{1})+O\left(t^{-1}e^{-\frac{1}{2}|y-x_{1}|}\right),\\
       ( \Lambda_{2}Q_{2})(y)=(\Lambda Q)(y-x_{2})+O\left(t^{-1}e^{-\frac{1}{2}|y-x_{2}|}\right),\\
         (\Lambda_{3}Q_{3})(y)=(\Lambda Q)(y-x_{3})+O\left(t^{-1}e^{-\frac{1}{2}|y-x_{3}|}\right).
   \end{aligned}
\end{equation}
Combining the above estimates with (ii) of Lemma~\ref{le:boundinte} and the bootstrap setting~\eqref{est:Boot}, we complete the proof of the estimates in the first line.

\smallskip
On the other hand, from the definition of $m_{0}$ in~\eqref{equ:defm0}, we have
\begin{equation}\label{equ:inteLambdaQ}
    \int_{\RR} \Lambda Q(y)\dd y=\frac{1}{2}\int_{\RR} Q(y)\dd y+\int_{\RR} yQ'(y)\dd y=-2m_{0},
\end{equation}
which implies that 
\begin{equation*}
    |\rho(y)+2m_{0}|\textbf{1}_{(-\infty,0]}(y)+|(\rho(y)+2m_{0})'|\lesssim (1+|y|)e^{-|y|}.
\end{equation*}
Combining the above estimate with~\eqref{est:LambdaQk}, (i) of Lemma~\ref{le:boundinte} and the Bootstrap setting~\eqref{est:Boot}, we obtain
\begin{equation*}
     \left(\Lambda_{1}Q_{1},\rho_{2}+2m_{0}\right)=(\Lambda_{1}Q_{1},\rho_{3}+2m_{0})
            =(\Lambda_{2}Q_{2},\rho_{3}+2m_{0})=O\left(t^{-1}\right),
\end{equation*}
which directly completes the proof of the estimates in the second line.

\smallskip
Last, using again~\eqref{est:LambdaQk} and~\eqref{equ:inteLambdaQ}, for any $k\in [\![1,3]\!]$, we compute 
\begin{equation*}
\begin{aligned}
    \left(\Lambda_{k}Q_{k},\rho_{k}\right)
    &=\int_{\RR}\Lambda Q(y)\int_{y}^{\infty}\Lambda Q(\sigma)\dd \sigma \dd y+O\left(t^{-1}\right)\\
    &=\frac{1}{2}\left(\int_{\RR}\Lambda Q(y)\dd y\right)^{2}+O\left(t^{-1}\right)=2m_{0}^{2}+O\left(t^{-1}\right),
    \end{aligned}
\end{equation*}
which directly completes the proof of the estimates in the third line.

\smallskip
Proof of (iii). Note that, from the bootstrap setting~\eqref{est:Boot}, for any $y\in (0,\infty)$, 
\begin{equation*}
    \sum_{k=1}^{3}|\rho_{k}|\phi_{2}\lesssim \sum_{k=1}^{3}e^{-\frac{1}{2}|y-x_{k}|}\lesssim e^{-\frac{1}{2}|y|}.
\end{equation*}
It follows directly from~\eqref{equ:defZ} that 
\begin{equation*}
\sum_{k_{1}=1}^{3}\sum_{k_{2}=1}^{3}\int_{0}^{\infty}|Z_{k_{1}}||\rho_{k_{2}}|\phi_{2}\dd y\lesssim \int_{0}^{\infty}e^{-\frac{1}{2}|y|}\dd y\lesssim 1.
\end{equation*}
Note also that, for $y\in (-20\log t,0)$, we find 
\begin{equation*}
    \sum_{k=1}^{3}|\rho_{k}|\phi_{2}\lesssim 1\Longrightarrow
    \sum_{k_{1}=1}^{3}\sum_{k_{2}=1}^{3}\int_{-20\log t}^{0}|Z_{k_{1}}||\rho_{k_{2}}|\phi_{2}\dd y\lesssim \log t.
\end{equation*}
Last, from Corollary~\ref{coro:X}, for $y\in (-\infty,-20\log t)$, we have 
\begin{equation*}
    \sum_{k=1}^{3}|X_{k}|\lesssim e^{-\frac{1}{2}|y-x_{k}|}\Longrightarrow
    \sum_{k_{1}=1}^{3}\sum_{k_{2}=1}^{3}\int^{-20\log t}_{-\infty}|Z_{k_{1}}||\rho_{k_{2}}|\phi_{2}\dd y\lesssim 1.
\end{equation*}
Combining the above estimates, we complete the proof of the estimates in (iii).

\smallskip
Proof of (iv). Using a similar argument as in the proof of (iii), we find 
\begin{equation*}
\begin{aligned}
    \sum_{k_{1}=1}^{3}\sum_{k_{2}=1}^{3}\int_{\RR}|X_{k_{1}}||\rho_{k_{2}}|\phi \phi_{2}\dd y&\lesssim \log t,\\
    \sum_{k_{1}=1}^{3}\sum_{k_{2}=1}^{3}\int_{\RR}|X_{k_{1}}||\rho_{k_{2}}||\partial_{t}\phi-\partial_{y}\phi|\phi_{2}\dd y&\lesssim t^{-1}\log t.
    \end{aligned}
\end{equation*}
Therefore, from the bootstrap setting~\eqref{est:Boot}, the expression of $(\Phi_{1},\Phi_{2},\Phi_{3})$ in Proposition~\ref{prop:appPhi} and Lemma~\ref{le:para1}, we complete the proof of the estimates in (iv).
\end{proof}

\begin{corollary}\label{coro:Phirho}
    It holds 
    \begin{equation*}
    \begin{aligned}
        \left(\Phi(S),\rho_{1}\phi_{2}\right)=m_{0}^{2}(\dot{\nu}_{1}-2b_{1})+O\left(\frac{C_{0}}{t^{\frac{11}{4}}}\right),\\
         \left(\Phi(S),(2\rho_{1}-\rho_{2})\phi_{2}\right)=m_{0}^{2}(\dot{\nu}_{2}-2b_{2})+O\left(\frac{C_{0}}{t^{\frac{11}{4}}}\right),\\
          \left(\Phi(S),\left(-2\rho_{1}+2\rho_{2}-\rho_{3}\right)\phi_{2}\right)=m_{0}^{2}(\dot{\nu}_{3}-2b_{3})+O\left(\frac{C_{0}}{t^{\frac{11}{4}}}\right).
        \end{aligned}
    \end{equation*}
\end{corollary}

\begin{proof}
    First, from~\eqref{equ:defMod}--\eqref{equ:defMQ},~\eqref{est:Boot} and (i)--(iii) of Lemma~\ref{le:estrho1}, we find
    \begin{equation*}
    \begin{aligned}
    &\sum_{k=1}^{3}\left({\vec{\rm{Mod}}}_{k}\cdot\vec{M}_{k}Q,\rho_{1}\phi_{2}\right)\\
    &=m_{0}^{2}\left(\dot{\nu}_{1}-2b_{1}\right)+O\bigg(t^{-1}\sum_{k=1}^{3}|\dot{x}_{k}-\nu_{k}|+t^{-1}\sum_{k=1}^{3}|\dot{\nu}_{k}-2b_{k}|\bigg)
        \\
        &+O\left(\log t\left(|\dot{b}_{1}+\alpha e^{-r_{1}}|
        +|\dot{b}_{2}+\alpha e^{-r_{1}}-\alpha e^{-r_{2}}|
        +|\dot{b}_{3}+3\alpha e^{-r_{2}}|\right)\right).
        \end{aligned}
    \end{equation*}
    It follows directly from Lemma~\ref{le:para1} that 
    \begin{equation*}
        \sum_{k=1}^{3}\left({\vec{\rm{Mod}}}_{k}\cdot\vec{M}_{k}Q,\rho_{1}\phi_{2}\right)
        =m_{0}^{2}\left(\dot{\nu}_{1}-2b_{1}\right)+O\left(\frac{C_{0}}{t^{\frac{11}{4}}}\right).
    \end{equation*}
    Then, for any $f\in \mathcal{S}$, from the bootstrap setting~\eqref{est:Boot} and Lemma~\ref{le:para1}, we find\footnote{We refer to Definition~\ref{def:Sfunction} for the definition of $\mathcal{S}$.}
    \begin{equation*}
        \left|(f,\rho_{1}\phi_{2})\right|\lesssim \|f\|_{L^{2}}\|\rho_{1}\phi_{2}\|_{L^{2}}\lesssim \frac{C_{0}}{t^{\frac{11}{4}}}+\frac{C_{0}^{2}}{t^{3}}.
    \end{equation*}
    Combining the above estimates with (iv) of Lemma~\ref{le:estrho1}, we complete the proof of the estimate in the first line of Corollary~\ref{coro:Phirho}.

    \smallskip
    Second, using again~\eqref{equ:defMod}--\eqref{equ:defMQ},~\eqref{est:Boot}, Lemma~\ref{le:para1} and Lemma~\ref{le:estrho1}, we obtain
    \begin{equation*}
    \begin{aligned}
    \sum_{k=1}^{3}\left({\vec{\rm{Mod}}}_{k}\cdot\vec{M}_{k}Q,\left(2\rho_{1}-\rho_{2}\right)\phi_{2}\right)&=m_{0}^{2}(\dot{\nu}_{2}-2b_{2})+O\left(\frac{C_{0}}{t^{\frac{11}{4}}}\right),\\
     \sum_{k=1}^{3}\left({\vec{\rm{Mod}}}_{k}\cdot\vec{M}_{k}Q,\left(-2\rho_{1}+2\rho_{2}-\rho_{3}\right)\phi_{2}\right)&=m_{0}^{2}(\dot{\nu}_{3}-2b_{3})+O\left(\frac{C_{0}}{t^{\frac{11}{4}}}\right).
    \end{aligned}
    \end{equation*}
    Therefore, using a similar argument to that in the above argument, we complete the proof for the remaining estimate in Corollary~\ref{coro:Phirho}.
\end{proof}

We are in a position to complete the proof of Lemma~\ref{le:refinednu}.

\begin{proof}
    [Proof of Lemma~\ref{le:refinednu}]
    First, from the definition of $\mathcal{F}_{1}$ in~\eqref{equ:defF}, we compute 
    \begin{equation*}
    \begin{aligned}
        {\dot{\mathcal{F}}}_{1}
        &=\frac{\dot{x}_{1}}{m_{0}^{2}}\left(\varepsilon(t),\left(\Lambda Q(\cdot-x_{1})\right)\phi_{2}\right)\\
        &+\frac{1}{m_{0}^{2}}\left(\partial_{t}\varepsilon(t),\rho_{1}\phi_{2}\right)+\frac{1}{m_{0}^{2}}\left(\varepsilon(t),\rho_{1}\partial_{t}\phi_{2}\right).
        \end{aligned}
    \end{equation*}
    Using the bootstrap setting~\eqref{est:Boot},~\eqref{est:Naloc} and Lemma~\ref{le:para1},
    \begin{equation}\label{est:F11}
        \left|\dot{x}_{1}\left(\varepsilon(t),\left(\Lambda Q(\cdot-x_{1})\right)\phi_{2}\right)\right|\lesssim
        |\dot{x}_{1}|\mathcal{N}_{A}(\varepsilon)\lesssim \frac{C_{0}}{t^{\frac{11}{4}}}.
    \end{equation}
    Then, using again the definition of $\phi_{2}$ in~\eqref{equ:defphi2} and the bootstrap setting~\eqref{est:Boot},
    \begin{equation}\label{est:F12}
    \begin{aligned}
    \left|\left(\varepsilon(t),\rho_{1}\partial_{t}\phi_{2}\right)\right|
    \lesssim t^{-1}\int_{-2t}^{-t}|\varepsilon(t)|\dd y\lesssim t^{-\frac{1}{2}}\mathcal{N}_{A}(\varepsilon)\lesssim \frac{C_{0}}{t^{\frac{9}{4}}}.
        \end{aligned}
    \end{equation}
    On the other hand, from Lemma~\ref{le:eque}, we decompose
    \begin{equation*}
    \begin{aligned}
        \frac{1}{m_{0}^{2}}\left(\partial_{t}\varepsilon(t),\rho_{1}\phi_{2}\right)
        &=\frac{1}{m_{0}^{2}}\left((S+\varepsilon)^{5}-S^{5}-5Q_{1}^{4}\varepsilon,\partial_{y}(\rho_{1}\phi_{2})\right)\\
        &+\frac{1}{m_{0}^{2}}\left(\partial_{y}^{2}\varepsilon-\varepsilon+5Q_{1}^{4}\varepsilon,\partial_{y}(\rho_{1}\phi_{2})\right)
        -\frac{1}{m_{0}^{2}}\left(\Phi(S),\rho_{1}\phi_{2}\right).
        \end{aligned}
    \end{equation*}
    Based on the definition of $\rho_{1}$ and $\phi_{2}$ in~\eqref{equ:defrhok}--\eqref{equ:defphi2}, we find 
    \begin{equation}\label{est:pyrho1phi21}
    \begin{aligned}
        \partial_{y}(\rho_{1}\phi_{2})&=-(\Lambda Q)(y-x_{1})
        +\rho_{1}\partial_{y}\phi_{2}
        \\
       & -
        \left((\Lambda Q)(y-x_{1})\right)\left(\phi_{2}-1\right),
        \end{aligned}
    \end{equation}
    which directly implies that 
    \begin{equation}\label{est:pyrho1phi2}
    \begin{aligned}
        \partial_{y}(\rho_{1}\phi_{2})&=-(\Lambda Q)(y-x_{1})+O\left(t^{-1}\textbf{1}_{(-2t,-t)}(y)\right)\\
        &+O\left(|(\Lambda Q)(y-x_{1})|\textbf{1}_{(-\infty,-t)}(y)\right).
        \end{aligned}
    \end{equation}
    Note that, from the definition of $S$ in~\eqref{equ:defS}, 
    \begin{equation}\label{est:Se5}
    \begin{aligned}
        \left|(S+\varepsilon)^{5}-S^{5}-5Q_{1}^{4}\varepsilon\right|
&\lesssim\left(|U|^{3}+|V\phi|^{3}+|W|^{3}\right)\varepsilon^{2}+|\varepsilon|^{5}\\
&+
\left(|Q_{2}|+|Q_{3}|\right)\left(|Q_{1}|^{3}+|Q_{2}|^{3}+|Q_{3}|^{3}\right)|\varepsilon|.
\end{aligned}
    \end{equation}
    Next, using Lemma~\ref{le:boundinte} and the bootstrap setting~\eqref{est:Boot}, we find 
    \begin{equation*}
    \left\|Q_{2}\left((\Lambda Q)(\cdot-x_{1})\right)\right\|_{L^{2}}+
    \left\|Q_{3}\left((\Lambda Q)(\cdot-x_{1})\right)\right\|_{L^{2}}\lesssim t^{-2}.
    \end{equation*}
    It follows from~\eqref{est:pyrho1phi2}--\eqref{est:Se5} and Lemma~\ref{le:L2e} that 
    \begin{equation}\label{est:F13}
    \begin{aligned}
      &\left|\left((S+\varepsilon)^{5}-S^{5}-5Q_{1}^{4}\varepsilon,\partial_{y}(\rho_{1}\phi_{2})\right)\right|\\
      &\lesssim\left(\|Q_{2}\left((\Lambda Q)(\cdot-x_{1})\right\|_{L^{2}}+\|Q_{3}\left((\Lambda Q)(\cdot-x_{1})\right\|_{L^{2}}\right)\mathcal{N}_{A}(\varepsilon)\\
      &+\left(\mathcal{N}_{A}(\varepsilon)+t^{-1}\right)\mathcal{N}_{A}(\varepsilon)+\left(\|\varepsilon\|_{L^{\infty}}^{3}+t^{-1}\right)\|\varepsilon\|_{L^{2}}^{2}\lesssim \frac{C_{0}}{t^{\frac{11}{4}}}+\frac{1}{t^{3}}.
      \end{aligned}
    \end{equation}
    Then, from~\eqref{est:pyrho1phi21} and (ii) of Proposition~\ref{prop:L}, we compute 
    \begin{equation*}
    \begin{aligned}
        &\partial_{y}^{2}\left(\partial_{y}(\rho_{1}\phi_{2})\right)-\partial_{y}(\rho_{1}\phi_{2})+5Q_{1}^{4}\partial_{y}\left(\rho_{1}\phi_{2}\right)\\
        &=-2Q_{1}+O\left(|\nu_{1}|e^{-\frac{1}{2}|y-x_{1}|}+e^{-\frac{1}{2}|y-x_{1}|}\textbf{1}_{(-\infty,-t)}(y)+\textbf{1}_{[-2t,-t]}(y)\right).
        \end{aligned}
    \end{equation*}
    From the above estimate, integration by parts, the bootstrap setting~\eqref{est:Boot} and the orthogonality condition $(\varepsilon,Q_{1})=0$, we deduce that 
    \begin{equation*}
    \left|\left(\partial_{y}^{2}\varepsilon-\varepsilon+5Q_{1}^{4}\varepsilon,\partial_{y}(\rho_{1}\phi_{2})\right)\right|\lesssim t^{-\frac{1}{2}}\mathcal{N}_{A}(\varepsilon)\lesssim \frac{C_{0}}{t^{\frac{9}{4}}}.
    \end{equation*}
    Combining the above estimate with~\eqref{est:F11}--\eqref{est:F12},~\eqref{est:F13} and Corollary~\ref{coro:Phirho}, we compete the proof of the first estimate in Lemma~\ref{le:refinednu}.

    \smallskip
    Second, from the definition of $\mathcal{F}_{2}$ in~\eqref{equ:defF}, we compute 
    \begin{equation}\label{equ:dotF2}
    \begin{aligned}
        \dot{\mathcal{F}}_{2}= &\frac{2\dot{x}_{1}}{m_{0}^{2}}\left(\varepsilon(t),\left(\Lambda Q(\cdot-x_{1})\right)\phi_{2}\right)
        -\frac{\dot{x}_{2}}{m_{0}^{2}}\left(\varepsilon(t),\left(\Lambda Q(\cdot-x_{2})\right)\phi_{2}\right)
        \\
        &+\frac{1}{m_{0}^{2}}\left(\partial_{t}\varepsilon(t),\left(2\rho_{1}-\rho_{2}\right)\phi_{2}\right)+\frac{1}{m_{0}^{2}}\left(\varepsilon(t),\left(2\rho_{1}-\rho_{2}\right)\partial_{t}\phi_{2}\right).
        \end{aligned}
    \end{equation}
    Based on an argument similar to the above argument, we deduce that 
    \begin{equation*}
           \left| \dot{x}_{1}\left(\varepsilon(t),\left(\Lambda Q(\cdot-x_{1})\right)\phi_{2}\right)\right|
        +\left|{\dot{x}_{2}}\left(\varepsilon(t),\left(\Lambda Q(\cdot-x_{2})\right)\phi_{2}\right)\right|\lesssim \frac{C_{0}}{t^{\frac{11}{4}}}.
    \end{equation*}
    In addition, from the definition of $\phi_{2}$ in~\eqref{equ:defphi2} and the bootstrap setting~\eqref{est:Boot},
    \begin{equation*}
          \left|\left(\varepsilon(t),\left(2\rho_{1}-\rho_{2}\right)\partial_{t}\phi_{2}\right)\right|
    \lesssim t^{-1}\int_{-2t}^{-t}|\varepsilon(t)|\dd y\lesssim t^{-\frac{1}{2}}\mathcal{N}_{A}(\varepsilon)\lesssim \frac{C_{0}}{t^{\frac{9}{4}}}.
    \end{equation*}
    On the other hand, using again Lemma~\ref{le:boundinte}, Lemma~\ref{le:eque}, Lemma~\ref{le:L2e}, (ii) of Proposition~\ref{prop:L} and the bootstrap setting~\eqref{est:Boot}, we find 
    \begin{equation*}
         \frac{1}{m_{0}^{2}}\left(\partial_{t}\varepsilon(t),\left(2\rho_{1}-\rho_{2}\right)\phi_{2}\right)
        =-\frac{1}{m_{0}^{2}}\left(\Phi(S),\left(2\rho_{1}-\rho_{2}\right)\phi_{2}\right)
        +O\left(\frac{C_{0}}{t^{\frac{9}{4}}}+\frac{1}{t^{3}}\right).
    \end{equation*}
    It follows directly from Corollary~\ref{coro:Phirho} that 
    \begin{equation*}
         \frac{1}{m_{0}^{2}}\left(\partial_{t}\varepsilon(t),\left(2\rho_{1}-\rho_{2}\right)\phi_{2}\right)=-\left(\dot{\nu}_{2}-2b_{2}\right)+O\left(\frac{C_{0}}{t^{\frac{9}{4}}}+\frac{1}{t^{3}}\right).
    \end{equation*}
    Combining the above estimates with~\eqref{equ:dotF2}, we complete the proof of the second estimate in Lemma~\ref{le:refinednu}. Since the proof of the third estimate is similar to that above, we omit it. The proof of Lemma~\ref{le:refinednu} is complete.
\end{proof}

\subsection{Energy functional}\label{SS:Energy}
Consider the nonlinear energy functional for $\varepsilon$\footnote{See~\eqref{equ:defS} and \eqref{equ:defchi1}-\eqref{equ:defphi1} for the definition of $S$ and $\phi_{1}$.}:
\begin{equation*}
    \mathcal{E}(t,\varepsilon)=\int_{\RR}\left((\partial_{y}\varepsilon)^{2}+\varepsilon^{2}\phi_{1}
    -\frac{1}{3}\left(\left(S+\varepsilon\right)^{6}-S^{6}-6S^{5}\varepsilon\right)
    \right)\dd y.
\end{equation*}
Let $\chi_{3}:\RR\to [0,1]$ be a $C^{\infty}$ function such that 
\begin{equation}\label{est:defchi3}
    {\chi_{3}}_{|\left(-\infty,-1\right)}\equiv {\chi_{3}}_{|\left(1,\infty\right)}\equiv 0\quad \mbox{and}\quad 
    {\chi_{3}}_{|\left(-\frac{1}{2},\frac{1}{2}\right)}\equiv 1.
\end{equation}
Then, we consider the following suitable cut-off functions: 
\begin{equation}\label{equ:defphik}
    \phi_{k+2}(t,y)=\chi_{3}\left(\frac{y-x_{k}(t)}{\log t}\right),\quad \mbox{for any}\ k\in [\![1,3]\!].
\end{equation}
Note that, from~\eqref{equ:defphi1} and the above definition of $\phi_{k}$, we have 
\begin{equation}\label{est:phi345}
    0\lesssim \min_{k\in \in [\![3,5]\!]} \phi_{k}\lesssim
    \max_{k\in \in [\![3,5]\!]} \phi_{k}\lesssim \phi_{1}.
\end{equation}
We now introduce the refined term related to the localized $L^{2}$ norm of $\varepsilon$:
\begin{equation*}
    \mathcal{P}(t,\varepsilon)=\nu_{1}\int_{\RR}\varepsilon^{2}\phi_{3}\dd y+\nu_{2}\int_{\RR}\varepsilon^{2}\phi_{4}\dd y+\nu_{3}\int_{\RR}\varepsilon^{2}\phi_{5}\dd y.
\end{equation*}
Last, we define the following functional $\mathcal{K}$ which is a combination of
$\mathcal{E}$ and $\mathcal{P}$:
\begin{equation*}
    \mathcal{K}(t,\varepsilon)=\mathcal{E}(t,\varepsilon)+\mathcal{P}(t,\varepsilon).
\end{equation*}

We mention here that, the functional $\mathcal{K}$ is coercive in $\varepsilon$ at leading order and is an almost monotonicity quantity for this problem.

\begin{proposition}
    The following estimates hold on $[T^{*},T_{n}]$.
    \begin{enumerate}
        \item \emph{Coercivity of $\mathcal{K}$.} We have 
        \begin{equation}\label{est:coerK}
            \mathcal{N}^{2}_{A}(\varepsilon)\lesssim \mathcal{K}+\frac{C_{0}}{t^{4}}.
        \end{equation}

        \item \emph{Time control of $\mathcal{K}$.} We have 
        \begin{equation}\label{est:dtK}
            -\frac{\dd \mathcal{K}}{\dd t}\lesssim 
            \frac{C_{0}A^{\frac{1}{2}}}{t^{\frac{9}{2}}}+\frac{C_{0}^{2}}{t^{\frac{9}{2}}\log t}+
         \frac{C_{0}^{3}A^{\frac{1}{2}}}{t^{\frac{23}{5}}}.
        \end{equation}
    \end{enumerate}
\end{proposition}

\begin{proof}
We mention here that, in what follows, we will use the bootstrap setting~\eqref{est:Boot} frequently without further mention. Moreover, we tacitly take the initial time $T_{0}\gg 1$ large enough which may depend on the constants $C_{0}$ and $A$.

\smallskip
    Proof of (i). The coercivity of $\mathcal{K}$ follows as a standard consequence of (iv) of Proposition~\ref{prop:L} with the orthogonality condition~\eqref{equ:ortho} and an elementary localization argument. Hence, we only sketch the argument and refer the reader to~\cite[Proposition 3.10]{LYPRINT} for more details.

    \smallskip
    First, from Lemma~\ref{le:L2e}, we decompose 
    \begin{equation*}
    \begin{aligned}
        \mathcal{E}
        &=\int_{\RR}\left(\partial_{y}\varepsilon\right)^{2}(1-\phi_{1})\dd y-\frac{1}{2}
        \int_{\RR}\varepsilon^{2}\left(
        \frac{\left(\partial_{y}\phi_{1}\right)^{2}}{2\phi_{1}}-\partial_{y}^{2}\phi_{1}
        \right)\dd y\\
        &+\int_{\RR}\left(
        (\partial_{y}(\varepsilon\sqrt{\phi_{1}}))^{2}
        +(\varepsilon\sqrt{\phi_{1}})^{2}
        -5U^{4}(\varepsilon\sqrt{\phi_{1}})^{2}
        \right)\dd y+O\left(t^{-4}\right).
        \end{aligned}
    \end{equation*}
    Using~\eqref{equ:ortho} and (iv) of Proposition~\ref{prop:L}, we obtain 
    \begin{equation*}
       \int_{\RR}\left(
        (\partial_{y}(\varepsilon\sqrt{\phi_{1}}))^{2}
        +(\varepsilon\sqrt{\phi_{1}})^{2}
        -5U^{4}(\varepsilon\sqrt{\phi_{1}})^{2}
        \right)\dd y\ge \mu\|\varepsilon\sqrt{\phi_{1}}\|_{H^{1}}^{2}+O\left(t^{-4}\right),
    \end{equation*}
    which directly implies that 
    \begin{equation*}
        \mathcal{N}_{A}^{2}(\varepsilon)\lesssim \|\varepsilon\sqrt{\phi_{1}}\|_{H^{1}}^{2}+\int_{\RR}\left(\partial_{y}\varepsilon\right)^{2}(1-\phi_{1})\dd y\lesssim \mathcal{E}+O\left(t^{-4}\right),\ \ \mbox{for} \ A\gg 1.
    \end{equation*}
    Here, we use the fact that 
    \begin{equation*}
        A|\partial_{y}\phi_{1}|+ A^{2}|\partial^{2}_{y}\phi_{1}|\lesssim \phi_{1}
        \Longrightarrow
       \Big|  \int_{\RR}\varepsilon^{2}\left(
        \frac{\left(\partial_{y}\phi_{1}\right)^{2}}{2\phi_{1}}-\partial_{y}^{2}\phi_{1}
        \right)\dd y\Big|\lesssim A^{-1}\int_{\RR}\varepsilon^{2}\phi_{1}\dd y.
    \end{equation*}
    Second, using the pointwise estimate~\eqref{est:phi345}, we directly have 
    \begin{equation*}
        |\mathcal{P}|\lesssim 
        \sum_{k=1}^{3}|\nu_{k}|\int_{\RR}\varepsilon^{2}\phi_{1}\dd y
        \lesssim
        \sum_{k=1}^{3}|\nu_{k}|\mathcal{N}_{A}^{2}(\varepsilon)\lesssim \frac{C_{0}^{2}}{t^{\frac{9}{2}}}.
    \end{equation*}
    Combining the above estimates, we complete the proof of the estimate~\eqref{est:coerK}.

    \smallskip
    Proof of (ii). \textbf{Step 1.} Time control of $\mathcal{E}$. We claim that 
    \begin{equation}\label{est:dtE}
    \begin{aligned}
    \frac{\dd \mathcal{E}}{\dd t}&\ge 
    20\sum_{k=1}^{3}\dot{x}_{k}\int_{\RR}\left(\partial_{y}Q_{k}\right)
            \left(Q_{k}^{3}\varepsilon^{2}\right)\dd y\\
           & -\frac{1}{4}\int_{\RR}\varepsilon^{2}\partial_{y}\phi_{1}\dd y
        +O\left(\frac{C_{0}A^{\frac{1}{2}}}{t^{\frac{9}{2}}}+
         \frac{C_{0}^{3}A^{\frac{1}{2}}}{t^{\frac{19}{4}}}\right).
         \end{aligned}
    \end{equation}
    Indeed, from Lemma~\ref{le:eque} and integration by parts, we decompose
    \begin{equation*}
        \frac{\dd \mathcal{E}}{\dd t}=\mathcal{J}_{1}+\mathcal{J}_{2}+\mathcal{J}_{3},
    \end{equation*}
    where
    \begin{equation*}
        \begin{aligned}
            \mathcal{J}_{1}&=
            2\int_{\RR}\Phi(S)\left(\partial_{y}^{2}\varepsilon-\varepsilon\phi_{1}+(S+\varepsilon)^{5}-S^{5}\right)\dd y,
            \\
            \mathcal{J}_{2}&=2\int_{\RR}\varepsilon(1-\phi_{1})\partial_{y}\left(\partial_{y}^{2}\varepsilon-\varepsilon+(S+\varepsilon)^{5}-S^{5}\right)\dd y,\\
             \mathcal{J}_{3}&=-2\int_{\RR}(\left(\partial_{t}S\right)\left(
            (S+\varepsilon)^{5}-S^{5}-5S^{4}\varepsilon
            \right)\dd y+\int_{\RR}\varepsilon^{2}\partial_{t}\phi_{1}\dd y.
        \end{aligned}
    \end{equation*}

    \emph{Estimate on $\mathcal{J}_{1}$.} We claim that 
    \begin{equation}\label{est:J1}
        \mathcal{J}_{1}=O\left(\frac{C_{0}A^{\frac{1}{2}}}{t^{\frac{9}{2}}}
        +\frac{C_{0}^{3}A^{\frac{1}{2}}}{t^{\frac{19}{4}}}
        +\frac{C_{0}^{2}}{t^{5}}\right).
    \end{equation}
    Indeed, we first rewrite 
    \begin{equation*}
    \begin{aligned}
        \partial_{y}^{2}\varepsilon-\varepsilon\phi_{1}+(S+\varepsilon)^{5}-S^{5}&=\partial_{y}^{2}\varepsilon-\varepsilon+5(Q_{1}^{4}+Q_{2}^{4}+Q_{3}^{4})\varepsilon\\
        &+10S^{3}\varepsilon^{2}+10S^{2}\varepsilon^{3}+5S\varepsilon^{4}+\varepsilon^{5}\\
          &+5(S^{4}-Q_{1}^{4}-Q_{2}^{4}-Q_{3}^{4})\varepsilon+\varepsilon(1-\phi_{1}).
        \end{aligned}
    \end{equation*}
    It follows from $(\varepsilon,\partial_{y}Q_{k})=0$,
    ${\rm{Ker}}\mathcal{L}={\rm{Span}}\left\{Q'\right\}$ and Lemma~\ref{le:boundinte} that 
    \begin{equation*}
    \begin{aligned}
       &\sum_{k=1}^{3}\left|\left(\partial_{y}Q_{k},\partial_{y}^{2}\varepsilon-\varepsilon\phi_{1}+(S+\varepsilon)^{5}-S^{5}\right)\right|\\
       &\lesssim \left(t^{-2}+t^{-3}\right)\mathcal{N}_{A}(\varepsilon)+\sum_{k=1}^{3}|b_{k}|(1+|\nu_{k}|)\mathcal{N}_{A}(\varepsilon)\\
       &+\mathcal{N}_{A}^{2}(\varepsilon)\left(1+
       \|\varepsilon\|_{L^{\infty}}+
       \|\varepsilon\|_{L^{\infty}}^{3}\right)
       \lesssim \frac{C_{0}^{2}}{t^{\frac{7}{2}}}+\frac{C_{0}}{t^{\frac{15}{4}}}.
        \end{aligned}
    \end{equation*}
    Based on the above estimate and (i) of Lemma~\ref{le:para1}, we find
    \begin{equation}\label{est:energyMOD1}
        \sum_{k=1}^{3}\left|(\dot{x}_{k}-\nu_{k})\left(\partial_{y}Q_{k},\partial_{y}^{2}\varepsilon-\varepsilon\phi_{1}+(S+\varepsilon)^{5}-S^{5}\right)\right|\lesssim
        \frac{C_{0}^{3}}{t^{\frac{21}{4}}}.
    \end{equation}
    Similarly, using $(\varepsilon,Q_{k})=(\varepsilon,\Lambda_{k}Q_{k})=0$, $\mathcal{L}\Lambda Q=-2Q$ and Lemma~\ref{le:boundinte}, we find 
    \begin{equation*}
        \begin{aligned}
           \sum_{k=1}^{3}\left|\left(\Lambda_{k}Q_{k},\partial_{y}^{2}\varepsilon-\varepsilon\phi_{1}+(S+\varepsilon)^{5}-S^{5}\right)\right|\lesssim\frac{C_{0}^{2}}{t^{\frac{7}{2}}}+\frac{C_{0}}{t^{\frac{15}{4}}},
        \end{aligned}
    \end{equation*}
    which implies that
    \begin{equation}\label{est:energyMOD2}
        \sum_{k=1}^{3}\left|(\dot{\nu}_{k}-2\nu_{k})\left(\Lambda_{k}Q_{k},\partial_{y}^{2}\varepsilon-\varepsilon\phi_{1}+(S+\varepsilon)^{5}-S^{5}\right)\right|\lesssim
        \frac{C_{0}^{3}}{t^{\frac{21}{4}}}.
    \end{equation}
    Then, from integration by parts and Lemma~\ref{le:L2e}, we have
    \begin{equation*}
    \begin{aligned}
        &\sum_{k=1}^{3}\left|\left(Z_{k},\partial_{y}^{2}\varepsilon-\varepsilon\phi_{1}+(S+\varepsilon)^{5}-S^{5}\right)\right|\\
        &\lesssim \sum_{k=1}^{3}\left(1+\|Z_{k}\|_{\dot{H}^{1}}+\|Z_{k}\sqrt{\phi_{1}}\|_{L^{2}}\right)\mathcal{N}_{A}(\varepsilon)+\|\varepsilon\|_{L^{2}}^{2}\|\varepsilon\|_{L^{\infty}}^{3}\\
        &+\|\varepsilon\|_{L^{2}}\left(\|V\phi\|_{L^{\infty}}^{3}\|V\phi\|_{L^{2}}
        +\|W\|_{L^{\infty}}^{3}\|W\|_{L^{2}}
        \right)\lesssim \frac{C_{0}A^{\frac{1}{2}}}{t^{\frac{3}{2}}}.
        \end{aligned}
    \end{equation*}
    It follows from (ii) of Lemma~\ref{le:para1} that 
    \begin{equation*}
        \begin{aligned}
            \left| \left(\dot{b}_{1}+\alpha e^{-r_{1}}\right)
             \left(Z_{1},\partial_{y}^{2}\varepsilon-\varepsilon\phi_{1}+(S+\varepsilon)^{5}-S^{5}\right)\right|\lesssim \frac{C_{0}^{3}A^{\frac{1}{2}}}{t^{\frac{19}{4}}},\\
             \left| \left(\dot{b}_{3}+3\alpha e^{-r_{2}}\right)
             \left(Z_{3},\partial_{y}^{2}\varepsilon-\varepsilon\phi_{1}+(S+\varepsilon)^{5}-S^{5}\right)\right|\lesssim  \frac{C_{0}^{3}A^{\frac{1}{2}}}{t^{\frac{19}{4}}}
             ,\\
             \left| \left(\dot{b}_{2}+\alpha e^{-r_{1}}-\alpha e^{-r_{2}}\right)
             \left(Z_{2},\partial_{y}^{2}\varepsilon-\varepsilon\phi_{1}+(S+\varepsilon)^{5}-S^{5}\right)\right|\lesssim  \frac{C_{0}^{3}A^{\frac{1}{2}}}{t^{\frac{19}{4}}}.
        \end{aligned}
    \end{equation*}
    Combining the above estimate with~\eqref{est:energyMOD1} and~\eqref{est:energyMOD2}, we obtain 
    \begin{equation}\label{est:energyMOD}
\sum_{k=1}^{3}\left|\int_{\RR}\left({\vec{\rm{Mod}}}_{k}\cdot\vec{M}_{k}Q\right)\left(
      \partial_{y}^{2}\varepsilon-\varepsilon\phi_{1}+(S+\varepsilon)^{5}-S^{5}
      \right)\dd y\right|\lesssim \frac{C_{0}^{3}A^{\frac{1}{2}}}{t^{\frac{19}{4}}}.
    \end{equation}
    Then, from integration by parts, Lemma~\ref{le:L2e} and Lemma~\ref{le:Phiweight}, 
    \begin{equation}\label{est:energyPhikS}
    \begin{aligned}
        &\sum_{k=1}^{3}\left|\int_{\RR}\Phi_{k}(S)\left(\partial_{y}^{2}\varepsilon-\varepsilon\phi_{1}+(S+\varepsilon)^{5}-S^{5}\right)\dd y\right|\\
        &\lesssim \sum_{k=1}^{3}\left(\|\Phi_{k}(S)\|_{\dot{H}^{1}}+
        \left(\int_{\RR}|\Phi_{k}(S)|^{2}\phi_{1}\dd y\right)^{\frac{1}{2}}\right)\mathcal{N}_{A}(\varepsilon)\\
        &+\sum_{k=1}^{3}\|\varepsilon\|_{L^{2}}\|\Phi_{k}(S)\|_{L^{2}}\left(\|V\phi\|_{L^{\infty}}^{4}+\|W\|_{L^{\infty}}^{4}+\|\varepsilon\|_{L^{\infty}}^{4}\right)\lesssim \frac{C_{0}A^{\frac{1}{2}}}{t^{\frac{9}{2}}}+\frac{1}{t^{5}}.
        \end{aligned}
    \end{equation}
    On the other hand, using again Definition~\ref{def:Sfunction}, Lemma~\ref{le:L2e} and Lemma~\ref{le:Phiweight}, for any real-valued function $f\in \mathcal{S}$, we have 
    \begin{equation*}
    \begin{aligned}
        &\left|\int_{\RR}f
        \left(\partial_{y}^{2}\varepsilon-\varepsilon\phi_{1}+(S+\varepsilon)^{5}-S^{5}\right)\dd y\right|\\
         &\lesssim \|f\|_{H^{1}}\mathcal{N}_{A}(\varepsilon)+\|\varepsilon\|_{L^{2}}\|f\|_{L^{2}}\|V\phi\|_{L^{\infty}}^{4}\\
         &+\|f\|_{L^{2}}\|\varepsilon\|_{L^{2}}\left(\|W\|_{L^{\infty}}^{4}+\|\varepsilon\|_{L^{\infty}}^{4}\right)
         \lesssim \frac{C^{2}_{0}}{t^{5}}.
        \end{aligned}
    \end{equation*}
    We see that~\eqref{est:J1} follows from the above estimate and~\eqref{est:energyMOD}--\eqref{est:energyPhikS}.

    \smallskip
    \emph{Estimate on $\mathcal{J}_{2}$.} We claim that 
    \begin{equation}\label{est:J2}
        \mathcal{J}_{2}\ge -\frac{1}{2}\int_{\RR}\varepsilon^{2}\partial_{y}\phi_{1}\dd y
        +O\left(\frac{C_{0}^{2}}{t^{6}}\right).
    \end{equation}
    Indeed, from integration by parts, we find 
    \begin{equation*}
    \begin{aligned}
        \mathcal{J}_{2}&=-3\int_{\RR}(\partial_{y}\varepsilon)^{2}\partial_{y}\phi_{1}\dd y+\int_{\RR}\varepsilon^{2}\partial_{y}^{3}\phi_{1}\dd y\\
        &-2\int_{\RR}(\partial_{y}\varepsilon)(1-\phi_{1})\left((S+\varepsilon)^{5}-S^{5}\right)\dd y\\
        &+2\int_{\RR}\varepsilon\left(\partial_{y}\phi_{1}\right)\left((S+\varepsilon)^{5}-S^{5}\right)\dd y-\int_{\RR}\varepsilon^{2}\partial_{y}\phi_{1}\dd y.
        \end{aligned}
    \end{equation*}
    First, from the definition of $\phi_{1}$ in~\eqref{equ:defphi1}, we check that 
    \begin{equation*}
        \left|\partial_{y}^{3}\phi_{1}\right|\lesssim A^{-2}\left|\partial_{y}\phi_{1}\right|\Longrightarrow
       \left| \int_{\RR}\varepsilon^{2}\partial_{y}^{3}\phi_{1}\dd y\right|\lesssim
       A^{-2}\left|\int_{\RR}\varepsilon^{2}\partial_{y}\phi_{1}\dd y\right|.
    \end{equation*}
    Then, using again the definition of $\phi_{1}$ in~\eqref{equ:defphi1}, 
    \begin{equation*}
        \begin{aligned}
           \left| \partial_{y}\phi_{1}\right|&\lesssim t^{-3}\ \ \mbox{on}\ \left(-\infty,t^{\frac{1}{2}}\right)\Longrightarrow U^{4}|\partial_{y}\phi_{1}|\lesssim t^{-3}\left(Q_{1}^{2}+Q_{2}^{2}+Q_{3}^{2}\right),\\
           \left| 1-\phi_{1}\right|&\lesssim t^{-3}\ \ \mbox{on}\ \left(-\infty,t^{\frac{1}{2}}\right)\Longrightarrow U^{4}|1-\phi_{1}|\lesssim t^{-3}\left(Q_{1}^{2}+Q_{2}^{2}+Q_{3}^{2}\right).
        \end{aligned}
    \end{equation*}
    It follows directly from Lemma~\ref{le:L2e} that 
    \begin{equation*}
    \begin{aligned}
       & \int_{\RR}\left(|\partial_{y}\varepsilon||1-\phi_{1}|+|\varepsilon||\partial_{y}\phi_{1}|\right)\left|(S+\varepsilon)^{5}-S^{5}\right|\dd y\\
       &\lesssim t^{-3}\mathcal{N}_{A}^{2}(\varepsilon)+\|\varepsilon\|_{L^{2}}^{2}\left(\|V\phi\|_{L^{\infty}}^{4}+\|W\|_{L^{\infty}}^{4}+\|\varepsilon\|_{L^{\infty}}^{4}\right)\\
       &+\mathcal{N}_{A}(\varepsilon)\|\varepsilon\|_{L^{2}}\left(\|V\phi\|_{L^{\infty}}^{4}+\|W\|_{L^{\infty}}^{4}+\|\varepsilon\|_{L^{\infty}}^{4}\right)\lesssim \frac{C_{0}^{2}}{t^{\frac{13}{2}}}+\frac{1}{t^{6}}.
        \end{aligned}
    \end{equation*}
    We see that~\eqref{est:J2} follows from above estimates and the sign condition $\partial_{y}\phi_{1}<0$.

    \smallskip
    \emph{Estimate on $\mathcal{J}_{3}$.} We claim that 
    \begin{equation}\label{est:J3}
    \begin{aligned}
    \mathcal{J}_{3}&=20\sum_{k=1}^{3}\dot{x}_{k}\int_{\RR}\left(\partial_{y}Q_{k}\right)
            \left(Q_{k}^{3}\varepsilon^{2}\right)\dd y\\
            &+O\left(\frac{C_{0}^{3}}{t^{5}}
            +\frac{A}{t^{\frac{1}{2}}}\int_{\RR}\varepsilon^{2}|\partial_{y}\phi_{1}|\dd y
            \right).
            \end{aligned}
    \end{equation}
Indeed, from the definition of $S$ in~\eqref{equ:defS} and Lemma~\ref{le:para1}, we have\footnote{See Definition~\ref{def:Sfunction} and the proof of~\eqref{est:ptU}--\eqref{equ:dtS} for more details.} 
    \begin{equation}\label{est:dtSinfty}
    \|\partial_{t}S\|_{L^{\infty}}\lesssim t^{-1}\ \ \mbox{and}\ \ 
        \partial_{t}S=-\sum_{k=1}^{3}\sigma_{k}\dot{x}_{k}\partial_{y}Q_{k}+O_{H^{1}}\left(\frac{C_{0}}
        {t^{\frac{7}{4}}}\right).
    \end{equation}
    Using an elementary computation, we decompose 
    \begin{equation}\label{equ:decomS}
    \begin{aligned}
        (S+\varepsilon)^{5}-S^{5}-5S^{4}\varepsilon
        &=10 U^{3}\varepsilon^{2}+O\left(U^{2}\left(|V\phi|+|W|\right)\varepsilon^{2}\right)\\
        &+O\left(\left(|V\phi|^{3}+|W|^{3}\right)\varepsilon^{2}+|S|\varepsilon^{4}+|\varepsilon|^{5}\right).
        \end{aligned}
    \end{equation}
    Note that, from~\eqref{est:dtSinfty} and Lemma~\ref{le:boundinte}, we find 
    \begin{equation*}
        -20\int_{\RR}\left(\partial_{t}S\right)U^{3}\varepsilon^{2}\dd y=\sum_{k=1}^{3}20\dot{x}_{k}\int_{\RR}\left(\partial_{y}Q_{k}\right)
            \left(Q_{k}^{3}\varepsilon^{2}\right)\dd y+O\left(\frac{C_{0}^{3}}{t^{\frac{21}{4}}}\right).
    \end{equation*}
    Note also that, using Lemma~\ref{le:L2e} and Cauchy-Schwarz inequality, 
\begin{equation*}
\begin{aligned}
&\int_{\RR}|\partial_{t}S|\left(U^{2}(|V\phi|+|W|)+|V\phi|^{3}+|W|^{3}+|S|\varepsilon^{2}+|\varepsilon|^{3}\right)\varepsilon^{2}\dd y\\
&\lesssim t^{-2}\mathcal{N}_{A}^{2}(\varepsilon)+\|\partial_{t}S\|_{L^{\infty}}
\left(\|V\phi\|_{L^{\infty}}^{3}+\|W\|_{L^{\infty}}^{3}+\|\varepsilon\|_{L^{\infty}}^{3}\right)\|\varepsilon\|_{L^{2}}^{2}\lesssim \frac{C_{0}^{2}}{t^{\frac{11}{2}}}+\frac{1}{t^{5}}.
\end{aligned}
\end{equation*}
Combining the above two estimates with~\eqref{equ:decomS}, we find 
\begin{equation}\label{est:J31}
\begin{aligned}
    &-2\int_{\RR}(\left(\partial_{t}S\right)\left(
            (S+\varepsilon)^{5}-S^{5}-5S^{4}\varepsilon
            \right)\dd y\\
          &=20\sum_{k=1}^{3}\dot{x}_{k}\int_{\RR}\left(\partial_{y}Q_{k}\right)
            \left(Q_{k}^{3}\varepsilon^{2}\right)\dd y+O\left(\frac{C_{0}^{3}}{t^{5}}\right).
            \end{aligned}
\end{equation}
On the other hand, from the definition of $\phi_{1}$ in~\eqref{equ:defphi1},
\begin{equation*}
    \left|\partial_{t}\phi_{1}\right|\lesssim \frac{A}{t^{\frac{1}{2}}}|\partial_{y}\phi_{1}|
    \Longrightarrow 
    \int_{\RR}\varepsilon^{2}\partial_{t}\phi_{1}\dd y=O\left(\frac{A}{t^{\frac{1}{2}}}\int_{\RR}\varepsilon^{2}|\partial_{y}\phi_{1}|\dd y\right). 
    \end{equation*}
    We see that~\eqref{est:J3} follows from~\eqref{est:J31} and the above estimate.

\smallskip
    Combining the estimates~\eqref{est:J1}, \eqref{est:J2} and~\eqref{est:J3} with the sign condition $\partial_{y}\phi_{1}<0$, we complete the proof of estimate~\eqref{est:dtE} by taking $T_{0}\gg 1$ large enough.

\smallskip
\textbf{Step 2.} Time control of $\mathcal{P}$. We claim that 
\begin{equation}\label{est:dtP}
\begin{aligned}
    \frac{\dd \mathcal{P}}{\dd t}&=  -20\sum_{k=1}^{3}\nu_{k}\int_{\RR}\left(\partial_{y}Q_{k}\right)
            \left(Q_{k}^{3}\varepsilon^{2}\right)\dd y\\
           &+O\left(\frac{C_{0}^{2}}{t^{\frac{9}{2}}\log t}+\frac{C_{0}^{3}}{t^{\frac{23}{5}}}+\frac{C_{0}A^{\frac{1}{2}}}{t^{\frac{11}{2}}}\right).
            \end{aligned}
\end{equation}
Indeed, from Lemma~\ref{le:eque} and integration by parts, we compute 
\begin{equation*}
    \frac{\dd \mathcal{P}}{\dd t}=\mathcal{J}_{4}+\mathcal{J}_{5}+\mathcal{J}_{6},
\end{equation*}
where 
\begin{equation*}
    \begin{aligned}
        \mathcal{J}_{4}=\dot{\nu}_{1}\int_{\RR}\varepsilon^{2}\phi_{3}\dd y+\nu_{1}\int_{\RR}\varepsilon^{2}\partial_{t}\phi_{3}\dd y+2\nu_{1}\int_{\RR}(\partial_{t}\varepsilon)\varepsilon\phi_{3}\dd y,\\
          \mathcal{J}_{5}=\dot{\nu}_{2}\int_{\RR}\varepsilon^{2}\phi_{4}\dd y+\nu_{2}\int_{\RR}\varepsilon^{2}\partial_{t}\phi_{4}\dd y+2\nu_{2}\int_{\RR}(\partial_{t}\varepsilon)\varepsilon\phi_{4}\dd y,\\
            \mathcal{J}_{6}=\dot{\nu}_{3}\int_{\RR}\varepsilon^{2}\phi_{5}\dd y+\nu_{3}\int_{\RR}\varepsilon^{2}\partial_{t}\phi_{5}\dd y+2\nu_{3}\int_{\RR}(\partial_{t}\varepsilon)\varepsilon\phi_{5}\dd y.
    \end{aligned}
\end{equation*}
\emph{Estimate on $\mathcal{J}_{4}$.} We claim that 
\begin{equation}\label{est:J4}
\begin{aligned}
    \mathcal{J}_{4}&=-20\nu_{1}\int_{\RR}\left(\partial_{y}Q_{1}\right)
            \left(Q_{1}^{3}\varepsilon^{2}\right)\dd y\\
            &+O\left(\frac{C_{0}^{2}}{t^{\frac{9}{2}}\log t}+\frac{C_{0}^{3}}{t^{\frac{23}{5}}}+\frac{C_{0}A^{\frac{1}{2}}}{t^{\frac{11}{2}}}\right).
            \end{aligned}
\end{equation}
Indeed, from~\eqref{est:phi345} and Lemma~\ref{le:para1}, we find 
\begin{equation}\label{est:J41}
    \Big|\dot{\nu}_{1}\int_{\RR}\varepsilon^{2}\phi_{3}\dd y\Big|
    \lesssim |\dot{\nu}_{1}|\int_{\RR}\varepsilon^{2}\phi_{1}\dd y\lesssim
    |\dot{\nu}_{1}|\mathcal{N}^{2}_{A}(\varepsilon)\lesssim \frac{C_{0}^{3}}{t^{\frac{21}{4}}}.
\end{equation}
Then, using the definition of $\phi_{3}$ in~\eqref{equ:defphik} and Lemma~\ref{le:para1}, we have 
\begin{equation*}
    |\partial_{t}\phi_{3}|\lesssim \left(\frac{|\dot{x}_{1}|}{\log t}+\frac{1}{t\log t}\left|\frac{y-x_{1}(t)}{\log t}\right|\right)\left|\chi'_{3}\left(\frac{y-x_{1}(t)}{\log t}\right)\right|\lesssim \frac{\phi_{1}}{t\log t},
\end{equation*}
which implies that 
\begin{equation}\label{est:J42}
   \Big| \nu_{1}\int_{\RR}\varepsilon^{2}\partial_{t}\phi_{3}\dd y\Big|
   \lesssim \frac{1}{t^{2}\log t}\int_{\RR}\varepsilon^{2}\phi_{1}\dd y
   \lesssim \frac{1}{t^{2}\log t}\mathcal{N}_{A}^{2}(\varepsilon)\lesssim \frac{C_{0}^{2}}{t^{\frac{11}{2}}\log t}.
\end{equation}
On the other hand, from integration by parts and Lemma~\ref{le:eque}, we compute 
\begin{equation*}
\begin{aligned}
   2\nu_{1} \int_{\RR}(\partial_{t}\varepsilon)\varepsilon\phi_{3}\dd y
    &=-\nu_{1}\int_{\RR}
    \left(3(\partial_{y}\varepsilon)^{2}+\varepsilon^{2}\right)\partial_{y}\phi_{3}\dd y
    +\nu_{1}\int_{\RR}\varepsilon^{2}\partial_{y}^{3}\phi_{3}\dd y\\
    &-2\nu_{1}\int_{\RR}\Phi(S)\left(\varepsilon\phi_{3}\right)\dd y
    +10\nu_{1}\int_{\RR}\left(U^{4}\varepsilon\right)
    \left(\partial_{y}(\varepsilon\phi_{3})\right)\dd y
    \\
    &+2\nu_{1}\int_{\RR}\left((S+\varepsilon)^{5}-S^{5}-5U^{4}\varepsilon\right)
    \left(\partial_{y}(\varepsilon\phi_{3})\right)\dd y.
    \end{aligned}
\end{equation*}
First, the definition of $\phi_{3}$ in~\eqref{equ:defphik} gives
\begin{equation}\label{est:pyphi3}
    \left|\partial_{y}\phi_{3}\right|\lesssim \frac{\phi_{1}}{\log t}\quad \mbox{and}\quad 
    \left|\partial_{y}^{3}\phi_{3}\right|\lesssim \frac{\phi_{1}}{\log^{3}t},
\end{equation}
which implies that 
\begin{equation}\label{est:J431}
   \Big| \nu_{1}\int_{\RR}
    \left(3(\partial_{y}\varepsilon)^{2}+\varepsilon^{2}\right)\partial_{y}\phi_{3}\dd y\Big|
    +\Big|\nu_{1}\int_{\RR}\varepsilon^{2}\partial_{y}^{3}\phi_{3}\dd y\Big|\lesssim \frac{C_{0}^{2}}{t^{\frac{9}{2}}\log t}.
\end{equation}
Second, from~\eqref{equ:defphik} and the exponential decay of $\partial_{y}Q$ and $\Lambda Q$, we check that 
\begin{equation*}
\begin{aligned}
   \left| (\partial_{y}Q_{1})(1-\phi_{3})\right|&\lesssim t^{-\frac{1}{10}}|\partial_{y}Q_{1}|^{\frac{1}{2}},\\
    \left| (\Lambda_{1}Q_{1})(1-\phi_{3})\right|&\lesssim t^{-\frac{1}{10}}|\Lambda_{1}Q_{1}|^{\frac{1}{2}}.
   \end{aligned}
\end{equation*}
Analysis similar to that above shows that
\begin{equation}\label{est:localQ2Q3}
\begin{aligned}
    |(\partial_{y}Q_{2})\phi_{3}|+|(\partial_{y}Q_{3})\phi_{3}|&\lesssim
    t^{-\frac{1}{10}}
    \left(|\partial_{y}Q_{2}|^{\frac{1}{2}}+|\partial_{y}Q_{3}|^{\frac{1}{2}}\right),
    \\
     |(\Lambda_{2}Q_{2})\phi_{3}|+|(\Lambda_{3}Q_{3})\phi_{3}|&\lesssim
     t^{-\frac{1}{10}}
    \left(|\Lambda_{2}Q_{2}|^{\frac{1}{2}}+|\Lambda_{3}Q_{3}|^{\frac{1}{2}}\right).
    \end{aligned}
\end{equation}
Using the above estimates, the orthogonality condition~\eqref{equ:ortho}
and Lemma~\ref{le:para1}, 
\begin{equation*}
\begin{aligned}
   & \sum_{k=1}^{3}\Big|\nu_{1}\int_{\RR}\left({\vec{\rm{Mod}}}_{k}\cdot\vec{M}_{k}Q\right)(\varepsilon\phi_{3})\dd y\Big|\\
   &\lesssim \frac{C_{0}}{t^{\frac{57}{20}}}\mathcal{N}_{A}(\varepsilon)+\frac{C_{0}}{t^{\frac{17}{4}}}\sum_{k=1}^{3}\Big(1+\big\|X_{k}\sqrt{\phi_{3}}\big\|_{L^{2}}\Big)\mathcal{N}_{A}(\varepsilon)\lesssim \frac{C_{0}^{2}}{t^\frac{23}{5}}.
    \end{aligned}
\end{equation*}
 Based on Lemma~\ref{le:Phiweight} and the Cauchy-Schwarz inequality, we get 
\begin{equation*}
\begin{aligned}
    &\sum_{k=1}^{3}\Big|\nu_{1}\int_{\RR}\Phi_{k}(S)\left(\varepsilon\phi_{3}\right)\dd y\Big|\\
    &\lesssim 
    \sum_{k=1}^{3}t^{-1}\mathcal{N}_{A}(\varepsilon)\left(\int_{\RR}|\Phi_{k}(S)|^{2}\phi_{1}\dd y\right)^{\frac{1}{2}}\lesssim \frac{C_{0}A^{\frac{1}{2}}}{t^{\frac{11}{2}}}.
    \end{aligned}
\end{equation*}
Gathering the above estimates together and then using Definition~\ref{def:Sfunction}, we obtain 
\begin{equation}\label{est:J432}
    \Big|\nu_{1}\int_{\RR}\Phi(S)\left(\varepsilon\phi_{3}\right)\dd y\Big|\lesssim 
    \frac{C_{0}}{t^{\frac{21}{4}}}+
   \frac{C_{0}^{2}}{t^\frac{23}{5}}+\frac{C_{0}A^{\frac{1}{2}}}{t^{\frac{11}{2}}}.
\end{equation}
Third, using again~\eqref{est:pyphi3},~\eqref{est:localQ2Q3} and integration by parts,
\begin{equation}\label{est:J433}
\begin{aligned}
    &10\nu_{1}\int_{\RR}\left(U^{4}\varepsilon\right)
    \left(\partial_{y}(\varepsilon\phi_{3})\right)\dd y\\
    &=-20\nu_{1}\int_{\RR}\left(\partial_{y}U\right)\left(U^{3}\varepsilon^{2}\right)\phi_{3}\dd y+O\left(\frac{C_{0}^{2}}{t^{\frac{9}{2}}\log t}\right).
    \end{aligned}
\end{equation}
Last, from an elementary computation and the definition of $S$ in~\eqref{equ:defS},
\begin{equation*}
\begin{aligned}
    \left|(S+\varepsilon)^{5}-S^{5}-5U^{4}\varepsilon\right|&\lesssim
    |U|^{3}\left(|V\phi|+|W|\right)|\varepsilon|+|\varepsilon|^{5}\\
    &+\left(|V\phi|^{4}+|W|^{4}\right)|\varepsilon|+ |S|^{3}\varepsilon^{2}.
    \end{aligned}
\end{equation*}
Based on~\eqref{est:phi345},~\eqref{est:pyphi3}, Lemma~\ref{le:L2e} and the above estimate, we find 
\begin{equation*}
\begin{aligned}
  & \Big| \nu_{1}\int_{\RR}\left((S+\varepsilon)^{5}-S^{5}-5U^{4}\varepsilon\right)
    \left(\partial_{y}(\varepsilon\phi_{3})\right)\dd y\Big|\\
    &\lesssim t^{-1}\left(\|V\phi\|_{L^{\infty}}+\|W\|_{L^{\infty}}+\|\varepsilon\|_{L^{\infty}}\right)\mathcal{N}_{A}^{2}(\varepsilon)\lesssim \frac{C_{0}^{2}}{t^{\frac{23}{4}}}.
    \end{aligned}
\end{equation*}
Combining the above estimate with~\eqref{est:J431} and \eqref{est:J432}--\eqref{est:J433}, we obtain 
\begin{equation*}
\begin{aligned}
     2\nu_{1} \int_{\RR}(\partial_{t}\varepsilon)\varepsilon\phi_{3}\dd y
&=-20\nu_{1}\int_{\RR}\left(\partial_{y}U\right)\left(U^{3}\varepsilon^{2}\right)\phi_{3}\dd y\\
&+O\left(\frac{C_{0}^{2}}{t^{\frac{9}{2}}\log t}+\frac{C_{0}^{2}}{t^{\frac{23}{5}}}+\frac{C_{0}A^{\frac{1}{2}}}{t^{\frac{11}{2}}}\right).
     \end{aligned}
\end{equation*}
It follows from Lemma~\ref{le:boundinte} and \eqref{est:defchi3}-\eqref{equ:defphik} that 
\begin{equation*}
\begin{aligned}
      2\nu_{1} \int_{\RR}(\partial_{t}\varepsilon)\varepsilon\phi_{3}\dd y
      &=-20\nu_{1}\int_{\RR}\left(\partial_{y}Q_{1}\right)\left(Q_{1}^{3}\varepsilon^{2}\right)\dd y\\
      &+O\left(\frac{C_{0}^{2}}{t^{\frac{9}{2}}\log t}+\frac{C_{0}^{2}}{t^{\frac{23}{5}}}+\frac{C_{0}A^{\frac{1}{2}}}{t^{\frac{11}{2}}}\right).
      \end{aligned}
\end{equation*}
We see that~\eqref{est:J4} follows from~\eqref{est:J41}--\eqref{est:J42} and the above estimate. 

\smallskip
\emph{Estimate on $\mathcal{J}_{5}-\mathcal{J}_{6}$.} We claim that 
\begin{equation}\label{est:J5J6}
\begin{aligned}
    \mathcal{J}_{5}+\mathcal{J}_{6}&=-20\sum_{k=2,3}\nu_{k}\int_{\RR}\left(\partial_{y}Q_{k}\right)
            \left(Q_{k}^{3}\varepsilon^{2}\right)\dd y\\
            &+O\left(\frac{C_{0}^{2}}{t^{\frac{9}{2}}\log t}+\frac{C_{0}^{3}}{t^{\frac{23}{5}}}+\frac{C_{0}A^{\frac{1}{2}}}{t^{\frac{11}{2}}}\right).
            \end{aligned}
\end{equation}
Indeed, using a similar argument as in the proof of~\eqref{est:J4}, we find 
\begin{equation*}
\begin{aligned}
\mathcal{J}_{5}+\mathcal{J}_{6}&=-20\sum_{k=2,3}\nu_{k}\int_{\RR}\left(\partial_{y}U\right)\left(U^{3}\varepsilon^{2}\right)\phi_{k+2}\dd y\\
&+O\left(\frac{C_{0}^{2}}{t^{\frac{9}{2}}\log t}+\frac{C_{0}^{3}}{t^{\frac{23}{5}}}+\frac{C_{0}A^{\frac{1}{2}}}{t^{\frac{11}{2}}}\right).
    \end{aligned}
\end{equation*}
Hence, we see that~\eqref{est:J5J6} follows from Lemma~\ref{le:boundinte} and~\eqref{est:defchi3}--\eqref{equ:defphik}.

\smallskip
Combining~\eqref{est:J4} with~\eqref{est:J5J6}, we complete the proof of the estimate~\eqref{est:dtP}.

\smallskip
\textbf{Step 3.} Conclusion. From~\eqref{est:dtE} and~\eqref{est:dtP}, we deduce that
\begin{equation*}
\begin{aligned}
    \frac{\dd \mathcal{K}}{\dd t}
    &\ge  20\sum_{k=1}^{3}\left(\dot{x}_{k}-\nu_{k}\right)\int_{\RR}\left(\partial_{y}Q_{k}\right)
            \left(Q_{k}^{3}\varepsilon^{2}\right)\dd y\\
           & -\frac{1}{4}\int_{\RR}\varepsilon^{2}\partial_{y}\phi_{1}\dd y
        +O\left(\frac{C_{0}A^{\frac{1}{2}}}{t^{\frac{9}{2}}}+\frac{C_{0}^{2}}{t^{\frac{9}{2}}\log t}+
         \frac{C_{0}^{3}A^{\frac{1}{2}}}{t^{\frac{23}{5}}}\right).
         \end{aligned}
\end{equation*}
Based on the above estimate and Lemma~\ref{le:para1}, we obtain 
\begin{equation*}
\begin{aligned}
    \frac{\dd \mathcal{K}}{\dd t}&\ge -\frac{1}{4}\int_{\RR}\varepsilon^{2}\partial_{y}\phi_{1}\dd y+O\left(\frac{C_{0}^{3}}{t^{\frac{21}{4}}}\right)\\
    & +O\left(\frac{C_{0}A^{\frac{1}{2}}}{t^{\frac{9}{2}}}+\frac{C_{0}^{2}}{t^{\frac{9}{2}}\log t}+
         \frac{C_{0}^{3}A^{\frac{1}{2}}}{t^{\frac{23}{5}}}\right),
    \end{aligned}
\end{equation*}
which completes the proof of the estimate~\eqref{est:dtK} via the sign condition $\partial_{y}\phi_{1}<0$.
\end{proof}

\subsection{End of the proof for Proposition~\ref{prop:uni}}\label{SS:EndPropuni}
In this subsection, we complete the proof of Proposition~\ref{prop:uni} via a bootstrap argument. We start with the following technical lemma related to some standard facts from Linear Algebra.

\begin{lemma}\label{le:linearalgebra}
    Let the matrix $M$ be defined by 
    \begin{equation*}
        M=\begin{pmatrix}
            0 & I_{3} & 0 \\
            0 & I_{3} &2 I_{3}\\
            N & 0 & 2I_{3}
        \end{pmatrix}\in \RR^{9\times 9}.
    \end{equation*}
    Here, we set 
    \begin{equation*}
        I_{3}=
        \begin{pmatrix}
            1 & 0 & 0\\
            0 & 1 & 0\\
            0 & 0 & 1
        \end{pmatrix}\in \RR^{3\times 3}
        \quad \mbox{and}\quad 
        N=
        \begin{pmatrix}
            -15 & 15 & 0\\
            -15 & 18 & -3\\
            0 & -9 &9
        \end{pmatrix}\in \RR^{3\times 3}.
    \end{equation*}
    Then, the matrix $M$ is diagonalizable and there exists $P_{1}\in {\mathrm{GL}}(9;\mathbb{C})$ such that 
    \begin{equation*}
        P^{-1}_{1}MP_{1}={\operatorname{diag}}\left(0,1,2,-1,2+\sqrt{2}i,2-\sqrt{2}i,\lambda_{1},\lambda_{2},\lambda_{3}\right).
    \end{equation*}
    Here, we denote by $(\lambda_{1},\lambda_{2},\lambda_{3})$ the roots of the following equation
    \begin{equation*}
        \lambda^{3}-3\lambda^{2}+2\lambda-30=0.
    \end{equation*}
\end{lemma}

\begin{proof}
    {\textbf{Step 1.}} Eigenvalues of $N$. For any $\lambda \in \mathbb{C}$, we compute 
    \begin{equation*}
        \det \left(\lambda I_{3}-N\right)= \det \begin{pmatrix}
            \lambda+15 & -15 & 0\\
            15 & \lambda-18 & 3\\
            0 & 9 &\lambda -9
        \end{pmatrix}=\lambda (\lambda-15)(\lambda+3),
    \end{equation*}
    which directly implies that the eigenvalues of $N$ are given by $(-3,0,15)$.

   \smallskip
   {\textbf{Step 2.}} Eigenvalues of $M$. For any $\lambda \in \mathbb{C}$, from Schur complement, we compute 
   \begin{equation*}
   \begin{aligned}
    \det \left(\lambda I_{9}-M\right)
    &=\det 
    \begin{pmatrix}
        \lambda I_{3} & -I_{3} & 0\\
        0 &(\lambda-1) I_{3} & -2I_{3}\\
        -N & 0 & (\lambda-2)I_{3}
    \end{pmatrix}\\
    &=\det \left(\lambda (\lambda-1)(\lambda-2)I_{3}-2N\right).
    \end{aligned}
   \end{equation*}
   Therefore, by the result in Step 1, we obtain 
   \begin{equation*}
          \det \left(\lambda I_{9}-M\right)=\prod_{\mu\in \left\{-3,0,15\right\}}\left(\lambda (\lambda-1)(\lambda-2)-2\mu\right),
   \end{equation*}
   which directly completes the proof of Lemma~\ref{le:linearalgebra}.
\end{proof}
\begin{remark}\label{re:linearalgebra}
    Using a standard numerical computation, we find 
    \begin{equation*}
        (\lambda_{1},\lambda_{2},\lambda_{3})
        \approx \left(-0.606+2.5i,-0.606-2.5i,4.213\right).
    \end{equation*}
    Therefore, we could rearrange the eigenvalues of $M$ as $(\xi_{1},\dots,\xi_{9})\in \mathbb{C}^{9}$ such that 
    \begin{equation}\label{est:xirank}
        \mathrm{Re}\xi_{1}<\mathrm{Re}\xi_{2}<\mathrm{Re}\xi_{3}<-\frac{1}{15}<\mathrm{Re}\xi_{4}<\cdots<\mathrm{Re}\xi_{9}.
    \end{equation}
    Moreover, we choose the matrix $P_{2}\in \mathrm{GL}(9;\mathbb{C})$ such that 
    \begin{equation*}
        P^{-1}_{2}MP_{2}={\rm{diag}}\left(\xi_{1},\xi_{2},\dots,\xi_{9}\right).
    \end{equation*}
\end{remark}

We are in a position to complete the proof of Proposition~\ref{prop:uni}.
\begin{proof}[Proof of Proposition~\ref{prop:uni}]
\textbf{Step 1.} Closing the estimate for $\varepsilon$. Integrating~\eqref{est:dtK} on $[t,T_{n}]$ for any $t\in [T^{*},T_{n}]$ and then using $\varepsilon(T_{n})=0$, we get 
\begin{equation*}
    \mathcal{K}(t)\lesssim 
            \frac{C_{0}A^{\frac{1}{2}}}{t^{\frac{7}{2}}}+\frac{C_{0}^{2}}{t^{\frac{7}{2}}\log t}+
         \frac{C_{0}^{3}A^{\frac{1}{2}}}{t^{\frac{18}{5}}}.
\end{equation*}
Therefore, from~\eqref{est:coerK}, for $T_{0}\gg 1$ large enough (depending on $C_{0}$ and $A$), 
\begin{equation}\label{est:improveNA}
    \mathcal{N}^{2}_{A}(\varepsilon)\lesssim  \frac{C_{0}A^{\frac{1}{2}}}{t^{\frac{7}{2}}}+\frac{C_{0}^{2}}{t^{\frac{7}{2}}\log t}+
         \frac{C_{0}^{3}A^{\frac{1}{2}}}{t^{\frac{18}{5}}}\lesssim
         \frac{C_{0}A^{\frac{1}{2}}}{t^{\frac{7}{2}}}+\frac{C_{0}^{2}}{t^{\frac{7}{2}}\log t}.
\end{equation}
This strictly improves the estimate on $\varepsilon$ in~\eqref{est:Boot} for $C_{0}\gg 1$ large enough.

\smallskip
\textbf{Step 2.} Linearization of the ODE system for $\Gamma$. 
For any $k\in [\![1,3]\!]$, we denote 
\begin{equation*}
\left\{
\begin{aligned}
\widehat{x}_{k}&=x_{k}+(18-3k)\log t-\beta_{k},\\
\widehat{\nu}_{k}&=\nu_{k}+\mathcal{F}_{k}+\frac{18-3k}{t} \ \ \mbox{and}\ \ 
\widehat{b}_{k}=b_{k}-\frac{18-3k}{2t^{2}}.
\end{aligned}
\right.
\end{equation*}
In addition, we also denote 
\begin{equation}\label{est:Pipoint}
\vec{\Omega}=\big(\widehat{x}_{1},\widehat{x}_{2},\widehat{x}_{3},
t\widehat{\nu}_{1},t\widehat{\nu}_{2},t\widehat{\nu}_{3},
t^{2}\widehat{b}_{1},t^{2}\widehat{b}_{2},t^{2}\widehat{b}_{3}
\big)^{T}\ \ \mbox{and}\ \ \vec{\Pi}=P_{2}^{-1}\vec{\Omega}.
\end{equation}
On the one hand, from the bootstrap setting~\eqref{est:Boot} and the definition of $\vec{\Pi}$, 
\begin{equation}\label{est:pointPi}
    |\vec{\Pi}(t)|\lesssim |\vec{\Omega}(t)|\lesssim \sum_{k=1}^{3}
    \left(|\widehat{x}_{k}|+t|\widehat{v}_{k}|+t^{2}|\widehat{b}_{k}|\right)
    \lesssim t^{-\frac{1}{16}}.
\end{equation}
Here, we use the fact that 
\begin{equation*}
\sum_{k=1}^{3}|\mathcal{F}_{k}|\lesssim \sum_{k=1}^{3}\int_{\RR}|\varepsilon(t)|\rho_{k}\phi_{2}\dd y\lesssim t^{\frac{1}{2}}\mathcal{N}_{A}(\varepsilon)\lesssim \frac{C_{0}}{t^{\frac{5}{4}}}.
\end{equation*}
On the other hand, using again the bootstrap setting~\eqref{est:Boot} and~\eqref{equ:defbeta}, we find 
\begin{equation*}
\begin{aligned}
e^{-r_{1}}&=\frac{15}{\alpha t^{3}}e^{-(\widehat{x}_{2}-\widehat{x}_{1})}=\frac{15}{\alpha t^{3}}
\left(1+\widehat{x}_{1}-\widehat{x}_{2}+O\left(\frac{1}{t^{\frac{1}{8}}}\right)\right)
,\\
    e^{-r_{2}}&=\frac{3}{\alpha t^{3}}e^{-(\widehat{x}_{3}-\widehat{x}_{2})}
    =\frac{3}{\alpha t^{3}}\left(1+\widehat{x}_{2}-\widehat{x}_{3}+O\left(\frac{1}{t^{\frac{1}{8}}}\right)\right).
    \end{aligned}
\end{equation*}
It follows from (ii) of Lemma~\ref{le:para1} and Lemma~\ref{le:linearalgebra} that 
\begin{equation*}
    \frac{\dd}{\dd t}
    \begin{pmatrix}
        \widehat{b}_{1}\\
        \widehat{b}_{2}\\
        \widehat{b}_{3}
    \end{pmatrix}=\frac{N}{t^{3}}\begin{pmatrix}
        \widehat{x}_{1}\\
        \widehat{x}_{2}\\
        \widehat{x}_{3}
    \end{pmatrix}
    +O\left(\frac{1}{t^{\frac{25}{8}}}+\frac{C_{0}}{t^{\frac{13}{4}}}+\frac{C_{0}^{2}}{t^{\frac{7}{2}}}\right).
\end{equation*}
Based on the above identity, the bootstrap setting~\eqref{est:Boot}, Lemma~\ref{le:para1} and Lemma~\ref{le:refinednu}, we deduce that 
\begin{equation*}
    \frac{\dd \vec{\Omega}}{\dd t}=\frac{M}{t}\vec{\Omega}
    +O\left(\frac{1}{t^{\frac{9}{8}}}+\frac{C_{0}}{t^{\frac{5}{4}}}+\frac{C_{0}^{2}}{t^{\frac{3}{2}}}\right).
\end{equation*}
Here, we use again the fact that 
\begin{equation*}
\sum_{k=1}^{3}|\mathcal{F}_{k}|\lesssim \sum_{k=1}^{3}\int_{\RR}|\varepsilon(t)|\rho_{k}\phi_{2}\dd y\lesssim t^{\frac{1}{2}}\mathcal{N}_{A}(\varepsilon)\lesssim \frac{C_{0}}{t^{\frac{5}{4}}}.
\end{equation*}
Combining the above identity with~\eqref{est:Pipoint} and Remark~\ref{re:linearalgebra}, we obtain 
\begin{equation}\label{equ:odeeigen}
    \frac{\dd \vec{\Pi}}{\dd t}=\frac{1}{t}
    \begin{pmatrix}
        \xi_{1} & 0 & \cdots & 0\\
         0& \xi_{2} & \cdots & 0\\
         \vdots & \vdots &\vdots & \vdots\\
      0  & 0  & \cdots & \xi_{9}
    \end{pmatrix}\vec{\Pi}
    +O\left(\frac{1}{t^{\frac{9}{8}}}+\frac{C_{0}}{t^{\frac{5}{4}}}+\frac{C_{0}^{2}}{t^{\frac{3}{2}}}\right).
\end{equation}

\textbf{Step 3.} Closing the estimate for $\Gamma$. We denote 
\begin{equation*}
\vec{\Pi}=\left(\Pi_{1},\Pi_{2},\Pi_{3},\Pi_{4},\Pi_{5},\Pi_{6},\Pi_{7},\Pi_{8},\Pi_{9}\right)^{T}\in \mathbb{C}^{9}.
\end{equation*}
We choose the final data $\Gamma^{in}$ such that $\vec{\Pi}(T_{n})\in \mathbb{C}^{9}$ takes the following form 
\begin{equation}\label{equ:finalPi}
    \vec{\Pi}(T_{n})=\left(T_{n}^{-\frac{1}{15}}\vec{v},\vec{0}\right)^{T},\ \ \mbox{with}\ \vec{v}=(v_{1},v_{2},v_{3})\in \bar{B}_{\RR^{3}}(1).
\end{equation}

For any $m\in [\![4,9]\!]$, from~\eqref{equ:odeeigen}, we find 
\begin{equation*}
    \frac{\dd \Pi_{m}}{\dd t}=\frac{\xi_{m}}{t}\Pi_{m}+ O\left(\frac{1}{t^{\frac{9}{8}}}+\frac{C_{0}}{t^{\frac{5}{4}}}+\frac{C_{0}^{2}}{t^{\frac{3}{2}}}\right),
\end{equation*}
which implies that 
\begin{equation*}
    \frac{\dd }{\dd t}\left(t^{-\xi_{m}}\Pi_{m}\right)=
     O\left(\frac{1}{t^{\frac{9}{8}+\mathrm{Re}\xi_{m}}}+\frac{C_{0}}{t^{\frac{5}{4}+\mathrm{Re}\xi_{m}}}+\frac{C_{0}^{2}}{t^{\frac{3}{2}+\mathrm{Re}\xi_{m}}}\right).
\end{equation*}
Integrating the above identity over $[t,T_{n}]$ for any $t\in [T^{*},T_{n}]$ and then using~\eqref{est:xirank} and~\eqref{equ:finalPi}, we deduce that 
\begin{equation}\label{est:Pi49}
\begin{aligned}
    \left|\Pi_{m}(t)\right|
    &\lesssim \int_{t}^{T_{n}}\left(\frac{|t^{\xi_{m}}|}{s^{\frac{9}{8}+\mathrm{Re}\xi_{m}}}+\frac{C_{0}|t^{\xi_{m}}|}{s^{\frac{5}{4}+\mathrm{Re}\xi_{m}}}\right)\dd s\\
    &+\int_{t}^{T_{n}}\frac{C_{0}^{2}|t^{\xi_{m}}|}{s^{\frac{3}{2}+\mathrm{Re}\xi_{m}}}\dd s
    \lesssim \frac{1}{t^{\frac{1}{8}}}+\frac{C_{0}}{t^{\frac{1}{4}}}+\frac{C_{0}^{2}}{t^{\frac{1}{2}}}.
\end{aligned}
\end{equation}
On the other hand, for the control of the unstable directions, we introduce 
\begin{equation*}
    \Upsilon(t)=t^{\frac{2}{15}}\sum_{m=1}^{3}|\Pi_{m}(t)|^{2}\quad \mbox{and}\quad 
    \varrho=\min_{m\in [\![1,3]\!]}\left(-\mathrm{Re}\xi_{m}-\frac{1}{15}\right)>0.
\end{equation*}
Moreover, we introduce the following bootstrap setting:
\begin{equation}\label{def:T2}
    T^{**}=T^{**}(\vec{v})=\inf \left\{t\in [T_{0},T_{n}]:\Upsilon(s)\le 1\ \mbox{on}\ [t,T_{n}]\right\}.
\end{equation}
It follows from \eqref{est:improveNA}-\eqref{est:Pipoint} and~\eqref{est:Pi49} that
\begin{equation}\label{est:T1T2}
    T_{0}\le T^{*}(\Gamma^{in})\le T^{**}(\vec{v})\le T_{n}<\infty,\ \ \mbox{for}\ T_{0}\gg 1.
\end{equation}
Then, using~\eqref{est:pointPi} and~\eqref{equ:odeeigen}, we have
\begin{equation*}
\begin{aligned}
    \frac{\dd \Upsilon}{\dd t}
    &=t^{\frac{2}{15}}\sum_{m=1}^{3}\left(\overline{\Pi}_{m}\frac{\dd \Pi_{m}}{\dd t}+{\Pi}_{m}\frac{\dd \overline{\Pi}_{m}}{\dd t}\right)
    +\frac{2}{15}t^{-\frac{13}{15}}\sum_{m=1}^{3}|\Pi_{m}|^{2}
    \\
    &=2t^{-\frac{13}{15}}\sum_{m=1}^{3}\left({\mathrm{Re}}\xi_{m}+\frac{1}{15}\right)|\Pi_{m}|^{2}+O\left(\frac{1}{t^{\frac{253}{240}}}+\frac{C_{0}^{2}}{t^{\frac{283}{240}}}\right),
    \end{aligned}
\end{equation*}
which implies that 
\begin{equation}\label{est:dtUpsilon}
    \frac{\dd \Upsilon}{\dd t}\le -\frac{2\varrho}{t}\Upsilon+O\left(\frac{1}{t^{\frac{253}{240}}}+\frac{C_{0}^{2}}{t^{\frac{283}{240}}}\right).
\end{equation}
The transversality relation~\eqref{est:dtUpsilon} is enough to justify the existence of at least one choose of $\vec{v}\in B_{\RR^{3}}(1)$ such that  $T^{*}(\Gamma^{in})=T^{**}(\vec{v})=T_{0}$.

\smallskip
For the sake of contradiction, we assume that for all $\vec{v}\in B_{\RR^{3}}(1)$, it holds $T_{0}<T^{**}(\vec{v})$. Then, a contradiction follows from the following discussion (see for instance more details in~\cite[Page 1153]{MMARMA} and~\cite[Page 730]{MartelRaphael}).

\smallskip
		\emph{Continuity of $T^{**}(\vec{v})$.}
		The above transversality relation~\eqref{est:dtUpsilon} directly implies that the following map 
		\begin{equation*}
			\vec{v}\in\bar{{B}}_{\RR^{3}}(1)\mapsto 
			T^{**}(\vec{v}),
		\end{equation*}
		is continuous and 
		\begin{equation*}
			T^{**}(\vec{v})=T_{n},\quad \mbox{for any}\ \vec{v}\in {S}_{\RR^{3}}(1).
		\end{equation*}
		
		\smallskip
		\emph{Construction of a retraction.} We define 
		\begin{equation*}
			\begin{aligned}
				\mathcal{A}:\vec{v}\in \bar{{B}}_{\RR^{3}}(1)\mapsto
                \left(\Pi_{1}(T^{**}(\vec{v})),\Pi_{2}(T^{**}(\vec{v})),\Pi_{3}(T^{**}(\vec{v}))\right)\in {S}_{\RR^{3}}(1).
			\end{aligned}
		\end{equation*}
        From what precedes, the map $\mathcal{A}:\bar{B}_{\RR^{3}}(1)\mapsto S_{\RR^{3}}(1)$ is continuous. In addition, from~\eqref{def:T2} and~\eqref{est:dtUpsilon}, the map $\mathcal{A}$ restricted to ${S}_{\RR^{3}}(1)$ is the identity. The existence of such a map is contradictory to the No Retraction Theorem for continuous maps from the ball to the sphere. Hence, the existence of $\vec{v}\in B_{\RR^{3}}(1)$ has proved and the uniform estimates for $(\varepsilon,\Gamma)$ are direct consequences of~\eqref{est:Boot},~\eqref{def:T2}-\eqref{est:T1T2} and Lemma~\ref{le:L2e}. The proof of Proposition~\ref{prop:uni} is complete.
\end{proof}

\section{Compactness argument}
In this section, we finish the construction of three-solitons in Theorem~\ref{thm:main} by passing to the limit on a sequence of solutions given by Proposition~\ref{prop:uni}.

\begin{proof}
    [End of the proof of Theorem~\ref{thm:main}]
    Applying Proposition~\ref{prop:uni} for any large enough $n\in \mathbb{N}^{+}$, we obtain a solution $u_{n}(t)$ of~\eqref{equ:gKdV} defined on the time interval $[T_{0},T_{n}]$ such that its decomposition $(\varepsilon_{n},\Gamma_{n})$ satisfies the uniform estimates~\eqref{est:unisolu} on $[T_{0},T_{n}]$. It follows directly that the sequence $\{u_{n}(T_{0})\}_{n\in \mathbb{N}^{+}}$ is bounded in $H^{1}$. Therefore, up to the extraction of a subsequence, $\{u_{n}(T_{0})\}_{n\in \mathbb{N}^{+}}$ converges weakly in $H^{1}$ to some $u_{0}\in H^{1}$ as $n\uparrow \infty$. Let $u(t)$ be the solution of~\eqref{equ:gKdV} corresponding to initial data $u_{|t=T_{0}}=u_{0}$.
    Based on the local Cauchy theory and weak $H^{1}$ continuity of the flow\footnote{See~\cite[Lemma 2.10]{Comkdv} and~\cite[Appendix D]{MMJMPA} for the statement and proof of this property.}, we see that $u(t)\in C([T_{0},\infty);H^{1})$ and $w(t,y)=u(t,y+t)$ satisfies~\eqref{est:modu} for all $t\in [T_{0},\infty)$. In addition, the decomposition $(\varepsilon(t),\Gamma(t))$ of $u(t)$ satisfies
    \begin{equation}\label{est:weaku}
        \varepsilon_{n}(t)\rightharpoonup \varepsilon(t)\ \ \mbox{weakly in}\ H^{1}\ \ \mbox{and}\ \ 
        \Gamma_{n}(t)\to \Gamma(t), \ \ \mbox{as}\ n\uparrow \infty.
    \end{equation}
    In particular, for any $t\in [T_{0},\infty)$, the solution $u(t)$ decompose as 
    \begin{equation*}
        u(t,x)=S(t,x-t;\Gamma)+\varepsilon(t,x-t)\in H^{1}.
    \end{equation*}
    Using~\eqref{est:unisolu} and~\eqref{est:weaku}, we find 
    \begin{equation}\label{est:unisoluu}
        \begin{aligned}
            |x_{1}(t)+15\log t-\beta_{1}|\le \frac{1}{t^{\frac{1}{16}}},\ \ \big|\nu_{1}(t)+\frac{15}{t}\big|
            \le \frac{1}{t^{\frac{17}{16}}},
            \\
            |x_{2}(t)+12\log t-\beta_{2}|\le \frac{1}{t^{\frac{1}{16}}}, \ \ 
            \big|\nu_{2}(t)+\frac{12}{t}\big|\le \frac{1}{t^{\frac{17}{16}}},
            \\
            |x_{3}(t)+9\log t-\beta_{3}|\le \frac{1}{t^{\frac{1}{16}}}, \ \ 
            \big|\nu_{3}(t)+\frac{9}{t}\big|\le \frac{1}{t^{\frac{17}{16}}},\\
            \big|b_{1}-\frac{15}{2t^{2}}\big|+
             \big|b_{2}-\frac{6}{t^{2}}\big|+
              \big|b_{3}-\frac{9}{2t^{2}}\big|\le \frac{1}{t^{\frac{33}{16}}}\ \ \mbox{and}\ \ 
             \|\varepsilon(t)\|_{H^{1}}\le \frac{1}{t}.
        \end{aligned}
    \end{equation}
    It follows from the definition of $S$ in~\eqref{equ:defS} that 
    \begin{equation*}
        \|S-U\|_{H^{1}}\lesssim \|V\phi\|_{H^{1}}+\|W\|_{H^{1}}\lesssim \frac{1}{t^{\frac{3}{2}}}+\frac{1}{t^{2}}\lesssim \frac{1}{t^{\frac{3}{2}}}.
    \end{equation*}
    Therefore, using~\eqref{equ:Taylornu} and~\eqref{est:unisoluu}, we obtain the following estimate of $u$:
    \begin{equation*}
        u(t,x)=\sum_{k=1}^{3}\sigma_{k}Q(x-t-x_{k}(t))
        +\varepsilon(t,x-t)
        +O_{H^{1}}\left(t^{-1}\right).
    \end{equation*}
    Combining above estimate with~\eqref{est:unisoluu}, we complete the proof of Theorem~\ref{thm:main}.
\end{proof}

\appendix
\section{Discussion on the ODE system}\label{Appen:ODE}
In this appendix, we analyze the structure of the ODE system~\eqref{equ:ODEsystem} and show that $K=3$ and $\vec{\sigma}=(1,-1,-1)$ is the only choice such that the system admits an explicit solution of the form~\eqref{equ:solutionapp}.
Following the standard approach in the study of multi-soliton dynamics for nonlinear dispersive or wave equations, we decompose
\begin{equation*}
    \partial_{t}S+\partial_{y}\left(\partial_{y}^{2}S-S+S^{5}\right)={\rm{Mod}}+{\mathrm{Error}}+{\rm{L.O.T}}.
\end{equation*}
Here, we denote by Mod the modulation part related to the approximate solution $S$ and the ODE system~\eqref{equ:ODEsystem} and by Error the leading-order of the error terms. Combining~\eqref{equ:defSintro} and~\eqref{equ:ODEsystem} with the estimate of nonlinear interaction between solitons (see Lemma~\ref{le:asymptotic}), we decompose\footnote{See for example Proposition~\ref{prop:appPhi} for the structure of Error in the three-solitons case.}
\begin{equation*}
    {\rm{Error}}=\sum_{k=1}^{K-1}e^{-r_{k}}\mathcal{Y}_{k},\ \ \mbox{on the multi-solitons region}.
\end{equation*}
Here, for any $k\in [\![1,K-1]\!]$, we set
\begin{equation}\label{equ:defYk}
\begin{aligned}
    \mathcal{Y}_{k}&=5c_{Q}\sigma_{k}(e^{-y}Q^{4})'(y-x_{k+1})+\sigma_{k}\alpha_{k,k}X_{k}\\
    &+\sigma_{k+1}\alpha_{k+1,k}X_{k+1}
    -10m_{0}\sigma_{k+1}\alpha_{k+1,k}\sum_{\ell=k+2}^{K}Y_{\ell}\\
    &+5c_{Q}\sigma_{k+1}(e^{y}Q^{4})'(y-x_{k})-10m_{0}\sigma_{k}\alpha_{k,k}\sum_{\ell=k+1}^{K}Y_{\ell}.
    \end{aligned}
\end{equation}
Our main goal is to choose signs $\vec{\sigma}$ and coefficients $\vec{\alpha}$ such that {$\rm{Error}$} does not affect the evolution of the ODE system~\eqref{equ:ODEsystem} especially for the parameter $b_{k}$ for $k\in [\![1,K]\!]$. Therefore, we need to find a choice such that the error $\mathcal{Y}_{k}$ is almost orthogonal to $Q_{\ell}$ for any $(k,\ell)\in [\![1,K]\!]\times [\![1,K]\!]$.
We split the discussion into the following two cases according to the number of solitons. 

\smallskip
{\textbf{Case 1}}. Let $K=2$. Using~\eqref{equ:defm0}, Lemma~\ref{le:Y} and Corollary~\ref{coro:X}, we find 
\begin{equation*}
   \left( \mathcal{Y}_{1},Q_{1}\right)=\left(\alpha \sigma_{2}-\sigma_{1}\alpha_{1,1}\right)m_{0}^{2}+{\rm{L.O.T}},
\end{equation*}
which implies that 
\begin{equation*}
    (\mathcal{Y}_{1},Q_{1})={\rm{L.O.T}}\Longrightarrow \alpha_{1,1}=\alpha \frac{\sigma_{2}}{\sigma_{1}}=\alpha \sigma_{2}.
\end{equation*}
Based on a similar argument, we also find 
\begin{equation*}
    \left(\mathcal{Y}_{1},Q_{2}\right)=-\left((\alpha+2\alpha_{1,1}) \sigma_{1}+\sigma_{2}\alpha_{2,1}\right)m_{0}^{2}+{\rm{L.O.T}},
\end{equation*}
which implies that 
\begin{equation*}
    \left(\mathcal{Y}_{1},Q_{2}\right)={\rm{L.O.T}}\Longrightarrow
    \alpha_{2,1}=-\frac{\sigma_{1}}{\sigma_{2}}\left(\alpha+2\alpha_{1,1}\right).
\end{equation*}
Combining the above two identities for $(\alpha_{1,1},\alpha_{2,1})$, we obtain 
\begin{equation*}
\left\{\begin{aligned}
    \left(\alpha_{1,1},\alpha_{2,1}\right)&=(\alpha,-3\alpha),\quad \quad \mbox{for}\ \left(\sigma_{1},\sigma_{2}\right)=(1,1),\\
    \left(\alpha_{1,1},\alpha_{2,1}\right)&=(-\alpha,-\alpha),\quad  \ \ \mbox{for}\ \left(\sigma_{1},\sigma_{2}\right)=(1,-1).
    \end{aligned}
    \right.
\end{equation*}
Note that, for the above two cases, the function in~\eqref{equ:solutionapp} is not a solution for the ODE system~\eqref{equ:ODEsystem}. In particular, for the case of $(\sigma_{1},\sigma_{2})=(1,-1)$, the system~\eqref{equ:ODEsystem} admits an explicit solution related to the two-bubble blow up solution of~\eqref{equ:gKdV} (see~\cite{LYPRINT} for more discussion) but not in the multi-solitons regime.

\smallskip
{\textbf{Case 2}}. Let $K\in \mathbb{N}^{+}$ with $K\ge 3$. Using again~Lemma~\ref{le:Y}, Corollary~\ref{coro:X} and the definition of $\mathcal{Y}_{k}$ in~\eqref{equ:defYk}, for any $k\in [\![1,K-2]\!]$, we find
\begin{equation*}
(\mathcal{Y}_{k},Q_{k+2})=-2\left(\sigma_{k}\alpha_{k,k}+\sigma_{k+1}\alpha_{k+1,k}\right)m_{0}^{2}+{\rm{L.O.T}},
\end{equation*}
which implies that 
\begin{equation}\label{equ:odec1}
    (\mathcal{Y}_{k},Q_{k+2})={\rm{L.O.T}}\Longrightarrow \sigma_{k}\alpha_{k,k}+\sigma_{k+1}\alpha_{k+1,k}=0.
\end{equation}
Based on an argument similar to that in Case 1, for any $k\in [\![1,K-2]\!]$, we also find 
\begin{equation*}
\left\{\begin{aligned}
    \left(\mathcal{Y}_{k},Q_{k}\right)&=\left(\alpha \sigma_{k+1}-\sigma_{k}\alpha_{k,k}\right)m_{0}^{2}+{\rm{L.O.T}},\\
    \left(\mathcal{Y}_{k},Q_{k+1}\right)&=-\left(\left(\alpha +2\alpha_{k,k}\right)\sigma_{k}+\sigma_{k+1}\alpha_{k+1,k}\right)m_{0}^{2}+{\rm{L.O.T}}.
    \end{aligned}\right.
\end{equation*}
It follows directly from~\eqref{equ:odec1} that 
\begin{equation*}
\left(\mathcal{Y}_{k},Q_{k}\right)
   = \left(\mathcal{Y}_{k},Q_{k+1}\right)={\rm{L.O.T}}
    \Longrightarrow
   \left( \alpha_{k,k},\sigma_{k+1}\right)=\left(-\alpha,-\sigma_{k}\right).
\end{equation*}
Combining the above identity with $\sigma_{1}=1$ and~\eqref{equ:odec1}, we obtain 
\begin{equation}\label{equ:ODESYS1}
    \left(\alpha_{k,k},\alpha_{k+1,k},\sigma_{k+1}\right)=
    \left(-\alpha,-\alpha,(-1)^{k}\right),\ \ \mbox{for any}\ k\in [\![1,K-2]\!].
\end{equation}
On the other hand, for the case of $k=K-1$, we compute
\begin{equation*}
\left\{\begin{aligned}
  &\left(\mathcal{Y}_{K-1},Q_{K-1}\right)=\left(\alpha\sigma_{K}-\sigma_{K-1}\alpha_{K-1,K-1}\right)m_{0}^{2}+{\rm{L.O.T}},\\
   &\left(\mathcal{Y}_{K-1},Q_{K}\right)=-\left((\alpha+2\alpha_{K-1,K-1})\sigma_{K-1}+\sigma_{K}\alpha_{K,K-1}\right)m_{0}^{2}+{\rm{L.O.T}},
    \end{aligned}\right.
\end{equation*}
which implies that 
\begin{equation*}
   \left(\mathcal{Y}_{K-1},Q_{K-1}\right)= \left(\mathcal{Y}_{K-1},Q_{K}\right)={\rm{L.O.T}}
   \Longrightarrow
    \left\{
    \begin{aligned}
    &\alpha_{K-1,K-1}=
    \alpha \frac{\sigma_{K}}{\sigma_{K-1}},\\
    &\alpha_{K,K-1}=-\frac{\sigma_{K-1}}{\sigma_{K}}(\alpha+2\alpha_{K-1.K-1}).
    \end{aligned}
    \right.
\end{equation*}
It follows directly from~\eqref{equ:ODESYS1} that 
    \begin{equation}\label{equ:ODESYS2}
\left\{\begin{aligned}
    \left(\alpha_{K-1,K-1},\alpha_{K,K-1}\right)&=(\alpha,-3\alpha),\quad\ \  \  \mbox{for}\ \sigma_{K-1}-\sigma_{K}=0,\\
    \left(\alpha_{K-1,K-1},\alpha_{K,K-1}\right)&=(-\alpha,-\alpha),\quad  \ \ \mbox{for}\ \sigma_{K-1}+\sigma_{K}=0.
    \end{aligned}
    \right.
\end{equation}

{\textbf{Subcase 2.1.}} Assume that $\sigma_{K-1}-\sigma_{K}=0$. Using~\eqref{equ:ODESYS1} and~\eqref{equ:ODESYS2}, we find
\begin{equation*}
\left\{\begin{aligned}
    &\frac{\dd b_{k}}{\dd t}=-\alpha e^{-r_{k-1}}-\alpha e^{-r_{k}}, \  \mbox{for}\ k\in [\![1,K-2]\!],\\
   & \frac{\dd b_{K-1}}{\dd t}=-\alpha e^{-r_{K-2}}+\alpha e^{-r_{K-1}} \ \ \mbox{and}\ \ \frac{\dd b_{K}}{\dd t}=-3\alpha e^{-r_{K-1}}.
    \end{aligned}
    \right.
\end{equation*}
For $K\ge 4$, we directly have 
\begin{equation*}
    \frac{\dd b_{2}}{\dd t}-\frac{\dd b_{1}}{\dd t}=-\alpha e^{-r_{2}}<0,
\end{equation*}
which implies that the ODE system~\eqref{equ:ODEsystem} does not have a straightforward solution of the form~\eqref{equ:solutionapp}. Surprisingly, for $K=3$, the system admits an explicit solution 
\begin{equation*}
    (x_{k},\nu_{k},b_{k})=\left((-18+3k)\log t+\beta_{k},\frac{-18+3k}{t},\frac{18-3k}{2t^{2}}\right),\quad \mbox{for}\ k\in [\![1,3]\!],
\end{equation*}
which matches the asymptotic behavior of solution in the statement of Theorem~\ref{thm:main}.

\smallskip
\textbf{Subcase 2.2.} Assume that $\sigma_{K-1}+\sigma_{K}=0$. Using again~\eqref{equ:ODESYS1} and~\eqref{equ:ODESYS2}, we find
\begin{equation*}
    \frac{\dd b_{k}}{\dd t}=-\alpha e^{-r_{k-1}}-\alpha e^{-r_{k}}, \ \   \mbox{for}\ k\in [\![1,K]\!],
\end{equation*}
which directly implies that 
\begin{equation*}
    \frac{\dd b_{2}}{\dd t}-\frac{\dd b_{1}}{\dd t}=-\alpha e^{-r_{2}}<0.
\end{equation*}
Based on the above inequality, we know that the ODE system~\eqref{equ:ODEsystem} does not have a straightforward solution of the form~\eqref{equ:solutionapp}.  Therefore, we do not expect the existence of multi-soliton solutions in the strong interaction regime for this case.

\end{document}